\documentclass{amsart}

\usepackage[margin=1in]{geometry}
\usepackage{amsfonts,amsmath,amssymb,amsthm}
\usepackage{esint,xfrac,mathtools}
\usepackage[normalem]{ulem}

\usepackage[dvipsnames]{xcolor}
\usepackage[colorlinks=true, pdfstartview=FitV, linkcolor=blue, citecolor=blue, urlcolor=blue, pdfencoding=auto, breaklinks=true]{hyperref}

\usepackage[T1]{fontenc}
\usepackage{lmodern}
\usepackage{enumerate}

\allowdisplaybreaks

\numberwithin{equation}{section}

\newtheorem{theorem}{Theorem}[section]

\newtheorem{proposition}[theorem]{Proposition}
\newtheorem{lemma}[theorem]{Lemma}
\newtheorem{definition}[theorem]{Definition}
\theoremstyle{definition}
\newtheorem{remark}[theorem]{Remark}
\newtheorem*{theorem-non}{Conjecture}

\newcommand{\norm}[1]{\left\|#1\right\|}
\newcommand{\abs}[1]{\left|#1\right|}
\newcommand{\brak}[1]{\left\langle#1\right\rangle}

\newcommand{\eps}{\varepsilon}
\renewcommand{\epsilon}{\varepsilon}

\newcommand{\N}{\mathbb{N}}
\newcommand{\bp}{\mathbb{P}}
\newcommand{\R}{\mathbb{R}}
\newcommand{\T}{\mathbb{T}}
\newcommand{\Z}{\mathbb{Z}}

\newcommand{\Id}{\mathrm{Id}}

\newcommand{\dd}{\, {\rm d}}
\newcommand{\pa}{\partial}

\newcommand{\ilr}{(-\Delta)^{-1}_{\R^3}}

\def\Xint#1{\mathchoice
{\XXint\displaystyle\textstyle{#1}}%
{\XXint\textstyle\scriptstyle{#1}}%
{\XXint\scriptstyle\scriptscriptstyle{#1}}%
{\XXint\scriptscriptstyle\scriptscriptstyle{#1}}%
\!\int}
\def\XXint#1#2#3{{\setbox0=\hbox{$#1{#2#3}{\int}$ }
\vcenter{\hbox{$#2#3$ }}\kern-.6\wd0}}

\def\dashint{\Xint-}

\newcommand{\les}{\lesssim}

\mathtoolsset{showonlyrefs}

\title[A Sharp Rigidity/Flexibility Threshold]{A sharp rigidity/flexibility threshold for\\ the isotropic Landau equation}
\author{Nicholas Gismondi, William Golding, and Matthew Novack}
\date{}
\thanks{M.N. was supported by the NSF through grant DMS-2307357.}

\begin{document}

\begin{abstract}
We establish a sharp rigidity/flexibility threshold for stationary solutions of the Krieger--Strain equation, an isotropic model of the Landau--Coulomb equation. For every $1< p < \sfrac65$, we use Nash iteration to construct nontrivial, nonnegative solutions in $L^1(\R^3) \cap L^p(\R^3)$ with arbitrarily strong exponential localization. Conversely, every stationary weak solution in $L^{\sfrac65}(\R^3)$ is trivial, identifying $L^{\sfrac65}(\R^3)$ as a new sharp integrability threshold.

To our knowledge, this is the first use of Nash iteration for a nonlinear equation from collisional kinetic theory. The construction is based on a high--high--low cancellation within the Krieger--Strain operator and suggests that such mechanisms may occur more broadly in kinetic theory. The construction must accommodate kinetic features unusual for the method including a strongly nonlocal collision operator; an equation fundamentally posed on the whole space---not the periodic box; and a positive scalar unknown.

At this low level of regularity, the usual formulations of the collision operator are not a priori well-defined, so a central part of the problem is specifying in what sense the constructed objects solve the equation. We isolate the notion of \emph{mollifier confluence}, a simple and canonical way to interpret a nonlinearity below naive thresholds related to multiplying distributions. We complement this definition with a systematic treatment of weak solution notions and several explicit formal computations and clarifying examples that may be of independent interest.

\end{abstract}

\maketitle

\setcounter{tocdepth}{1}
\tableofcontents

\section{Introduction}

\subsection{The Isotropic Landau Model}
In this paper, we study the isotropic Landau equation, a spatially homogeneous kinetic model first introduced by Krieger and Strain in \cite{KriegerStrain}; see also \cite{GressmanKriegerStrain}:
\begin{equation}\notag
    \partial_t f = Q_{\rm KS}(f) \qquad \text{on } \R^+\times \R^3.
\end{equation}
Here $f$ is an unknown probability density $f(t,v)\colon \R^+ \times \R^3 \to \R^+$ and $Q_{\rm KS}(f)$ is the corresponding collision operator.  The operator $Q_{\rm KS}$ is a mathematically convenient simplification of the Landau collision operator with Coulomb potential: it retains several of its central analytical features, including quadratic nonlinearity; nonlocality; the Coulomb singularity; and the associated scaling, while replacing the anisotropic diffusion matrix by a scalar coefficient.

For smooth functions with sufficient decay, $Q_{\rm KS}(f)$ admits several equivalent representations; it can be written in collisional form (the integral form below) or as a quasilinear elliptic operator in either divergence or non-divergence form (second and third lines below, respectively):
\begin{equation}\label{defn:collision_operator_smooth}
\begin{aligned}
    Q_{\rm KS}(f) &\coloneqq \nabla_v \cdot \int_{\R^3} \frac{f(w)\nabla_v f(v) - f(v)\nabla_w f(w)}{4\pi\abs{v-w}} \dd w\\
        &= \nabla_v \cdot \left(a[f]\nabla_v f - \nabla_v a[f]f\right)\\
        &= a[f]\Delta f + f^2\, .
\end{aligned}
\end{equation}
Here, the choice of $4\pi$ normalizes the nonlocal diffusion coefficient $a[f]$ to be the usual inverse Laplacian
\begin{equation*}
    a[f] \coloneqq (-\Delta)^{-1}_{\R^3} f = \frac{1}{4\pi} \int_{\R^3} \frac{f(w)}{\abs{v-w}} \dd w.
\end{equation*}
The Landau--Coulomb collision operator $Q_{\rm L}$ has an analogous elliptic structure, but its diffusion coefficient $A[f]$ is matrix-valued and anisotropic. More precisely, the Krieger--Strain model replaces the matrix coefficient $A[f]$ by the scalar coefficient $a[f] = \mathrm{tr}\,A[f]$. This replacement explains the term \textit{isotropic Landau} and makes the equation a natural testing ground for new methods for nonlinear kinetic equations.

\subsection*{Background on the Isotropic Landau Equation}

The collisional formulation in~\eqref{defn:collision_operator_smooth} yields conservation of mass, momentum, and an $H$-theorem. Unlike the Landau--Coulomb equation, however, the isotropic model does not conserve kinetic energy, which is instead monotone increasing with an explicit production functional; see Lemma~\ref{lem:conservation_laws} below. These identities provide the basic a priori estimates underlying the analysis of smooth solutions.

The first global smooth solutions were obtained by Gualdani and Guillen for a class of radially symmetric, decreasing data~\cite{GualdaniGuillen}. Subsequently, Gualdani and Zamponi constructed weak solutions for even initial data~\cite{GualdaniZamponi}. More recently, Bowman and Ji removed these symmetry restrictions and established the existence of global smooth solutions for a broad class of classical initial data~\cite{BowmanJi}.
They proved that for smooth solutions, the Fisher information 
\begin{equation}\notag
i(f) = \int_{\R^3} \frac{|\nabla f |^2}{f} \dd v = \int_{\R^3} |\nabla \log f|^2 f \dd v = 4 \int_{\R^3} |\nabla \sqrt{f}|^2 \dd v
\end{equation}
is monotone decreasing, using the framework developed by Guillen and Silvestre in~\cite{GuillenSilvestre}. The Sobolev embedding $H^1(\R^3) \hookrightarrow  L^6(\R^3)$ shows that the Fisher information controls the $L^3(\R^3)$ norm, a known subcritical quantity with respect to parabolic regularity, thereby resolving the question of finite time blowup for the smooth solutions to the Krieger--Strain equation. 

For the Landau--Coulomb equation, Ji proved that the Fisher information becomes finite instantaneously for nonnegative initial data $f_{\rm in} \in L^1_5\cap L\log L$, yielding solutions that are smooth for every positive time~\cite{Ji}. Conditional uniqueness of smooth solutions is known under additional spacetime integrability assumptions~\cite{Fournier,GoldingGualdaniLoherCritical} or for initial data satisfying $\brak{v}^m f_{\rm in} \in L^p$ for $p \ge \sfrac32$~\cite{ChernGualdani,GualdaniSun,HeJiLuo}. Below the $p = \sfrac32$ threshold, uniqueness remains widely open and a source of motivation for the present work.

In view of the shared structure of the two equations, it is natural to expect analogous rough-data existence and conditional uniqueness results for the Krieger--Strain equation, although these have not been established.

\subsection*{Guiding Analogy and Motivating Question}

One important consequence of the conservation laws and entropy dissipation is a quantitative lower bound on the diffusion coefficient:
\begin{equation}\label{diff:lower}
    a[f](t,v) \ge \frac{c_0}{\brak{v}}, \qquad \text{where }c_0 > 0\text{ depends only on time and the initial data.}
\end{equation}
A related anisotropic lower bound holds for the Landau--Coulomb diffusion $A[f]$; see \cite{Desvillettes, Ji2}. In both cases, this suggests a natural analogy with the semilinear heat equation
\begin{equation}\label{eq:semilinear_heat}
    \partial_t f = \Delta f + f^2.
\end{equation}
Heuristically, \eqref{diff:lower} suggests that the nonlocal, nonlinear diffusions in both Landau and Krieger--Strain are \emph{qualitatively stronger} than the corresponding linear diffusion in \eqref{eq:semilinear_heat}, and the quasilinear equations should inherit many stability and regularity properties known for \eqref{eq:semilinear_heat}.

The analogy is nevertheless one-sided. The semilinear heat equation admits finite time blowup for large data, whereas the nonlinear diffusions in Krieger--Strain and Landau suppress blowup via the decay of the Fisher information.
The kinetic equations therefore avoid the most severe form of instability of the semilinear model \eqref{eq:semilinear_heat}.

The analogy also offers some insight into the recurring regularity threshold $L^{\sfrac32}(\R^3)$. Both Landau and Krieger--Strain have two-parameter scaling symmetry groups, which do not select a single scaling--critical Lebesgue space. By contrast, \eqref{eq:semilinear_heat} has a one-parameter symmetry group that isolates $L^{\sfrac32}(\R^3)$ as scaling-critical. Whether this threshold is intrinsic to the kinetic models or merely an artifact of current proofs remains unclear.

Below this threshold, the semilinear heat equation additionally exhibits subtler forms of instability, even without blowup. Notably, for each $1 \le p < \sfrac32$, solutions belonging to $C([0,T];L^p(\R^3)) \cap C^\infty((0,T)\times\R^3)$ need not be unique; nontrivial, smooth forward self-similar solutions with zero initial data are constructed in \cite{HarauxWeissler}. This motivates our guiding question:
\begin{quote}
    Although the nonlinear diffusion in the Krieger--Strain equation prevents finite-time blowup, do weaker low-regularity instabilities of the semilinear heat equation persist?
\end{quote}

We do not address this question about \emph{dynamical} instabilities directly. Instead, we construct nontrivial, well-localized stationary solutions to the Krieger--Strain equation (Theorem \ref{thm:main} below). Here a new threshold emerges: nontrivial stationary solutions exist in $L^1(\R^3) \cap L^{p}(\R^3)$ for each $1 < p < \sfrac65$, whereas every stationary solution in $L^{\sfrac65}(\R^3)$ is trivial. Thus, the stationary problem selects a threshold distinct from the $L^{\sfrac32}(\R^3)$ exponent suggested by the semilinear heat equation.

These low-regularity stationary solutions exhibit a weak form of instability: they never regularize. Consequently, any class of weak solutions containing them cannot satisfy instantaneous $L^p\to L^\infty$ regularization estimates. They do not, however, establish instability for smooth solutions; whether they arise as limits of smooth solutions remains open.

\subsection*{Heuristic Discussion of Weak Solutions}

The central issue is therefore one of admissibility: What should count as a stationary weak solution at low regularity? Formal computations are illuminating, but insufficient. The semilinear model equation already illustrates this point. Consider, for example, the stationary elliptic problem associated to \eqref{eq:semilinear_heat},
\begin{equation}
    \Delta u + u^2 = 0 \quad \text{ on } \R^3 \, , \qquad \text{ which has multiple low-regularity solutions including }\quad u_{\rm sing} = -\frac{2}{\abs{x}^2}.
\end{equation}
The profile $u_{\rm sing}$ is signed and lies in the critical space $u_{\rm sing} \in L^{\sfrac32,\infty}(\R^3)\setminus L^{\sfrac32}(\R^3)$, precisely identifying the barrier to regularity for \eqref{eq:semilinear_heat}. Nevertheless, $u_{\rm sing}^2 \notin L^1_{\rm loc}(\R^3)$, so $u_{\rm sing}$ does not satisfy a standard distributional formulation. Interpreting $u_{\rm sing}$ as a stationary solution requires significant care and can be misleading.  Despite its apparent simplicity and relevance, it should be excluded; see Appendix \ref{appendix:power_law_elliptic} for details. 

Analogous, as well as genuinely nonlocal, difficulties arise when searching for exotic radial solutions to the Krieger--Strain equation. A formal computation produces several simple explicit profiles, including
\begin{equation}
    f_{\rm sing}(v) = C\abs{v}^{-\sfrac32}, \qquad \frac{\sin(\abs{v})}{\abs{v}}, \qquad \text{and} \qquad \frac{\sinh(\abs{v})}{\abs{v}}.
\end{equation}
However, none of these formal profiles is admissible as a steady state for the original kinetic equation. The oscillatory profile changes sign; the hyperbolic profile grows at infinity; and the power--law is both insufficiently localized and exhibits a hidden defect measure---it is more appropriately considered a solution in some generalized sense of $Q_{\rm KS}(f_{\rm sing}) = \delta_0$. Appendix~\ref{appendix:power_law} treats in detail the most instructive case, the power--law profile.

These failures do not make the radial analysis irrelevant. On the contrary, the reduction isolates the possible singular behaviors and reveals a surprisingly rich family of explicit formal solutions. For the present work, however, these examples motivate three basic admissibility requirements for stationary solutions:
\begin{itemize}
    \item \textit{Positivity.} The kinetic interpretation of $f$ as a probability density requires $f \ge 0$. This excludes sign-changing profiles such as $\sin(\abs{v})/\abs{v}$.
    \item \textit{Decay and Localization.} 
    The fundamentally nonlocal structure of the collision operator $Q_{\rm KS}$ makes the behavior of $f$ at large velocities relevant. We therefore require sufficient decay to control its tails, while the physical interpretation of moments further motivates imposing weighted $L^1$ localization. This excludes growing or insufficiently localized profiles such as $\sinh(\abs{v})/\abs{v}$.
    \item \textit{Absence of hidden defect measures.} Weak solutions should serve, at least heuristically, as possible limits of smooth solutions. We therefore seek a canonical interpretation of the equation that is stable under regularization. This excludes pointwise solutions that contain hidden defect measures such as the power-law $\abs{v}^{-\sfrac32}$; see Appendix~\ref{appendix:power_law} for details.
\end{itemize}
We now formalize these requirements by introducing notions of collisional weak solutions and mollifier-confluent solutions and proving consistency with the notions of classical solution defined by \eqref{defn:collision_operator_smooth}.

\subsection*{Notions of stationary solution}

For a smooth, rapidly decaying density $f$, the expressions in \eqref{defn:collision_operator_smooth} are equivalent and define $Q_{\rm KS}(f)$ unambiguously. At lower regularity, however, these expressions no longer have the same a priori meaning: the nondivergence form contains the product $f^2$, while the divergence and collisional forms involve nonlocal products whose definition requires both local regularity and control at infinity. Since we study the stationary equation $Q_{\rm KS}(f)=0$ below the natural integrability threshold for these products, the notion of solution must be made precise for a rigorous statement of our main results.

\medskip

\subsubsection*{Collisional weak solutions.} 

\;\newline 

\noindent Formally, integrating by parts twice in $\R^6$ in the collisional form of $Q_{\rm KS}(f)$ and symmetrizing in $v$ and $w$ yields an expression for the pairing of $Q_{\rm KS}(f)$ with a test function $\varphi \in C^\infty_c(\R^3)$:
\begin{equation}\label{eq:collisional_weak_form}
\begin{aligned}
    \int_{\R^3} \varphi(v) Q_{\rm KS}(f)(v) \dd v &= -\frac{1}{4\pi}\int_{\R^3}\int_{\R^3} \nabla\varphi(v) \cdot \frac{(\nabla_v - \nabla_w)(f(v)f(w))}{\abs{v-w}} \dd v \dd w\\
        &= -\frac{1}{8\pi}\int_{\R^3}\int_{\R^3} \left[\nabla_v\varphi - \nabla_w\varphi\right] \cdot \frac{(\nabla_v - \nabla_w)(f(v)f(w))}{\abs{v-w}} \dd v \dd w\\
        &= \frac{1}{8\pi}\int_{\R^3}\int_{\R^3} f(v)f(w)(\nabla_v - \nabla_w)\cdot  \left(\frac{\nabla_v\varphi - \nabla_w\varphi}{\abs{v-w}}\right) \dd v \dd w.
\end{aligned}
\end{equation}
This formulation removes the need for regularity of $f$, but still requires sufficient integrability for the convergence of the integrals on the right-hand side. A sufficient condition is $f\in \dot{H}^{-1}(\R^3)$:
\begin{equation}\label{eq:weak:65}
\begin{aligned}
   &\frac{1}{8\pi}\abs{\int_{\R^3}\int_{\R^3} f(v)f(w)(\nabla_v - \nabla_w)\cdot  \left(\frac{\nabla_v\varphi - \nabla_w\varphi}{\abs{v-w}}\right) \dd v \dd w}\\
    &\qquad\qquad\qquad\qquad\qquad\qquad\qquad\le C(\varphi) \int_{\R^3}\int_{\R^3} \frac{f(v)f(w)}{\abs{v-w}} \dd v \dd w = C(\varphi)\norm{f}_{\dot{H}^{-1}(\R^3)}^2.
\end{aligned}
\end{equation}
Thus, $Q_{\rm KS}(f)$ is well-defined as a distribution, or more precisely as an element of $W^{-2,1}(\R^3)$, whenever $f\in \dot{H}^{-1}(\R^3)$.

\begin{definition}\label{def:weak_soln}
    We say a nonnegative $f\in L^{\sfrac 65}(\R^3)$ is a \textnormal{stationary collisional weak solution} to Krieger--Strain if $Q_{\rm KS}(f) = 0$ in the sense of distributions, where $Q_{\rm KS}(f)$ is defined distributionally via \eqref{eq:collisional_weak_form}.
\end{definition}

\begin{remark}
    To obtain the weak formulation \eqref{eq:collisional_weak_form} we merely integrated by parts and performed a change of variables using Fubini's theorem. For $f$ sufficiently regular and localized, it follows that $f$ is a stationary weak solution if and only if it is a classical stationary solution defined by any of the formulas in \eqref{defn:collision_operator_smooth}.
\end{remark}

We will present the rigidity result for collisional weak solutions in Theorem \ref{prop:rigidity} below.

\medskip

\subsubsection*{Mollifier-confluent solutions.}

\,\newline

\noindent For a general nonnegative $L^1(\R^3)$ density, the $\dot{H}^{-1}$ seminorm may be infinite, precluding the use of the weak formulation \eqref{eq:collisional_weak_form} to define $Q_{\rm KS}$. Instead, we introduce a new notion of solution based on regularization by mollification and passage to the limit. We say that $\rho \in \mathcal{S}(\R^3)$ is an admissible mollifier if $\rho \ge 0$ and $\int_{\R^3} \rho = 1$. 
As usual for $\eps > 0$, define mollification to scale $\eps$ via
\begin{equation*}
\rho_\eps(v)\coloneqq \frac{1}{\eps^{3}}\rho\left(\frac{v}{\eps}\right)\qquad \text{and} \qquad f_\eps^\rho\coloneqq f\ast\rho_\eps\, .
\end{equation*}
The superscript $\rho$ emphasizes that ``mollification to scale $\epsilon$'' depends on a choice of kernel; the definition of solution below requires consistency across all admissible mollifiers $\rho$.

\begin{definition}\label{def:moll:conf}
    We say a nonnegative $f\in L^1(\R^3)$ is a \textnormal{stationary mollifier-confluent solution} to Krieger--Strain if 
    \begin{equation*}
        \lim_{\eps \to 0^+} Q_{\rm KS}\left(f_{\eps}^\rho\right) = 0 \qquad \text{in the sense of distributions, for every admissible mollifier }\rho.
    \end{equation*}
     Here $Q_{\rm KS}\left(f_\eps^\rho\right)$ is defined unambiguously via \eqref{defn:collision_operator_smooth}.
\end{definition}

The ``confluence'' in mollifier-confluent refers to the independence of the notion of solution with respect to the mollification process.  It is likely possible to construct a pathological nonnegative function $f$ for which $Q_{\rm KS}(f^\rho_\epsilon) \rightarrow 0$ as $\epsilon \rightarrow 0^+$ for some admissible mollifiers $\rho$, but not for others.\footnote{At a heuristic level, possible inconsistencies could arise between two mollifiers with effective spatial supports of very different diameters, or equivalently (by the uncertainty principle) between two mollifiers with effective frequency supports of very different diameters.  Mollifiers with spatial supports of larger diameters can detect possible cancellations at larger scales, which could in principle be invisible to mollifiers with narrower spatial supports.}  Imposing consistency across all admissible mollifiers allows for a canonical interpretation of $Q_{\rm KS}(f)$ for those rough densities $f$ which satisfy Definition~\ref{def:moll:conf}.  Our definition has precedents in parts of the PDE literature, such as in work of Christ~\cite{Christ1, Christ2} on generalized solutions to the nonlinear Schr\"odinger and Navier--Stokes equations.  In particular, Christ points out that general theories of multiplication for distributions do exist but depend on the approximation procedure~\cite[pg. 3]{Christ1}.  Similar notions of solutions were also employed by Lemarié--Rieusset~\cite{LR} for the 2D Navier--Stokes equations, the first and third authors~\cite{ABGN} for the 2D Navier--Stokes equations, the third author~\cite{GMPR} for the KdV equation, and Cheskidov and Hou~\cite{CH} for the $d$--dimensional fractional Navier--Stokes equations.

We now prove that mollifier-confluent solutions generalize collisional weak solutions, in the sense that the collisional weak form and mollifier-confluent form are equivalent for densities in $L^1(\R^3) \cap L^{\sfrac 65}(\R^3)$. 
For convenience, we introduce an auxiliary symmetric, bilinear operator $Q(g,h)$, defined distributionally via
\begin{equation*}
    \left\langle Q(g,h), \varphi \right\rangle \coloneqq \frac{1}{8\pi} \int_{\R^3} \int_{\R^3} g(v) h(w) (\nabla_v - \nabla_w) \cdot \left( \frac{\nabla_v \varphi -\nabla_w \varphi}{|v-w|} \right) \dd v \dd w.
\end{equation*}
The main tool is a simple bilinear estimate for $Q(g,h)$:
\begin{equation}\label{eq:Q:symm}
    \abs{\left\langle Q(g,h), \varphi \right\rangle}  \lesssim \norm{g}_{L^{\sfrac65}}\norm{h}_{L^{\sfrac65}}\norm{\nabla^2\varphi}_{L^\infty},
\end{equation}
which is a consequence of the Hardy--Littlewood--Sobolev inequality. Now, for any $f \in L^1(\R^3) \cap L^{\sfrac65}(\R^3)$ and any admissible mollifier $\rho$,  bilinearity of $Q$ implies the identity
\begin{equation}\label{eq:something}
\begin{aligned}
    Q_{\rm KS}\left(f_\epsilon^\rho\right) &= Q\left(f_\epsilon^\rho, f_\epsilon^\rho\right)\\
        &= Q\left(f_\epsilon^\rho-f, f_\epsilon^\rho\right) + Q\left(f, f_\epsilon^\rho-f\right) + Q(f, f).
\end{aligned}
\end{equation}
Since $f \in L^{\sfrac65}(\R^3)$, a standard property of mollifiers implies $f^\rho_\eps \to f$ strongly in $L^{\sfrac65}(\R^3)$. Therefore, combining the estimate \eqref{eq:Q:symm} and identity \eqref{eq:something}, we obtain
\begin{equation}
    \lim_{\eps \to 0^+} Q_{\rm KS}\left(f_\epsilon^\rho\right) = Q_{\rm KS}(f) \qquad \text{in the sense of distributions.}
\end{equation}
We conclude that $f$ is a collisional weak solution if and only if $f$ is a mollifier-confluent solution.
The main result for mollifier-confluent solutions is flexibility in Theorem \ref{thm:main} below. 

\subsection{Main results and commentary}

We first present the rigidity theorem for collisional weak solutions.

\begin{theorem}[\textbf{Rigidity}]\label{prop:rigidity}
Suppose $f$ is any nonnegative, stationary collisional weak solution for Krieger--Strain in the sense of Definition~\ref{def:weak_soln}. Then $f \equiv 0$.
\end{theorem}

\noindent The proof of rigidity is based on the failure of conservation of kinetic energy for the isotropic Landau equation, and consequently does not apply to the Landau equation. For the Landau equation, the analogous proof uses the entropy/entropy dissipation identity that requires greater regularity of the solutions.

In contrast to the above rigidity result, we also prove the following flexibility result.

\begin{theorem}[\textbf{Flexibility}]\label{thm:main}
Let $1<p_0<\sfrac 65$ be fixed. There exist nontrivial, nonnegative functions
$$
f \in L^{p_0}(\R^3, dv) \cap \bigcap_{m=0}^\infty L^1(\R^3, \exp(\langle v \rangle^m)\, dv)
$$
which are mollifier-confluent solutions of the stationary Krieger--Strain equation in the sense of Definition~\ref{def:moll:conf}.
\end{theorem}

\noindent At the moment, we do not explore in detail possible connections between our notions of solution---Definition~\ref{def:weak_soln} and Definition~\ref{def:moll:conf}---and other specialized notions of low-regularity solutions tailored to the Landau and Boltzmann equations, such as $H$-solutions or renormalized solutions. We view Theorem~\ref{thm:main} as the first step in a program to clarify the properties of and rigidity/flexibility thresholds of solutions to kinetic equations.

In the kinetic literature, weak formulations analogous to Definition~\ref{def:weak_soln} are relatively standard for homogeneous equations. In the spatially inhomogeneous setting, however, the natural conservation laws and entropy inequality generally provide comparatively weaker a priori estimates that are not strong enough to construct distributional solutions as in Definition~\ref{def:weak_soln}.
Instead, the standard notion of ``global weak solutions'' is built around the renormalized solutions introduced by DiPerna and Lions for the inhomogeneous Boltzmann equation~\cite{DPL}; for Landau, see~\cite{Lions94,Villani96}. Renormalization yields a formulation compatible with the available a priori bounds, at the cost of imposing a substantially weaker notion of solution and, in some settings, allowing a defect measure.

We emphasize that the uniqueness of renormalized solutions remains an longstanding open problem of some importance. Renormalized solutions underlie the formal derivations of many models in kinetic theory, including grazing collision limits from Boltzmann to Landau and hydrodynamic limits from kinetic equations to fluid equations. However, due to the variety of pathologies permissible in the renormalized formulation, we suspect that this problem can be resolved in the negative by further development of our techniques. 

We therefore formulate the following conjecture.

\begin{theorem-non}[\textbf{Flexibility of Renormalized Solutions}]
\textit{There exist nonnegative, steady, renormalized solutions to the Landau and Boltzmann equations with finite mass, momentum, energy, and entropy, but which are not Maxwellians.  In addition, renormalized solutions to the Cauchy problem for the inhomogeneous Landau and Boltzmann equations are not unique for a dense class of initial data.}
\end{theorem-non}

We also view Theorem~\ref{thm:main} as indicative of the instability of solutions to the full nonlinear problem in $L^\infty(\R^+;L^{\sfrac 65-})$.  It seems plausible that our techniques will lead to the impossibility of certain conjectured smoothing estimates for the Krieger--Strain equation, and likely for Landau as well. For example, this may also rule out conjectured enhanced dissipation estimates (see work of Cabrera, Gualdani, Guillen~\cite{CGG}) as a method toward uniqueness. We leave investigation of these matters to future work.   

\subsection{A quick outline of the proof}
The proof of rigidity uses the formal identity 
$$ \frac{d}{dt} \int_{\R^3} f(v) |v|^2 \, dv = 4 \| f \|_{\dot H^{-1}(\R^3)}^2 \, , $$
observed for example in the survey of Gualdani and Zamponi~\cite{GZSurvey}.  For stationary solutions, the left--hand side must be zero, and thus the right--hand side must be zero, implying that the solution is trivial.  Justifying this formal argument requires $L^{\sfrac 65}(\R^3)$ integrability, since $L^{\sfrac 65}(\R^3) \subset \dot H^{-1}(\R^3)$.  

Since rigidity holds in $L^{\sfrac 65}(\R^3)$, which from~\eqref{eq:weak:65} is the minimum amount of integrability required to interpret the collision operator by classical methods, any construction of nontrivial stationary solutions must provide a canonical way to interpret the collision operator which is more subtle than H\"older's inequality.  The variant of Nash iteration introduced by the first and third authors in~\cite{ABGN} fulfills this need.  The basic principle of this variant is that the constructive nature of the iteration provides a way to pass to the limit in possibly ill--defined product terms.  This can be illustrated in the following simple example.  Consider the $\T$--periodic distribution $u = \sum_{k \geq 0} \sin(2\pi 2^k x)$.  It is clear that $u \notin L^2(\T)$, but $u \in H^{-1}(\T)$.  Therefore it is not obvious how to interpret $(u^2)'$ as a distribution on $\T$.  However, for any partial sum, 
\begin{align*}
    \left[ \left( \sum_{k \leq K} \sin(2\pi 2^k x) \right)^2 \right]' &= \left( \sum_{\substack{k \neq k' \\ k, k' \leq K}} \sin(2\pi 2^kx) \sin(2\pi 2^{k'}x) + \sum_{k \leq K} \frac{1-\cos(2\pi 2^{k+1}x)}{2}  \right)' \, .
\end{align*}
The sum over $k \neq k'$ converges in $H^{-1}(\T)$ as $K \rightarrow \infty$, since each term gains $\approx 2^{-\max(k,k')}$ from the $H^{-1}(\T)$ weight, and a short computation shows that $\sum_{k \neq k'}2^{-\max(k,k')}$ is absolutely summable over $k$ and $k'$.  In addition, the term $\sum_{k \leq K} \frac{-\cos(2\pi2^{k+1}x)}{2}$ converges in $H^{-1}(\T)$ as $K \rightarrow \infty$ as well.  The only problematic term is $\sum_{k \leq K} \sfrac 12$, which is not absolutely summable.  However, this sum is \textit{inside of a derivative}, implying that it vanishes for every $K$, and thus in the limit.  

In the context of the Krieger--Strain equation, and our iteration, we can make the following analogies. The equation $(u^2)'=0$ is replaced with $Q_{\rm KS}(f) := \pa_{ii} B(f,f)=0$, where $B$ is a quadratic, nonlocal integro--differential operator which will be defined explicitly in Section \ref{sec:reform_nonlin}.  Interpreting the operator by H\"older's inequality in this case requires $L^{\sfrac 65}(\R^3)$ regularity from~\eqref{eq:weak:65}, in contrast to $L^2(\T)$ regularity required to interpret the square.  Partial Fourier sums are replaced with mollifications with nonnegative Schwartz functions.  The differential operator on the outside of the complicated bilinear operator still annihilates any constant modes which grow as the mollification parameter approaches zero.  This annihilation ``renormalizes'' the collision operator.  The simple wave $\sin(2\pi 2^k x)$ is replaced with $g_k = a_k U_k$, where $a_k$ has relatively low--frequencies, and $U_k$ has relatively high frequencies, unit $L^{\sfrac 65}$ norm, but very small $L^1$ norm.  The terms $B(g_k,  g_{k'})$ for $k \neq k'$ converge for essentially the same reason as in the toy problem: they are well--separated in frequency, and therefore summable in negative Sobolev spaces.  However, the terms $B(g_k, g_k)$ will produce a ``high--high--low'' interaction, visible in the toy example through the trigonometric identity
$$ \sin^2(2\pi 2^kx) = \frac{1-\cos(2\pi2^{k+1}x)}{2} = \underbrace{\dashint_{\T} \sin^2(2\pi 2^ky) \, dy}_{ = \bp_{\rm low} \sin^2(2\pi 2^k x)} + \underbrace{\left( \Id - \dashint_\T \right) \sin^2(2\pi 2^kx)}_{{ = \bp_{\rm high} \sin^2(2\pi 2^k x)}}  \, . $$
The ``high--high--low'' term is the first term, and is so named because the multiplication of two high frequencies creates a low frequency, namely the mean, and leftover high frequencies.  We now further examine the high--high--low terms produced by $\sum_{k \leq K} B(g_k, g_k)$ and explain that the reason for their convergence as $K \rightarrow \infty$ is intimately connected to the fundamental cancellation mechanism behind all Nash iterations.  

In our iteration, we suppose that $f_K  = \sum_{k \leq K} g_k$ solves 
$$ \pa_{ii} B(f_K, f_K) = \pa_{ii} E_K \, . $$
We wish to construct $f_{K+1} = f_K + g_{K+1}$ in such a way so that $\|E_{K+1}\| \ll \| E_K \|$, in a norm which combines a negative Sobolev norm with a homogeneous first--order Sobolev norm.  Then if $f_K \rightarrow f$ and $E_K \rightarrow 0$, we can hope to find a solution in the limit.  Towards this end, we design $g_{K+1} = a_{K+1} U_{K+1}$ so that
\begin{equation}\label{eq:hhl:intro}
\bp_{\rm low} B(g_{K+1} , g_{K+1}) = -E_K \, . 
\end{equation}
Then
\begin{equation}\label{eq:cancellation}
\begin{split}
\pa_{ii} \left( E_K + B(g_{K+1} , g_{K+1}) \right) &= \pa_{ii} \left( E_K + \bp_{\rm low} B(g_{K+1} , g_{K+1}) +  \bp_{\rm high} B(g_{K+1}, g_{K+1}) \right)  \\
&= \pa_{ii} \left( \bp_{\rm high} B(g_{K+1}, g_{K+1})  \right)  \, .  
\end{split}
\end{equation}
The remaining term is high--frequency and can be made very small in negative Sobolev spaces by choosing the frequencies of $U_{K+1}$ to be very large.  This is the ``cancellation mechanism'' which drives all Nash iterations.   Then we can make an analogy between $\sum_{k \leq K} \sfrac 12$, the ``high--high--low'' term in the toy problem, and the ``high--high--low'' terms
$$ \sum_{k} \bp_{\rm low} B(g_{k+1}, g_{k+1}) = \sum_k -E_{k} $$
which appear in our iteration: in our setting, $-E_k$ essentially includes a component which is small in a negative Sobolev norm, and a large constant which is annihilated by the homogeneous first--order Sobolev norm.  The convergence of the high--high--low terms follows simply from the fact that Nash iteration \textit{prescribes} the high--high--low terms, so as to obtain a solution of the nonlinear PDE in the limit.  Convergence in the nonlinearity is then equivalent to summability of the error terms $E_k$ along the stages of the iteration.

There are several difficulties in carrying out the iteration described above.
First, designing $g_{K+1}$ so that~\eqref{eq:hhl:intro} holds requires a detailed analysis of the rather complicated operator $B(g,g)$.  Our strategy for this analysis is to rewrite $B(g,g)$ as a bilinear Fourier integral operator and use the frequency properties of $g_{K+1}$ to simplify and analyze the leading order behavior of the multiplier; see Proposition~\ref{prop:hhl} and its proof. The second main difficulty is in proving that
\begin{equation}\label{eq:full:limit}
\lim_{K \rightarrow \infty} \pa_{ii} B(f_K, f_K) = \lim_{K \rightarrow \infty} \pa_{ii} E_K = 0 \qquad \implies \qquad \lim_{\epsilon \rightarrow 0} \pa_{ii} B(f^\rho_\epsilon, f^\rho_\epsilon) = 0
\end{equation}
for all admissible mollifications $f^\rho_\epsilon$ of $f$.  In the work~\cite{ABGN} of the first and third authors, this difficulty was avoided by adopting a slightly weaker notion of solution involving a dyadic paraproduct frequency decomposition; the same definition was later adopted by Cheskidov and Hou~\cite{CH}.  The upshot of using a dyadic frequency decomposition is that $\lim_{\epsilon\rightarrow 0} \pa_{ii} B(f_\epsilon, f_{\epsilon})$ must only be computed along the dyadic subsequence of $\epsilon$'s given by $\epsilon_k = 2^{-k}$.  The side effect is that the convergence could in principle depend on the choice of Littlewood--Paley cutoff, and the fact that only dyadic subsequences are considered.  The convergence in the Nash iteration outlined in~\cite{ABGN} does not in fact depend on taking the limit along a dyadic subsequence, although it was not proven there.   Proving the analogous statement in the setting of Krieger--Strain requires a fair amount of detailed analysis; see Subsection~\ref{ss:limit}.  Since~\eqref{eq:full:limit} is not a statement which holds in general (in other words, convergence along a subsequence does not imply convergence along the whole sequence), proving it requires detailed properties of $f = \sum_k g_k$.  We address this through extensive inductive assumptions on the $g_k$'s and their properties; see Subsection~\ref{ss:ind}, which contains the inductive hypotheses.

\subsection{Further literature and background}


The theory of weak solutions to kinetic equations begins with the foundational work of DiPerna and Lions~\cite{DPL}, which introduced renormalized solutions to the Boltzmann equation and established global existence from initial data with finite mass, energy, and entropy.  Lions subsequently extended aspects of the theory to the Landau equation~\cite{Lions94}.  In the same vein, Villani introduced and proved global existence of renormalized solutions with defect measure in~\cite{Villani96}, and introduced $H$-solutions in~\cite{VillaniWeak}; we refer to his survey~\cite{VillaniReview} for further discussion.  A separate direction of research has been the extension of the De Giorgi--Nash--Moser theory to hypoelliptic kinetic equations.   Golse and Vasseur obtained H\"older regularity for kinetic Fokker--Planck equations with rough diffusion coefficients~\cite{GolseVasseur}, and Golse, Imbert, Mouhot, and Vasseur subsequently established a Harnack inequality and applied it to the Landau equation~\cite{GolseImbertMouhotVasseur}.  For the non-cutoff Boltzmann equation, Imbert and Silvestre proved a weak Harnack inequality and H\"older estimates~\cite{ImbertSilvestreHarnack}, and later obtained global smooth a priori estimates conditional on bounds for the macroscopic quantities~\cite{ImbertSilvestreGlobal}.  For the spatially homogeneous Landau equation with Coulomb potential, a substantial regularity and well-posedness theory has also been developed near and above the scaling-critical space $L^{\sfrac 32}(\R^3)$~\cite{ChernGualdani,GoldingThesis,GoldingGualdaniLoherRegularization,GoldingGualdaniLoherCritical}.  Crucially, the Fisher--information methods described earlier in the introduction have led to global regularity results for Landau~\cite{GuillenSilvestre,Ji,DesvillettesGoldingGualdaniLoher}, with analogous results for the Krieger--Strain equation obtained in~\cite{BowmanJi}.  We refer to the surveys of Gualdani and Zamponi~\cite{GZSurvey} and Imbert and Silvestre~\cite{ImbertSilvestreSurvey}, and the references therein, for further discussion.  Despite this progress, uniqueness of renormalized solutions remains open, and, to our knowledge, prior to the present work no flexibility construction was known for a fixed Landau--type collision operator.
  
The techniques used in the proof of Theorem~\ref{thm:main} belong to the Nash iteration lineage, which begins with Nash's seminal work~\cite{Nash} on $C^1$ isometric embeddings.  These techniques were adapted to the incompressible Euler equations in the foundational papers of De Lellis and Sz\'ekelyhidi in~\cite{DLS09, DLS13}, and to the incompressible Navier--Stokes equations in the foundational paper of Buckmaster and Vicol~\cite{BV19}. They have since been used extensively for Euler, Navier--Stokes, active scalar, and other fluid equations~\cite{Isett,BDLSV,DaneriSzekelyhidi,BSV,NovackVicol,GKN23,GKNStrong,GKNHelicity,LuoStationary,CheskidovLuoSharp,BCV,GiriRadu,BrueColomboKumarEuler}, as well as for the transport and continuity equations~\cite{ModenaSzekelyhidi,CheskidovLuoTransport,BrueColomboDeLellis,BrueColomboKumarTransport} and magnetohydrodynamics~\cite{BeekieBuckmasterVicol,FaracoLindbergSzekelyhidi,FaracoLindbergSzekelyhidiHelicity,GiardiSzekelyhidi,EncisoPenafielPeraltaMHD}.  We refer to the surveys of De Lellis and Sz\'ekelyhidi~\cite{DeLellisSzekelyhidiSurvey} and Buckmaster and Vicol~\cite{BuckmasterVicolSurvey} and references therein for further discussion. However, there are no examples prior to the present work which utilize Nash iteration/convex integration to construct stationary solutions to a kinetic equation.  The closest relative of the present work might be L\"u~\cite{LuFokkerPlanck}, which constructs non-unique solutions to several nonlinear Fokker--Planck equations and their associated distribution--dependent stochastic differential equations by constructing an additional rough divergence--free drift field; in contrast, the collision operator in the present work is fixed and no auxiliary drift is added.

\section{Toolkit and cancellation mechanism}
\subsection{Fourier analysis}

\begin{definition}[\textbf{Projection onto and off the mean of periodic functions}]
    Let $u \in L^1(\T^3)$. Then we define
    $$
    \mathbb{P}_{=0}(u) = \int_{\T^3} u(x)\, dx
   \qquad \textnormal{and}
    \qquad  \mathbb{P}_{\not=0}(u) = u - \mathbb{P}_{=0}(u) \, . $$
\end{definition}

\begin{definition}[\textbf{Fourier transform and Fourier coefficients}]
    For $f \in \mathcal{S}(\R^3)$ we define the Fourier transform by
    $$
    \hat{f}(\xi) = \int_{\R^3} f(x)e^{-2\pi i \xi \cdot x}\, dx\, .
    $$
    This definition extends to $f \in \mathcal{S}'(\R^3)$ by duality. Likewise, for $u \in \mathcal{D}'(\T^3)$ we define Fourier coefficients by
    $$
    \hat{u}(\xi) = \langle u, e^{-2\pi i \xi \cdot x}\rangle_{\mathcal{D}'(\T^3), \mathcal{D}(\T^3)}\, .
    $$
    When $u \in L^1(\T^3)$ the above coincides with the classical definition of Fourier coefficients given by
    $$
    \hat{u}(\xi) = \int_{\T^3} u(x)e^{-2\pi i \xi \cdot x}\, dx\, .
    $$
    We will frequently identify functions and distributions on $\T^3$ with their $\Z^3$-periodic lifts to $\R^3$. Under the identification $\T^3 \simeq [0,1]^3$, the Fourier transform (in the sense of tempered distributions) of a $\Z^3$-periodic distribution coincides with its Fourier coefficients.
\end{definition}

\begin{remark}
    If $u \in L^1(\T^3)$ then $\mathbb{P}_{=0}(u) = \hat{u}(0)$. We will utilize this observation frequently.
\end{remark}

\begin{definition}[\textbf{Spatial and frequency cutoffs}]\label{def:projs}
    Let $\chi:\R^3 \to \R$ be smooth, radial, monotone decreasing in every radial direction, $0 \leq \chi \leq 1$, $\chi(x) = 1$ for $|x| < 1$, $\chi(x) = 0$ for $|x| > 2$, and $|\nabla \chi| \leq 2$. For $R> 0$ put $\chi_R(x) = \chi(2x/R)$.  In addition, define the convolution type operator $\mathbb{P}_{\leq R}$ by
    $$
    \left(\mathbb{P}_{\leq R} (f)\right)^{\wedge}(\xi) = \chi_R(\xi) \hat{f}(\xi)\, .
    $$
    We also put $\mathbb{P}_{>R}(f) = f - \mathbb{P}_{\leq R}(f)$.
\end{definition}

\begin{lemma}[\textbf{Bounds on low--frequency projectors}]\label{lem:proj_est}
    For $\lambda > 0$ we have that
    $$
    \Vert \mathbb{P}_{\leq \lambda} (f)\Vert_{L^p(\R^3)} \leq \Vert \hat{\chi} \Vert_{L^1(\R^3)} \Vert f \Vert_{L^p(\R^3)}\, .
    $$
\end{lemma}
\begin{proof}
    This is just the Young convolution inequality with
    $$
    \Vert \left(\chi_{\lambda}\right)^{\wedge} \Vert_{L^1(\R^3)} = \Vert \lambda^3 \hat{\chi}(\lambda x)\Vert_{L^1(\R^3)} = \Vert \hat{\chi} \Vert_{L^1(\R^3)}\, .
    $$
\end{proof}

\begin{lemma}[\textbf{Bounds on high--frequency projectors}]\label{lem:good_kernel}
    For $f \in \mathcal{S}(\R^3)$, $1 \leq p \leq \infty$, $k \geq 0$, and any $N \geq 0$ we have that
    \begin{equation}\label{eq:good_kernel_est}
        \Vert \mathbb{P}_{>\lambda}(f) \Vert_{W^{k,p}(\R^3)} = \Vert f - \mathbb{P}_{\leq \lambda}(f) \Vert_{W^{k,p}(\R^3)} \lesssim_{N,\chi,k} \lambda^{-N} \Vert f \Vert_{W^{k+N,p}(\R^3)}\, .
    \end{equation}
    The implicit constant above depends on $N$, $k$, and our specific choice of Littlewood-Paley cutoff but not $p$, $\lambda$, or $f$.
\end{lemma}
\begin{proof}
    If $N=0$, the Lemma is trivial, so assume $N \geq 1$. We will prove this using induction on $k$. Let us start with the base case $k=0$. For notational simplicity let us put $K_\lambda(\cdot) = \lambda^3 K(\lambda \cdot) = \left(\chi_\lambda\right)^{\vee}(\cdot)$ so that $\hat{K}_\lambda = \chi_\lambda$ and
    \begin{equation}\label{eq:convolution_form_low_freq_cutoff}
        \mathbb{P}_{\leq \lambda} (f)(x) = \int_{\R^3} K_\lambda(y)f(x-y)\, dy
    \end{equation}
    Recall that $\chi_\lambda$ is identically $1$ in a neighborhood of the origin. Thus $\partial^\alpha \chi_\lambda(0) = 0$ for all $\alpha \not=0$ and so
    \begin{equation}\label{eq:vanish_moments}
        \int_{\R^3} y^\alpha K_\lambda(y)\, dy = 0
    \end{equation}
    for all $\alpha \not =0$. In addition, since $\chi_\lambda(0) = 1$,
    \begin{equation}\label{eq:mean_kernel}
        \int_{\R^3} K_\lambda(y)\, dy = 1\, .
    \end{equation}
    Now we Taylor expand $f(x-y)$ around $x$ up to order $N$ and use the exact integral remainder \cite[Theorem 2.68]{Folland} to get
    $$
    f(x-y) = \sum_{|\alpha| < N} (-1)^{|\alpha|}\frac{y^\alpha}{\alpha!} \partial^\alpha f(x) + (-1)^N N\sum_{|\alpha|=N} \frac{y^\alpha}{\alpha!} \int_0^1 (1-t)^{N-1} \partial^\alpha f(x-ty)\, dt\, .
    $$
    Multiplying this expression by $K_\lambda(y)$, integrating in $y$, and utilizing \eqref{eq:convolution_form_low_freq_cutoff}, \eqref{eq:vanish_moments}, and \eqref{eq:mean_kernel} we obtain
    $$
    \mathbb{P}_{\leq \lambda}(f)(x) = f(x) + (-1)^N N\sum_{|\alpha|=N} \frac{1}{\alpha!} \int_0^1 (1-t)^{N-1} \left(\int_{\R^3} y^\alpha K_\lambda(y) \partial^\alpha f(x-ty)\, dy\right)\, dt
    $$
    and thus
    $$
    \mathbb{P}_{>\lambda}(f)(x) = (-1)^{N+1} N\sum_{|\alpha|=N} \frac{1}{\alpha!} \int_0^1 (1-t)^{N-1} \left(\int_{\R^3} y^\alpha K_\lambda(y) \partial^\alpha f(x-ty)\, dy\right)\, dt\, .
    $$
    Now we apply the $L^p$ norm in the $x$-variable followed by the Minkowski inequality to get
    $$
    \Vert \mathbb{P}_{>\lambda} (f) \Vert_{L^p(\R^3)} \lesssim_N \sum_{|\alpha|=N} \frac{\Vert \partial^\alpha f \Vert_{L^p(\R^3)}}{\alpha!} \int_0^1 (1-t)^{N-1} \int_{\R^3} |y|^{N} |K_\lambda(y)|\, dy\, dt\, .
    $$
    Using a change of variables we have
    $$
    \int_{\R^3} |y|^{N} |K_\lambda(y)|\, dy = \lambda^{-N}\int_{\R^3} |y|^{N} |K(y)|\, dy
    $$
    and so
    $$
    \Vert \mathbb{P}_{>\lambda} (f) \Vert_{L^p(\R^3)} \lesssim_N \lambda^{-N} \sum_{|\alpha|=N} \Vert \partial^\alpha f \Vert_{L^p(\R^3)} \int_{\R^3} |y|^{N} |K(y)|\, dy \lesssim_{N,K} \lambda^{-N} \Vert f \Vert_{W^{N,p}(\R^3)}\, .
    $$
    This completes the base case. Now we assume there is $k$ such that \eqref{eq:good_kernel_est} holds. Then we have
    \begin{equation*}
        \begin{split}
            \Vert \mathbb{P}_{>\lambda}(f) \Vert_{W^{k+1,p}(\R^3)} &= \Vert \mathbb{P}_{>\lambda}(f) \Vert_{W^{k,p}(\R^3)} + \sum_{|\alpha|=1} \Vert \mathbb{P}_{>\lambda} \left( \nabla^\alpha f\right) \Vert_{\dot{W}^{k,p}(\R^3)}\\
            &\lesssim_{N,K,k} \lambda^{-N} \Vert f \Vert_{W^{k+N,p}(\R^3)} + \sum_{|\alpha|=1} \lambda^{-N} \Vert \nabla^\alpha f \Vert_{W^{k+N,p}(\R^3)}\\
            &\lesssim_{N,K,k} \lambda^{-N} \Vert f \Vert_{W^{k+1+N,p}(\R^3)}
        \end{split}
    \end{equation*}
    completing the induction.
\end{proof}

The following is based on \cite[Definition 1.2.1]{Grafakos} and \cite[Theorem 1.2.3]{Grafakos}.
\begin{definition}[\textbf{Riesz potentials on Euclidean space and the torus}]\label{def:inv_Laplace}
    For $f \in \mathcal{S}(\R^3)$ and $0 < s < 3$ we define the Riesz potential of order $s$ by
    $$
    \left((-\Delta)^{-\sfrac{s}{2}}_{\R^3}f\right)^{\wedge}(\xi) = (2\pi |\xi|)^{-s} \hat{f}(\xi). 
    $$
    One also has the equivalent formulation given by
    $$
    (-\Delta)_{\R^3}^{-\sfrac{s}{2}}f(x) = 2^{-s} \pi^{-3/2} \frac{\Gamma\left(\frac{3-s}{2}\right)}{\Gamma\left(\frac{s}{2}\right)} \int_{\R^3} \frac{f(x)}{|x-y|^{3-s}}\, dy\, .
    $$
    The following inequality of Hardy, Littlewood, and Sobolev states for $f \in \mathcal{S}(\R^3)$, $1 < p < q < \infty$ with
    $$
    \frac{1}{p} - \frac{1}{q} = \frac{s}{3}
    $$
    we have
    $$
    \Vert (-\Delta)^{-\sfrac{s}{2}}_{\R^3} f \Vert_{L^q(\R^3)} \lesssim_{s,p} \Vert f \Vert_{L^p(\R^3)}
    \, , \qquad \, 
    \Vert (-\Delta)^{-\sfrac{s}{2}}_{\R^3} f \Vert_{L^{\frac{3}{3-s},\infty}(\R^3)} \lesssim_{s} \Vert f \Vert_{L^1(\R^3)}\, .
    $$
    In this way we are able to extend the definition of $(-\Delta)^{-\sfrac{s}{2}}_{\R^3}f$ to $f \in L^p(\R^3)$ for which the Hardy--Littlewood--Sobolev inequality holds. Likewise, for $u \in L^1(\T^3)$ with $\hat{u}(0) = 0$ and $s > 0$ we define the Riesz potential of $u$ of order $s$ by
    $$
    \left((-\Delta)^{-\sfrac{s}{2}}_{\T^3} u\right)^{\wedge}(\xi) = (2\pi |\xi|)^{-s} \hat{u}(\xi)\, , \quad \xi \in \Z^3 \setminus \{0\}\, .
    $$
\end{definition}

\begin{remark}[\textbf{Riesz potentials on the torus}]
    Notice $(-\Delta)^{-\sfrac{s}{2}}_{\T^3}u$ is well defined for $u \in L^p(\T^3)$ with $\hat{u}(0)=0$ for all $s > 0$ and all $1 \leq p \leq \infty$. In addition if $u:\T^3 \to \R$ has frequency support outside the ball of radius $\lambda$, then it is easy to show that
    $$
    \Vert (-\Delta)_{\T^3}^{-\sfrac{s}{2}}u \Vert_{L^p(\T^3)} \lesssim \lambda^{-s} \Vert u \Vert_{L^p(\T^3)}.
    $$
    This will be useful when proving Lemma \ref{lem:tail_killing}.
\end{remark}

\subsection{The reformulated nonlinearity}\label{sec:reform_nonlin}
Using our convention regarding the Fourier transform, we may rewrite the nonlinear term $f^2 + \Delta f (-\Delta)_{\R^3}^{-1}f$ as the Laplacian of a bilinear operator $B(f,f)$.  

\begin{proposition}[\textbf{Reformulation of the nonlinearity}]\label{prop:reform}
    Assume $f \in W^{k,1}(\R^3)$ for all $k \geq 0$.  Then
    \begin{align}\notag
        \pa_{ii} \left(  B(f,f) \right) = \Delta f \ilr f + f^2 \, ,
    \end{align}
    where we define
\begin{align}\label{def:B:reform}
    B(f,g) &= g \ilr f + 2 \ilr  \pa_j \left( g \pa_j \ilr f \right) \, .  
    \end{align}
Equivalently, 
$$
    B(f,g)(x) = \int_{\R^3} \int_{\R^3} M(\xi,\eta) \hat{f}(\xi) \hat{g}(\eta) e^{2\pi i (\xi + \eta) \cdot x}\, d\xi\,d\eta
    $$
    where
\begin{equation}\label{eq:multiplier}
        M(\xi,\eta) = \frac{1}{4\pi^2}\left(\frac{1}{|\xi|^2} - \frac{2(\xi + \eta) \cdot \xi}{|\xi + \eta|^2 |\xi|^2}\right) \, .
    \end{equation}
\end{proposition}
\begin{proof}
    We write that
    \begin{align*}
         \pa_{ii} \left( f \ilr f + 2 \ilr \pa_j \left( f \pa_j \ilr f \right) \right) &= \Delta f \ilr f + 2 \pa_i f \pa_i \ilr f - f^2 \\
         &\qquad + 2 f f - 2 \pa_j f \pa_j \ilr f \\
         &= \Delta f \ilr f + f^2 \, . 
    \end{align*}
We note that since $f \in W^{k,1}(\R^3)$ for every $k \geq 0$, every function and product written in the above display is unambiguously defined, smooth, and integrable.  Next, we recall that, up to dimensional constants, the multiplier for the inverse (minus) Laplacian is $|\xi|^{-2}$ and the multiplier for differentiation $\pa_j$ is $i\xi_j$.  The formula for $M(\xi, \eta)$ then follows from the fact that the Fourier transform converts multiplication to convolution and direct computation.
\end{proof}

We now record some basic estimates for $B(f,g)$ based on the Hardy--Littlewood--Sobolev and H\"older inequalities.
\begin{lemma}[\textbf{HLS--H\"older bounds for $B(f,g)$}]\label{lem:HLS-Holder}
There exist dimensional constants such that the following estimates hold for $f,g \in \mathcal{S}(\R^3)$. 
\begin{enumerate}[(i)]
\item\label{i:hls:1} $\| g \ilr f \|_{L^1(\R^3)} \les \| f \|_{L^{\sfrac 65}(\R^3)} \| g \|_{L^{\sfrac 65}(\R^3)}$.
\item\label{i:hls:2} $\Vert \partial_j(-\Delta)^{-1}_{\R^3}(g\partial_j(-\Delta)^{-1}_{\R^3}f) \Vert_{L^2(\R^3)} \lesssim \Vert g \Vert_{L^{3}(\R^3)} \Vert f \Vert_{L^{\sfrac 65}(\R^3)}$.
\item\label{i:hls:2a} $\Vert \partial_j(-\Delta)^{-1}_{\R^3}(g\partial_j(-\Delta)^{-1}_{\R^3}f) \Vert_{L^{\frac{3(1+\gamma)}{2-\gamma}}(\R^3)}\lesssim  \Vert g \Vert_{L^{2+2\gamma}(\R^3)} \Vert f \Vert_{L^{\frac{6+6\gamma}{5+2\gamma}}(\R^3)}$ for $0 < \gamma \ll 1$.
\item\label{i:hls:3} $\| B(f,g) \|_{H^{-10}(\R^3)} \les \| f \|_{L^{\sfrac 65}(\R^3)} \| g \|_{W^{5,1}(\R^3)}$
\item\label{i:hls:4} $\| B(f,g) \|_{H^{-10}(\R^3)} \les \| g \|_{L^1(\R^3)} \| f \|_{W^{5,1}(\R^3)}$
\item\label{i:hls:6} $\Vert B(f,g) \Vert_{H^{-10}(\R^3)} \lesssim \Vert f \Vert_{L^{1+\gamma}(\R^3)} \Vert g \Vert_{W^{5,1}(\R^3)}$ for $0 < \gamma \ll 1$    \item\label{i:hls:7} $\Vert B(f,g) + B(g,f) \Vert_{H^{-10}(\R^3)} \lesssim \Vert f \Vert_{L^{\sfrac 65}(\R^3)} \Vert g \Vert_{L^{\sfrac 65}(\R^3)}$
\end{enumerate}
\end{lemma}
\begin{proof}
The first estimate follows from H\"older's inequality with $L^6(\R^3)$ and $L^{\sfrac 65}(\R^3)$ and the embedding $\| \ilr f \|_{L^{6}(\R^3)} \les \| f \|_{L^{\sfrac 65}(\R^3)}$.  The second estimate follows from writing
$$ \Vert \partial_j(-\Delta)^{-1}_{\R^3}(g\partial_j(-\Delta)^{-1}_{\R^3}f) \Vert_{L^2} \lesssim \Vert g\partial_j(-\Delta)^{-1}_{\R^3}f \Vert_{L^{\sfrac 65}} \lesssim \Vert g \Vert_{L^{3}} \Vert \partial_j(-\Delta)^{-1}_{\R^3} f \Vert_{L^2} \lesssim \Vert g \Vert_{L^{3}} \Vert f \Vert_{L^{\sfrac 65}} \, , $$
where in the second inequality we have used H\"older conjugates $\sfrac 25 + \sfrac 35=1$.
The third estimate follows from writing
\begin{equation*}
\begin{aligned}
    \norm{\partial_j(-\Delta)^{-1}_{\R^3}(g\partial_j(-\Delta)^{-1}_{\R^3}f)}_{L^{\frac{3(1+\gamma)}{2-\gamma}}} &\lesssim \Vert g\partial_j(-\Delta)^{-1}_{\R^3}f \Vert_{L^{1+\gamma}}\\
        &\lesssim \Vert g \Vert_{L^{2+2\gamma}} \Vert \partial_j(-\Delta)^{-1}_{\R^3} f \Vert_{L^{2+2\gamma}} \lesssim \Vert g \Vert_{L^{2+2\gamma}} \Vert f \Vert_{L^{\frac{6+6\gamma}{5+2\gamma}}}.
\end{aligned}
\end{equation*}
The fourth estimate is an immediate consequence of the first two and of Sobolev embedding.  To prove the fifth, we bound
$$ \| g \ilr f \|_{L^1(\R^3)} \les \| g \|_{L^1(\R^3)} \| \ilr f \|_{L^\infty(\R^3)} \les \| g \|_{L^1(\R^3)} \| f \|_{W^{5,1}(\R^3)} \, . $$
and
\begin{align*}
 \| \ilr \pa_j \left( g \pa_j \ilr f \right) \|_{H^{-10}(\R^3)} &\les \|  g \pa_j \ilr f \|_{L^1(\R^3)}  \\
 &\les \| g \|_{L^1(\R^3)} \| f \|_{W^{5,1}(\R^3)} \, .
\end{align*}
The sixth estimate follows similarly. For the seventh inequality, we utilize the identity
\begin{equation}\label{eq:B_sym_id}
    B(f,g) + B(g,f) = - \Delta \left((-\Delta)^{-1}_{\R^3} f (-\Delta)^{-1}_{\R^3} g\right) - 4 \mathcal{R}_j \mathcal{R}_k \left( \partial_j(-\Delta)^{-1}_{\R^3}f \partial_k(-\Delta)^{-1}_{\R^3}g\right)
\end{equation}
where $\mathcal{R}_i$ denotes the $i^{th}$ Riesz transform. The first ingredient in the proof of~\eqref{eq:B_sym_id} is the identity
\begin{align*}
-\pa_k &\left( \pa_k \ilr f \, \pa_j \ilr g + \pa_j \ilr f \, \pa_k \ilr g 
    - \delta_{kj}\, \pa_m \ilr f \, \pa_m \ilr g \right) \\
&= - \Delta \ilr f \, \pa_j \ilr g - \pa_k \ilr f \, \pa_k \pa_j \ilr g 
   - \pa_j \pa_k \ilr f \, \pa_k \ilr g  \\
&\qquad - \pa_j \ilr f \, \Delta \ilr g + \pa_j \pa_m \ilr f \, \pa_m \ilr g + \pa_m \ilr f \, \pa_j \pa_m \ilr g \\
&= f \pa_j \ilr g + g \pa_j \ilr f \, ,
\end{align*}
where in the last equality we have used that the second and sixth terms cancel, as do the third and fifth.  Then to prove~\eqref{eq:B_sym_id}, we utilize the identity just proven above to write that
\begin{align*}
B&(f,g) + B(g,f) \\ 
&= g \ilr f + f \ilr g + 2\, \ilr \pa_j \left( g\, \pa_j \ilr f + f\, \pa_j \ilr g \right) \\
&= g \ilr f + f \ilr g \\
&\qquad - 2\, \ilr \pa_j \pa_k \left( \pa_j \ilr f \, \pa_k \ilr g 
    + \pa_k \ilr f \, \pa_j \ilr g 
    - \delta_{jk}\, \nabla \ilr f \cdot \nabla \ilr g \right) \\
&= g \ilr f + f \ilr g 
    - 4\, \mathcal{R}_j \mathcal{R}_k \left( \pa_j \ilr f \, \pa_k \ilr g \right) 
    + 2\, \ilr \Delta \left( \nabla \ilr f \cdot \nabla \ilr g \right) \\
&= g \ilr f + f \ilr g - 2\, \nabla \ilr f \cdot \nabla \ilr g 
    - 4\, \mathcal{R}_j \mathcal{R}_k \left( \pa_j \ilr f \, \pa_k \ilr g \right) \\
&= -\Delta \left( \ilr f \, \ilr g \right) 
    - 4\, \mathcal{R}_j \mathcal{R}_k \left( \pa_j \ilr f \, \pa_k \ilr g \right) .
\end{align*}
With this identity in hand, we start by analyzing the first term on the right-hand side of \eqref{eq:B_sym_id}. Using the product rule, the Hardy--Littlewood--Sobolev inequality, and the H\"{o}lder inequality we see
\begin{equation*}
    \begin{split}
        \Vert \nabla \left((-\Delta)^{-1}_{\R^3}f (-\Delta)^{-1}_{\R^3} g\right) \Vert_{L^{\sfrac 32}(\R^3)} &\lesssim \Vert (-\Delta)^{-1}_{\R^3} f \Vert_{L^6(\R^3)} \Vert \nabla (-\Delta)^{-1}_{\R^3} g \Vert_{L^2(\R^3)} \\
        &\qquad + \Vert (-\Delta)^{-1}_{\R^3} g \Vert_{L^6(\R^3)} \Vert \nabla (-\Delta)^{-1}_{\R^3} f \Vert_{L^2(\R^3)}\\
        &\lesssim \Vert f \Vert_{L^{\sfrac 65}(\R^3)} \Vert g \Vert_{L^{\sfrac 65}(\R^3)}\, .
    \end{split}
\end{equation*}
Since $L^{\sfrac 32}(\R^3) \subset L^1(\R^3) + L^2(\R^3) \subset H^{-8}(\R^3)$,
\begin{equation}\label{eq:B_sym_id_ineq_1}
    \Vert \Delta \left((-\Delta)^{-1}_{\R^3}f (-\Delta)^{-1}_{\R^3} g\right) \Vert_{H^{-10}(\R^3)} \lesssim \Vert \nabla \left((-\Delta)^{-1}_{\R^3}f (-\Delta)^{-1}_{\R^3} g\right) \Vert_{H^{-8}(\R^3)} \lesssim \Vert f \Vert_{L^{\sfrac 65}(\R^3)} \Vert g \Vert_{L^{\sfrac 65}(\R^3)}\, .
\end{equation}
Now analyzing the second term on the right-hand side of \eqref{eq:B_sym_id}, again applying the Hardy--Littlewood--Sobolev inequality and the H\"{o}lder inequality we have that
\begin{equation*}
    \Vert \partial_j(-\Delta)^{-1}_{\R^3}f \partial_k(-\Delta)^{-1}_{\R^3}g \Vert_{L^1(\R^3)} \lesssim \Vert \partial_j(-\Delta)^{-1}_{\R^3}f \Vert_{L^2(\R^3)} \Vert \partial_j(-\Delta)^{-1}_{\R^3}g \Vert_{L^2(\R^3)} \lesssim \Vert f \Vert_{L^{\sfrac 65}(\R^3)} \Vert g \Vert_{L^{\sfrac 65}(\R^3)}\, .
\end{equation*}
Since the Riesz transforms are bounded from $L^1(\R^3)$ to $H^{-10}(\R^3)$,
\begin{equation}\label{eq:B_sym_id_ineq_2}
    \Vert \mathcal{R}_j \mathcal{R}_k\left( \partial_j(-\Delta)^{-1}_{\R^3}f \partial_k(-\Delta)^{-1}_{\R^3}g \right) \Vert_{H^{-10}(\R^3)} \lesssim \Vert f \Vert_{L^{\sfrac 65}(\R^3)} \Vert g \Vert_{L^{\sfrac 65}(\R^3)}\, .
\end{equation}
Combining \eqref{eq:B_sym_id}, \eqref{eq:B_sym_id_ineq_1}, and \eqref{eq:B_sym_id_ineq_2} completes the proof.
\end{proof}

\subsection{Product estimate}

\begin{lemma}[\textbf{Decoupling}]\label{lem:decoupling}
Suppose $p \in [1,\infty]$, $a \in W^{d+1,p}(\R^d)$ and $U \in L^p(\T^d)$. Then for every $k \geq 0$ we have
\begin{equation}\label{eq:decoupling}
    \Vert aU  \Vert_{W^{k,p}(\R^d)} \lesssim_{d,k} \Vert a \Vert_{W^{k+d+1,p}(\R^d)} \Vert U  \Vert_{W^{k,p}(\T^d)}
\end{equation}
where the implicit constant depends only on the dimension $d$ and the number of derivatives $k$. When $k=0$ the implicit constant is universal.
\end{lemma}
\begin{proof}
Let us focus on the case when $1 \leq p < \infty$. When $p = \infty$ the modifications necessary are obvious. We will prove this using induction on $k$. So first assume $k = 0$. For any $j \in \Z^d$, let $\T^d_j$ denote the periodic box $[0,1]^d$ shifted by $j$.  Then using the Sobolev embedding $W^{d+1,p}(\T^d_j) \subset L^\infty(\T^d_j)$ (which holds with a constant uniform in $p$), we have that
\begin{align*}
    \| a U  \|_{L^p(\R^d)}^p &= \int_{\R^d} |a U |^p \\
    &= \sum_{j \in \Z^d} \int_{\T^d_j} |aU |^p \\
    &\les \sum_{j \in \Z^d} \int_{\T^d_j} \| a \|_{W^{d+1,p}(\T^d_j)}^p |U |^p\\
    &= \sum_{j \in \Z^d} \| a \|_{W^{d+1,p}(\T^d_j)}^p \| U  \|_{L^p(\T^d)}^p \\
    &= \| a \|_{W^{d+1,p}(\R^d)}^p \| U  \|_{L^p(\T^d)}^p \, .
\end{align*}
This proves the base case. Now we assume that \eqref{eq:decoupling} holds for some $k$. We compute
\begin{equation*}
    \begin{aligned}
        \Vert aU \Vert_{W^{k+1,p}(\R^d)} &= \Vert aU \Vert_{W^{k,p}(\R^d)} + \Vert aU \Vert_{\dot{W}^{k+1,p}(\R^d)}\\
        &\lesssim_{d,k} \Vert a \Vert_{W^{k+d+1,p}(\R^d)} \Vert U \Vert_{W^{k,p}(\T^d)} + \sum_{|\alpha|=1} \Vert \nabla^\alpha(aU) \Vert_{\dot{W}^{k,p}(\R^d)}\\
        &\lesssim_{d,k} \Vert a \Vert_{W^{k+d+1,p}(\R^d)} \Vert U \Vert_{W^{k,p}(\T^d)} + \sum_{|\alpha|=1} \left( \Vert \nabla^\alpha aU \Vert_{W^{k,p}(\R^d)} + \Vert a\nabla^\alpha U \Vert_{W^{k,p}(\R^d)}\right)\\
        &\lesssim_{d,k} \Vert a \Vert_{W^{k+d+1,p}(\R^d)} \Vert U \Vert_{W^{k,p}(\T^d)} +  \Vert a \Vert_{W^{k+d+2,p}(\R^d)} \Vert U \Vert_{W^{k,p}(\T^d)} + \Vert a \Vert_{W^{k+d+1,p}(\R^d)} \Vert U \Vert_{W^{k+1,p}(\T^d)}\\
        &\lesssim_{d,k} \Vert a \Vert_{W^{k+d+2,p}(\R^d)} \Vert U \Vert_{W^{k+1,p}(\T^d)}\, .
    \end{aligned}
\end{equation*}
This completes the induction.
\end{proof}

\subsection{Intermittent building blocks}

\begin{proposition}[\textbf{Construction of intermittent building blocks}]\label{prop:intermittent_funcs}
For $\Lambda \gg 1$ and $0 < \delta <1$ such that $\Lambda,\Lambda^\delta \in \N$, there is a smooth function $U_\Lambda: (\T/\Lambda^{\delta})^3 \to \R$ such that 
\begin{enumerate}[(i)]
    \item\label{i:block:1} $U_\Lambda \geq 0$ and $|\mathbb{P}_{=0}(U_\Lambda)| \lesssim \Lambda^{\frac{3\delta - 1}{2}}$, where the implicit constant is independent of $\Lambda$ and $\delta$;
    \item\label{i:block:3} $\mathbb{P}_{\not=0}(U_\Lambda)$ has frequency support outside the ball of radius $\Lambda^{\delta}$. We also have that for some implicit constants independent of $\Lambda$ and $\delta$,
\begin{equation}\label{eq:U_est}
    \left\| \nabla^\alpha U_\Lambda \right\|_{L^p(\T^3)} \lesssim \Lambda^{|\alpha|+1+3(1-\delta)\left(\frac{1}{2}-\frac{1}{p}\right)}\, , \qquad \, \,  p \geq 1 \, ,
        \end{equation}
\item\label{i:block:4} We have
\begin{equation}\label{eq:mean_V_lambda_square}
\int_{\T^3} (-\Delta)_{\T^3}^{-1}\nabla U_\Lambda \otimes (-\Delta)^{-1}_{\T^3}\nabla U_\Lambda = \frac{1}{2}\Id \, .
\end{equation}
\end{enumerate}
\end{proposition}
\begin{proof}
    We break the proof up into steps.
\bigskip

\noindent\texttt{Step 1: Construction of $U_\Lambda$ and checking item~\ref{i:block:1}. }
Fix $\phi:\R^3 \to \R$ smooth, compactly supported in $B(0,1)$, radially symmetric, $\phi \geq 0$, and in addition with the property that
\begin{equation}
\frac{1}{12\pi^2} \int_{\R^3} \frac{|\hat \phi(\xi)|^2}{|\xi|^2} \, d\xi = \frac 12 \, .    \label{eq:norm} 
\end{equation}
Note that since $\phi \in \mathcal{S}(\R^3)$, the integral on the left-hand side is finite. Then set
\begin{equation}
    \psi(x) \;=\; \left[\, 2\cdot\frac{1}{12\pi^2} \sum_{\xi \in \Z^3\setminus\{0\}}
    \frac{1}{|\Lambda^{\delta-1}\xi|^2}
    \left| \hat \phi(\Lambda^{\delta-1}\xi) \right|^2 \Lambda^{3(\delta-1)}
    \,\right]^{-\sfrac 12} \phi(x)
    \;=:\; C_{\Lambda, \delta}\, \phi(x) \, . \label{eq:norm:alt}
\end{equation}
To justify this choice, note that the quantity inside the brackets is twice a
Riemann sum with mesh size $\Lambda^{\delta-1}$ for the integral
from~\eqref{eq:norm}; in particular, for $\Lambda$ sufficiently large it is
bounded above and below by positive constants depending only on $\phi$, so
that $C_{\Lambda,\delta}$ is well-defined, and $C_{\Lambda,\delta}\rightarrow 1$
and $\psi \rightarrow \phi$ as $\Lambda^{\delta-1} \rightarrow 0$.  The
normalization is designed precisely so that
\begin{equation}
    \frac{1}{12\pi^2}\sum_{\xi \in \Z^3 \setminus \{0\}}
    \frac{\left|\hat\psi(\Lambda^{\delta-1}\xi)\right|^2}{|\Lambda^{\delta-1}\xi|^2}
    \,\Lambda^{3(\delta-1)} \;=\; \frac12
    \qquad\textnormal{for every } \Lambda \, ,
    \label{eq:norm:exact}
\end{equation}
which is the exact identity used in \texttt{Step 3} below.

Now let us define $U_\Lambda:\T^3 \to \R$ by
\begin{equation}\label{eq:U_def_fourier}
        U_\Lambda(x) = \sum_{\xi \in \Z^3} \Lambda^{\frac{3\delta-1}{2}} \hat{\psi}\left(\Lambda^{\delta-1}\xi\right) e^{2\pi i\Lambda^\delta \xi \cdot x} \, .
    \end{equation}
    From this representation, it is immediate to see that
     $$
    |\mathbb{P}_{=0}(U_\Lambda)| = \left|\Lambda^{\frac{3\delta-1}{2}} \hat{\psi}\left(0\right)\right| \lesssim_{\phi, C_{\Lambda,\delta}} \Lambda^{\frac{3\delta-1}{2}} \, .
    $$
    Since $\phi$ is fixed and $C_{\Lambda,\delta} \rightarrow 1$ as $\Lambda^{\delta-1} \rightarrow 0$, we may ignore the dependence of this constant on $\phi$ and $C_{\Lambda, \delta}$.  In the remainder of the proof, we will similarly omit the dependence of various constants on $\phi$ and $C_{\Lambda, \delta}$. Upon applying the Poisson summation formula to~\eqref{eq:U_def_fourier},\footnote{From the right-hand side of the equality, we can see that the period of $U_\Lambda$ is $\Lambda^{-\delta}$, as desired.  The change in powers of $\Lambda$ comes from properties of the Fourier transform under dilations.} we see that
    \begin{equation}\label{eq:U_phys}
        U_\Lambda(x) = \sum_{n \in \Z^3} \Lambda^{\frac{5-3\delta}{2}} \psi(\Lambda x + \Lambda^{1-\delta}n)
    \end{equation}
    and so it is clear that $U_\Lambda \geq 0$.  Thus the proof of item~\ref{i:block:1} is finished.

    \bigskip
\noindent\texttt{Step 2: Proof of item~\ref{i:block:3}. } From \eqref{eq:U_def_fourier} we see that
$$
\mathbb{P}_{\not=0}(U_\Lambda) = \sum_{\xi \in \Z^3 \setminus \{0\}} \Lambda^{\frac{3\delta-1}{2}} \hat{\psi}\left(\Lambda^{\delta-1}\xi\right) e^{2\pi i\Lambda^\delta \xi \cdot x}
$$
and thus the frequency support of $\mathbb{P}_{\not=0}(U_\Lambda)$ is contained in $\{\xi \in \Z^3 : |\xi| \geq \Lambda^\delta\}$.

From the compact support of $\psi$ and taking $\Lambda$ sufficiently large, using~\eqref{eq:U_phys} it follows that
\begin{equation}\label{eq:U_L^inf}
        \Vert U_\Lambda \Vert_{L^\infty(\T^3)} \lesssim \Lambda^{\frac{5-3\delta}{2}}.
    \end{equation}
    Examining~\eqref{eq:U_phys}, one can see that the support in $[0,1]^3$ is contained in a collection of identical (up to translations) $\approx\Lambda^{3\delta}$ cubes, each of which contains a ball of volume $\approx\Lambda^{-3}$. The total volume of this region is $\approx\Lambda^{3(\delta-1)}$, which yields the estimate
    \begin{equation}\label{eq:U_L^1}
        \Vert U_\Lambda \Vert_{L^1(\T^3)} \lesssim \Lambda^{3(\delta-1)} \Vert U_\Lambda \Vert_{L^\infty(\T^3)} \lesssim \Lambda^{\frac{3\delta-1}{2}}\, .
    \end{equation}
    Interpolating between the estimates provided by~\eqref{eq:U_L^inf} and~\eqref{eq:U_L^1} yields~\eqref{eq:U_est} with $\alpha = 0$. Examining~\eqref{eq:U_phys} shows that derivatives will cost a factor of $\Lambda$, which establishes~\eqref{eq:U_est} for all multi-indices $\alpha$.

    \noindent\texttt{Step 3: Proof of item~\ref{i:block:4}. } The results in this item will essentially follow from the radial symmetry of $\psi$, which imposes further invariances (beyond symmetry, which holds no matter whether $\psi$ is radial or not) on the matrix appearing on the left-hand side of~\eqref{eq:mean_V_lambda_square}. Let $R$ be any orthogonal matrix generated by permutations $(x_1, x_2, x_3) \rightarrow (x_{\sigma(1)}, x_{\sigma(2)}, x_{\sigma(3)})$ of the three coordinates and sign changes $(x_1, x_2, x_3) \rightarrow (\pm x_1, \pm x_2, \pm x_3)$ of the three coordinates.  Then by the radial symmetry of $\psi$, and thus the radial symmetry of $\hat \psi$,
$$
U_\Lambda(Rx) = \sum_{\xi \in \Z^3 \setminus\{0\}} \hat{\psi}\left(\Lambda^{\delta-1}\xi\right) e^{2\pi i\Lambda^\delta (R^T\xi) \cdot x} = \sum_{\zeta \in \Z^3 \setminus\{0\}} \hat{\psi}\left(\Lambda^{\delta-1}R^{-T}\zeta\right) e^{2\pi i\Lambda^\delta \zeta \cdot x} = U_\Lambda(x) \, .
$$
Therefore using the invariance of the inverse Laplacian under orthogonal change of variables and the chain rule we have that
\begin{equation*}
    \begin{split}
        (-\Delta)^{-1}_{\T^3}\nabla (U_\Lambda(Rx)) = (-\Delta)^{-1}_{\T^3}\nabla U_\Lambda(x) &\implies R^T ((-\Delta)^{-1}_{\T^3}\nabla U_\Lambda(Rx)) = (-\Delta)^{-1}_{\T^3}\nabla U_\Lambda(x)\\
        &\implies ((-\Delta)^{-1} \nabla U_\Lambda)(Rx) = R ((-\Delta)^{-1} \nabla U_\Lambda(x)) \, ,
    \end{split}
\end{equation*}
and as a consequence,
\begin{equation*}
    \begin{split}
        \int_{\T^3} ((-\Delta)^{-1}_{\T^3}\nabla U_\Lambda \otimes (-\Delta)^{-1}_{\T^3}\nabla U_\Lambda)(x) \, dx &=  \int_{\T^3} ((-\Delta)^{-1}_{\T^3}\nabla U_\Lambda \otimes (-\Delta)^{-1}_{\T^3} \nabla U_\Lambda)(Rx) \, dx\\
        &= R \left[  \int_{\T^3} ((-\Delta)^{-1}_{\T^3} \nabla U_\Lambda \otimes (-\Delta)^{-1}_{\T^3} \nabla U_\Lambda)(x) \, dx  \right] R^T \, .
    \end{split}
\end{equation*}
Choosing $R = \textnormal{diag}(-1,1,1)$ and computing directly, we find that
$$ \int_{\T^3} (-\Delta)^{-1}_{\T^3}\pa_1 U_\Lambda(x) (-\Delta)^{-1}_{\T^3}\pa_3 U_\Lambda(x) = - \int_{\T^3} (-\Delta)^{-1}_{\T^3}\pa_1 U_\Lambda(x) (-\Delta)^{-1}_{\T^3}\pa_3 U_\Lambda(x)
$$
and
$$
\int_{\T^3} (-\Delta)^{-1}_{\T^3}\pa_1 U_\Lambda(x) (-\Delta)^{-1}_{\T^3}\pa_2 U_\Lambda(x) = - \int_{\T^3} (-\Delta)^{-1}_{\T^3}\pa_1 U_\Lambda(x) (-\Delta)^{-1}_{\T^3} \pa_2 U_\Lambda(x) \, .
$$
Repeating a similar argument with $R=\textnormal{diag}(1,-1,1)$, we find that $\int_{\T^3} (-\Delta)^{-1}_{\T^3}\nabla U_\Lambda \otimes (-\Delta)^{-1}_{\T^3} \nabla U_\Lambda$ must be a diagonal matrix.  Choosing instead $R$ to be the orthogonal transformation which switches $x_1$ and $x_2$, or $x_1$ and $x_3$, we see immediately that $\int_{\T^3} (-\Delta)^{-1}_{\T^3}\nabla U_\Lambda \otimes (-\Delta)^{-1}_{\T^3} \nabla U_\Lambda$ is a multiple of the identity.  To check which multiple of the identity, first we use Definition \ref{def:inv_Laplace} and \eqref{eq:U_def_fourier} to write
$$
(-\Delta)_{\T^3}^{-1}\nabla U_\Lambda = \sum_{\xi \in \Z^3 \setminus \{0\}} \Lambda^{\frac{3\delta-1}{2}} 2\pi i \Lambda^\delta \xi \frac{\hat{\psi}(\Lambda^{\delta-1}\xi)}{(2\pi \Lambda^{\delta} |\xi|)^2} e^{2\pi i \Lambda^\delta \xi \cdot x}
$$
and then we use the Plancherel theorem to compute that
\begin{equation}
    \begin{split}
        \int_{\T^3} (-\Delta)^{-1}_{\T^3}\nabla U_\Lambda \otimes (-\Delta)^{-1}_{\T^3} \nabla U_\Lambda &= \sum_{\xi \in \Z^3 \setminus \{0\}} (2\pi i \Lambda^\delta \xi) \otimes (-2\pi i \Lambda^\delta \xi) \left|\frac{1}{4\pi^2}\Lambda^{-\frac{1}{2}(1+\delta)} \frac{\hat{\psi}\left(\Lambda^{\delta-1}\xi\right)}{|\xi|^2}\right|^2\\
            &= \sum_{\xi \in \Z^3 \setminus \{0\}} \frac{1}{4\pi^2} \frac{\xi \otimes \xi}{|\xi|^4} \left|\hat{\psi}\left(\Lambda^{\delta-1}\xi\right)\right|^2 \Lambda^{\delta-1}\\
            &= \frac{1}{4\pi^2} \sum_{\xi \in \Z^3 \setminus\{0\}}  \frac{\left(\Lambda^{\delta-1}\xi\right) \otimes \left(\Lambda^{\delta-1}\xi\right) }{|\Lambda^{\delta-1}\xi|^4} \left|\hat{\psi}\left(\Lambda^{\delta-1}\xi\right)\right|^2 \Lambda^{3(\delta-1)} \\
            &= \frac{1}{4\pi^2} \sum_{\xi \in \Z^3 \setminus \{0\}} \frac{\sfrac 13 \Id}{|\Lambda^{\delta-1}\xi|^2} \left| \hat \psi(\Lambda^{\delta-1}\xi) \right|^2 \Lambda^{3(\delta-1)} \\
            &= \frac 12 \Id \, ,
        \end{split} \label{eq:sum:identity}
    \end{equation}
where in the final step, we used~\eqref{eq:norm},~\eqref{eq:norm:alt}, and~\eqref{eq:norm:exact}.  This concludes the proof of~\eqref{eq:mean_V_lambda_square}.
\end{proof}

\subsection{Cancellation mechanism}

\begin{proposition}[\textbf{High--high--low frequency interactions}]\label{prop:hhl}
    Let $U_\Lambda$ be given from Proposition \ref{prop:intermittent_funcs} with $\Lambda$ and $\delta$ such that $\Lambda, \Lambda^\delta \in \mathbb{N}$, and such that $\delta^{-1} \in \N$ with  $\delta^{-1} \geq 6$.  Let $\rho\in \left( \mathcal{S}(\R^3) \cap C(\R^3;[0,\infty)) \right) \cup \{\delta_0\}$ be either Schwartz, nonnegative, and satisfy  $\int_{\R^3} \rho = 1$, or $\rho = \delta_0 \in \mathcal{S}'(\R^3)$.  Let $a \in \mathcal{S}(\R^3)$ be compactly supported in frequency in a ball of radius $\lambda$, and set $g = a (\rho \ast U_\Lambda)$ (interpreted as $g=a U_\Lambda$ if $\rho = \delta_0$). Assume that $\Lambda^{\frac{\delta}{2}}>2\lambda$. Then there exists a matrix $C_{ij} = C_{ij}(\rho, \Lambda, \delta)$ whose entries depend on but are bounded independently of $\rho$, $\delta$, and $\Lambda$, such that 
\begin{align}\label{eq:hhl}
     B(g,g) = C_{ij} \mathcal{R}_i \mathcal{R}_j \left[ a^2 \right] + E \, ,
    \end{align}
where $\mathcal{R}_i$ is the $i^{\rm th}$ Riesz transform on $\R^3$, and there exists a constant $C$ independent of $a$, $\Lambda$, and $\rho$ such that
\begin{equation}
    \| E \|_{H^{-100\delta^{-1}}(\R^3)} \leq C \Lambda^{\frac{5\delta-1}{2}} \Vert a \Vert_{W^{4,6/5}(\R^3)}^2  \, . \label{e:estimate}
\end{equation}
In the case that $\rho = \delta_0$, $C_{ij} = -\delta_{ij}$.
\end{proposition}
\begin{remark}[\textbf{The Riesz transforms in~\eqref{eq:hhl}}]\label{rem:reeses}
If $C_{ij} = -\delta_{ij}$, then since $\delta_{ij} \mathcal{R}_i \mathcal{R}_j =  \mathcal{R}_i \mathcal{R}_i = -\Id$, we see that $B(g,g) = a^2 + E$.  We however require the generality of convolving with a kernel $\rho$ due to our notion of mollifier-confluent solution. 
\end{remark}
\begin{remark}[\textbf{Choice of $a$}]
    In order to cancel an error $E_q$ at step $q+1$ in the iteration, we will set $a_{q+1}= \chi_{R_{q+1}} \left( 2\| E_q \|_{L^\infty(\R^3)} - E_q \right)^{\sfrac 12}$, where $\chi_{R_{q+1}}$ is a cutoff function at radius $R_{q+1}$ chosen sufficiently large and depending on $q$. Then $a_{q+1}^2 = \chi_{{R_{q+1}}}^2 \left( 2\| E_q \|_{L^\infty(\R^3)} - E_q \right)$, and we are able to cancel $E_q$, up to an error a with small gradient. 
\end{remark}
\begin{proof}
We break the proof into steps.  In \texttt{Step 1}, we isolate the leading order oscillatory term and estimate the lower order terms in $H^{-100\delta^{-1}}$.  Next, in \texttt{Step 2} we estimate the leading order oscillatory term, which we call $B({g_{\rm high}, g_{\rm high}})$.  As this is the most involved part of the proof, we break the proof up into further steps.  First, in \texttt{Step 2a}, we split $B(g_{\rm high},g_{\rm high})$ into the ``high--high--low'' and ``high--high--high'' frequency interactions.  In \texttt{Step 2b}, we begin to analyze the high--high--low portion, which we call $B_\Delta$, by providing asymptotics on the multiplier for this frequency regime and splitting $B_\Delta = B_\Delta^1 + B_\Delta^2$.  Next, in \texttt{Step 2c}, we estimate $B_\Delta^1$ and obtain the leading order term with $a^2$ in~\eqref{eq:hhl}.  Then in \texttt{Step 2d}, we estimate $B_\Delta^2$.  Finally, in \texttt{Step 2e}, we analyze the ``high--high--high'' term $B_{\not\Delta}$.

\bigskip
\noindent\texttt{Step 1: Lower order terms. }
First we write
    $$
    g = a\mathbb{P}_{=0}(\rho\ast U_\Lambda) + a \mathbb{P}_{\not=0}(\rho\ast U_\Lambda) := g_{\rm low} + g_{\rm high}\, .
    $$
From Proposition~\ref{prop:intermittent_funcs}, item~\ref{i:block:1}, we have that $|\bp_{=0} (\rho \ast U_\Lambda)| \les \Lambda^{\frac{3\delta-1}{2}}$, so we will show that the only meaningful contribution comes from $B(g_{\rm high}, g_{\rm high})$. First, from the embeddings $L^1(\R^3) \subset H^{-100\delta^{-1}}(\R^3)$ and $L^2(\R^3) \subset H^{-100\delta^{-1}}(\R^3)$, which hold since $\delta^{-1} \geq 6$,~\eqref{def:B:reform}, and Lemma~\ref{lem:HLS-Holder}, we have that
\begin{equation}\label{eq:B_low_low_est}
    \Vert B(g_{\rm low},g_{\rm low}) \Vert_{H^{-100\delta^{-1}}} \lesssim  \Lambda^{3\delta-1} \Vert B(a,a) \Vert_{H^{-100\delta^{-1}}} \les \Lambda^{3\delta-1} \left( \Vert a \Vert_{L^{6/5}}^2 + \Vert a \Vert_{L^{3}} \Vert a \Vert_{L^{6/5}} \right) \les \Lambda^{\frac{5\delta-1}{2}} \| a \|_{W^{4, 6/5}}^2 \, .
\end{equation}
Similarly we have that
    $$
    \Vert B(g_{\rm low},g_{\rm high}) \Vert_{H^{-100\delta^{-1}}} \lesssim \Lambda^{\frac{3\delta-1}{2}} \Vert B(a,a\rho\ast U_\Lambda)\Vert_{H^{-100\delta^{-1}}}\, .
    $$
    Now applying Lemma \ref{lem:HLS-Holder}, Lemma \ref{lem:decoupling}, and~\eqref{eq:U_est} with $p=\sfrac 65$, we see that
    $$
    \Vert a (-\Delta)^{-1}_{\R^3}(a\rho\ast U_\Lambda) \Vert_{L^1(\R^3)} \lesssim \Vert a \Vert_{L^{6/5}(\R^3)} \Vert a \rho\ast U_\Lambda \Vert_{L^{6/5}(\R^3)} \lesssim  \| a \|^2_{W^{4, 6/5}(\R^3)} \Vert U_\Lambda \Vert_{L^{6/5}(\T^3)} \lesssim \| a \|_{W^{4, 6/5}(\R^3)}^2 \Lambda^{\delta} \, .
    $$
    Applying similar reasoning to the second term, we have that
\begin{align*}
     \Vert \partial_j(-\Delta)^{-1}_{\R^3}(a\partial_j(-\Delta)^{-1}_{\R^3}(a\rho\ast U_\Lambda)) \Vert_{L^2(\R^3)} &\lesssim \Vert a\partial_j(-\Delta)^{-1}_{\R^3}(a\rho\ast U_\Lambda) \Vert_{L^{6/5}(\R^3)} \\
     &\lesssim \| a \|_{L^3(\R^3)} \| \pa_j (-\Delta)_{\R^3}^{-1}(a\rho\ast U_\Lambda) \|_{L^2(\R^3)} \\
     &\lesssim \Vert a \Vert_{L^3(\R^3)} \Vert a\rho\ast U_\Lambda \Vert_{L^{6/5}(\R^3)} \\
     &\lesssim \Vert a \Vert_{L^3(\R^3)} \Vert a \Vert_{W^{4,6/5}(\R^3)} \Lambda^\delta  \, ,
\end{align*}
    which after using that $\delta^{-1} \geq 6$ and Sobolev embedding shows that
\begin{equation}\label{eq:B_low_high_est}
        \Vert B(g_{\rm low},g_{\rm high})\Vert_{H^{-100\delta^{-1}}} \lesssim \Vert a \Vert_{W^{4,6/5}(\R^3)}^2 \Lambda^{\delta + \frac{3\delta-1}{2}} = \Lambda^{\frac{5\delta-1}{2}} \Vert a \Vert_{W^{4,6/5}(\R^3)}^2  \, .
    \end{equation}
Finally, $\Vert B(g_{\rm high},g_{\rm low}) \Vert_{H^{-100\delta^{-1}}}$ can be handled nearly identically; we just require Lemma~\ref{lem:decoupling} and the modification
\begin{align*}
    \Vert \partial_j(-\Delta)^{-1}_{\R^3}((a\rho\ast U_\Lambda)\partial_j(-\Delta)^{-1}_{\R^3}a) \Vert_{L^{\frac{3(1+\gamma)}{2-\gamma}}(\R^3)} 
    &\lesssim \|(a\rho\ast U_\Lambda)\partial_j(-\Delta)^{-1}_{\R^3}a\|_{L^{1+\gamma}(\R^3)} \les \| a \|_{W^{4, \sfrac 65}(\R^3)}^2 \Lambda^{1+\frac{3(1-\delta)(\gamma-1)}{2(1+\gamma)}} \, .
\end{align*}
Choosing $\gamma <1$ and combining this estimate with the factor of $\Lambda^{\frac{3\delta-1}{2}}$ from $g_{\rm low}$, we obtain an estimate consistent with~\eqref{e:estimate}.

\bigskip
\noindent\texttt{Step 2: Estimating $B(g_{\rm high},g_{\rm high})$. }
\noindent\texttt{Step 2a: decomposition of $B({g_{\rm high}}, g_{\rm high})$. } From~\eqref{eq:U_def_fourier}, the Fourier transform of $\bp_{\neq 0}(U_\Lambda)$ (interpreted in the sense of tempered distributions on $\R^3$) is supported on the lattice $\Lambda^\delta\Z^3 - \{0\}$.  Convolving $\bp_{\neq 0}(U_\Lambda)$ with $\rho$, we have that
$$ \bp_{\neq 0} \left( \rho \ast U_\Lambda(x) \right) = \sum_{k \neq 0} \Lambda^{\frac{3\delta-1}{2}} \hat \psi(\Lambda^{\delta-1}k) \hat\rho(\Lambda^\delta k) e^{2\pi i \Lambda^\delta k\cdot x} \, . $$
Then multiplying by $a$ to obtain $g_{\rm high}$, taking the Fourier transform, and using that the Fourier transform converts multiplication to convolution, we can express $\hat{g}_{\rm high}$ as
$$
\hat{g}_{\rm high}(\xi) = \sum_{k \not =0} \Lambda^{\frac{3\delta-1}{2}} \hat{\psi}(\Lambda^{\delta-1}k) \hat \rho(\Lambda^\delta k) \hat{a}(\xi - \Lambda^\delta k)  \, .
$$
Recalling~\eqref{eq:multiplier}, we then have that 
\begin{align*}
    B(g_{\rm high},g_{\rm high}) = \sum_{k_1,k_2 \not =0} \Lambda^{3\delta-1} &\hat{\psi}(\Lambda^{\delta-1}k_1) \hat \rho(\Lambda^\delta k_1) \hat{\psi}(\Lambda^{\delta-1}k_2) \hat \rho(\Lambda^\delta k_2) \\
    &\times\int_{\R^3} \int_{\R^3} M(\xi,\eta) \hat{a}(\xi - \Lambda^\delta k_1) \hat{a}(\eta - \Lambda^\delta k_2) e^{2\pi i (\xi + \eta) \cdot x}\, d\xi\, d\eta\, .
\end{align*}
    Let us identify the ``diagonal terms,'' or ``high--high--low'' terms, from the above double sum as
    \begin{equation}\label{eq:B_diag}
        B_{\Delta} = \sum_{k \not=0} \Lambda^{3\delta-1} \left|\hat{\psi}(\Lambda^{\delta-1}k) \hat \rho(\Lambda^{\delta} k) \right|^2 \int_{\R^3} \int_{\R^3} M(\xi,\eta) \hat{a}(\xi - \Lambda^\delta k) \hat{a}(\eta + \Lambda^\delta k) e^{2\pi i (\xi + \eta) \cdot x}\, d\xi\, d\eta
    \end{equation}
    and the ``off-diagonal terms,'' or ``high--high--high'' terms, as
    \begin{align}
        B_{\not \Delta} = \sum_{\substack{k_1 \not= -k_2 , \\k_1,k_2 \not =0}} \Lambda^{3\delta-1} &\hat{\psi}(\Lambda^{\delta-1}k_1) \hat\rho(\Lambda^\delta k_1) \hat{\psi}(\Lambda^{\delta-1}k_2) \hat \rho(\Lambda^\delta k_2) \notag \\
        &\times\int_{\R^3} \int_{\R^3} M(\xi,\eta) \hat{a}(\xi - \Lambda^\delta k_1) \hat{a}(\eta - \Lambda^\delta k_2) e^{2\pi i (\xi + \eta) \cdot x}\, d\xi\, d\eta \, . \label{eq:B_off_diag}
    \end{align}
    It is clear that $B(g_{\rm high},g_{\rm high}) = B_{\Delta} + B_{\not\Delta}$. 
\bigskip

\noindent\texttt{Step 2b: Estimating $B_\Delta$. } Analyzing~\eqref{eq:B_diag} and~\eqref{eq:multiplier}, it is clear that swapping the variables $\xi$ and $\eta$ leaves $B_\Delta$ unchanged; in other words, since we are evaluating $B(f,g)$ with $f=g= g_{\rm high}$, we may symmetrize the multiplier in $B_\Delta$ without changing the expression. Hence we introduce
\begin{align}
    M_{\rm sym}(\xi,\eta) &= \frac{1}{2}\left(M(\xi,\eta) + M(\eta,\xi)\right)  \notag \\
    &= \frac{1}{8\pi^2} \left( \frac{1}{|\xi|^2} + \frac{1}{|\eta|^2} - \frac{2(\xi+\eta)\cdot \xi |\eta|^2}{|\xi+\eta|^2|\xi|^2|\eta|^2} - \frac{2(\xi+\eta)\cdot \eta |\xi|^2}{|\xi+\eta|^2|\xi|^2|\eta|^2} \right)\notag \\
    &= \frac{1}{8\pi^2} \left( \frac{|\eta|^2 |\eta+\xi|^2 + |\xi|^2|\xi+\eta|^2}{|\xi+\eta|^2|\xi|^2|\eta|^2} - \frac{4|\eta|^2|\xi|^2 + 2\eta \cdot \xi (|\eta|^2 + |\xi|^2)}{|\xi+\eta|^2|\xi|^2|\eta|^2} \right) \notag \\
    &= \frac{1}{8\pi^2} \left( \frac{|\eta|^4 + |\xi|^4 - 2 |\eta|^2|\xi|^2}{|\xi+\eta|^2|\xi|^2|\eta|^2} \right) \notag \\
    &= \frac{(|\xi|^2 - |\eta|^2)^2}{8\pi^2|\xi|^2 |\eta|^2 |\xi + \eta|^2}  \, . \label{eq:sym_multiplier}
    \end{align}
    Now we plug the symmetrized multiplier into $B_\Delta$ and make the change of variables $\xi \mapsto \xi + \Lambda^\delta k$ and $\eta \mapsto \eta - \Lambda^\delta k$ in $B_\Delta$ to obtain (note that the phase in the exponential remains unchanged due to the cancellation $\Lambda^\delta k - \Lambda^\delta k=0$)
    $$
    B_\Delta = \sum_{k \not=0} \Lambda^{3\delta-1} \left|\hat{\psi}(\Lambda^{\delta-1}k) \hat \rho(\Lambda^\delta k)\right|^2 \int_{\R^3} \int_{\R^3} M_{\rm sym}(\xi + \Lambda^\delta k,\eta - \Lambda^\delta k) \hat{a}(\xi) \hat{a}(\eta) e^{2\pi i (\xi + \eta) \cdot x}\, d\xi\, d\eta\, .
    $$
    To proceed further we will require the following lemma, which heuristically says that the multiplier $M_{\rm sym}(\xi + \Lambda^\delta k,\eta - \Lambda^\delta k)$ has a much simpler structure on the support of $\hat{a}(\xi) \hat{a}(\eta)$.

\begin{lemma}\label{lem:remainder_est_take2}
    Let $M_{\rm sym}$ be as in~\eqref{eq:sym_multiplier}, let $k \in \mathbb Z^3 \setminus \{0\}$, and suppose that we have both $|\xi|,|\eta| \lesssim \Lambda^{\delta'}$ for some $\delta' < \delta$ and $\Lambda \gg 1$, where the size of $\Lambda$ depends on $\delta$ and $\delta'$. Then for $\xi+\eta \neq 0$,
\begin{equation}\label{eq:sym_multiplier_decomp_2}
        M_{\rm sym}(\xi + \Lambda^\delta k,\eta - \Lambda^\delta k)
        =
        \frac{((\xi + \eta) \cdot k)^2}{2\pi^2\Lambda^{2\delta} |\xi + \eta|^2 |k|^4}
        + O(\Lambda^{-3\delta} |k|^{-3} |\xi - \eta|)
        + O\!\left(\Lambda^{-4\delta}|k|^{-4}(|\xi|+|\eta|)^2\right)\, .
    \end{equation}
\end{lemma}
\begin{proof}
First note that
\[
M_{\rm sym}(\xi + \Lambda^\delta k, \eta - \Lambda^\delta k)
=
\frac{(|\xi + \Lambda^\delta k|^2 - |\eta - \Lambda^\delta k|^2)^2}
{8\pi^2|\xi + \Lambda^\delta k|^2 |\eta - \Lambda^\delta k|^2 |\xi + \eta|^2}\, .
\]
Expanding the numerator, we have
\begin{align*}
  (|\xi + \Lambda^\delta k|^2 - |\eta - \Lambda^\delta k|^2)^2
  &= \left( |\xi|^2 - |\eta|^2 + 2 \Lambda^\delta (\xi+\eta)\cdot k  \right)^2 \\
  &= \left((\xi+\eta)\cdot(\xi-\eta) + 2 \Lambda^\delta (\xi+\eta)\cdot k\right)^2 \\
  &=4\Lambda^{2\delta}((\xi + \eta) \cdot k)^2 + 4\Lambda^\delta ((\xi+\eta)\cdot(\xi-\eta))((\xi + \eta) \cdot k) \\
  &\qquad + ((\xi+\eta)\cdot(\xi-\eta))^2\, .
\end{align*}
Expanding the denominator, save for the common factor of $|\xi + \eta|^2$, gives
\begin{align*}
|\xi+\Lambda^\delta k|^2 |\eta-\Lambda^\delta k|^2 
&= \left( |\xi|^2 + 2\Lambda^\delta\xi \cdot k + \Lambda^{2\delta} |k|^2 \right)
   \left( |\eta|^2 - 2\Lambda^\delta\eta \cdot k + \Lambda^{2\delta} |k|^2 \right) \\
&= \Lambda^{4\delta}|k|^4 + 2\Lambda^{3\delta} |k|^2 (\xi - \eta) \cdot k  + \Lambda^{2\delta}\bigl(|k|^2|\xi|^2 + |k|^2|\eta|^2 - 4(\xi \cdot k)( \eta \cdot k)\bigr) \\
&\qquad + 2\Lambda^\delta \bigl(|\eta|^2 \xi \cdot k - |\xi|^2 \eta \cdot k\bigr) + |\xi|^2|\eta|^2\, .
\end{align*}
Combining the above computations, we have that
\begin{align*}
&M_{\rm sym}(\xi+\Lambda^\delta k, \eta - \Lambda^\delta k)
- \frac{\left( (\xi+\eta)\cdot k  \right)^2}
{2\pi^2 \Lambda^{2\delta}|\xi+\eta|^2|k|^4} \\
&=
\frac{
\left( |\xi+\Lambda^\delta k|^2 - |\eta-\Lambda^\delta k|^2 \right)^2
\Lambda^{2\delta}|k|^4
-4\left((\xi+\eta)\cdot k\right)^2
|\xi+\Lambda^\delta k|^2|\eta-\Lambda^\delta k|^2}
{8\pi^2|\xi+\Lambda^\delta k|^2|\eta-\Lambda^\delta k|^2|\xi+\eta|^2\Lambda^{2\delta}|k|^4} \\
&= \bigg{[} \left[ 4\Lambda^{2\delta}((\xi + \eta) \cdot k)^2 + 4\Lambda^\delta ((\xi+\eta)\cdot(\xi-\eta))((\xi + \eta) \cdot k) +((\xi+\eta)\cdot(\xi-\eta))^2 \right] \Lambda^{2\delta} |k|^{4} \\
&\qquad - 4 \left( (\eta+\xi)\cdot k \right)^2 \big{[} \Lambda^{4\delta}|k|^4 + 2\Lambda^{3\delta} |k|^2 (\xi - \eta) \cdot k  + \Lambda^{2\delta}\bigl(|k|^2|\xi|^2 + |k|^2|\eta|^2 - 4(\xi \cdot k)( \eta \cdot k)\bigr) \\
&\qquad \qquad  + 2\Lambda^\delta \bigl(|\eta|^2 \xi \cdot k - |\xi|^2 \eta \cdot k\bigr) + |\xi|^2|\eta|^2 \big{]} \bigg{]} \bigg{[} {8\pi^2|\xi+\Lambda^\delta k|^2|\eta-\Lambda^\delta k|^2|\xi+\eta|^2\Lambda^{2\delta}|k|^4} \bigg{]}^{-1}\, .
\end{align*}
Notice crucially that the leading order term in $\Lambda$, namely $4\Lambda^{4\delta}|k|^4 ((\xi+\eta)\cdot k)^2$, appears with both a plus and a minus sign and therefore cancels.
Thus the numerator in the last display is
\begin{align}
&4\Lambda^{3\delta}|k|^4
((\xi+\eta)\cdot(\xi-\eta))((\xi+\eta)\cdot k) 
-8\Lambda^{3\delta}|k|^2
\left((\xi+\eta)\cdot k\right)^2((\xi-\eta)\cdot k) \notag \\
&\qquad
+\Lambda^{2\delta}|k|^4((\xi+\eta)\cdot(\xi-\eta))^2 
-4\Lambda^{2\delta}\left((\xi+\eta)\cdot k\right)^2
\bigl(|k|^2|\xi|^2+|k|^2|\eta|^2-4(\xi\cdot k)(\eta\cdot k)\bigr) \notag \\
&\qquad\qquad
-8\Lambda^\delta\left((\xi+\eta)\cdot k\right)^2
\bigl(|\eta|^2\xi\cdot k-|\xi|^2\eta\cdot k\bigr) 
-4\left((\xi+\eta)\cdot k\right)^2|\xi|^2|\eta|^2 \, . \label{eq:numerator}
\end{align}
Since $|\xi|,|\eta|\lesssim \Lambda^{\delta'}$ with $\delta'<\delta$ and $k\in\mathbb Z^3\setminus\{0\}$, we may choose $\Lambda$ sufficiently large and depending on $\delta, \delta'$ so that
\[
    |\xi+\Lambda^\delta k| \simeq \Lambda^\delta |k|\, ,
    \qquad
    |\eta-\Lambda^\delta k| \simeq \Lambda^\delta |k|\, .
\]
In particular,
\begin{equation}    |\xi+\Lambda^\delta k|^2|\eta-\Lambda^\delta k|^2
    \simeq \Lambda^{4\delta}|k|^4 \, .  \label{eq:ksquare}
\end{equation}
Also, for $\Lambda$ sufficiently large,
\[
    |\xi|+|\eta| \lesssim \Lambda^\delta |k|\, .
\]
Then we can bound the two terms of order $\Lambda^{3\delta}$ by
\[
    C\Lambda^{3\delta}|\xi+\eta|^2|\xi-\eta||k|^5\, .
\]
The terms of order $\Lambda^{2\delta}$ are bounded by
\[
    C\Lambda^{2\delta}|\xi+\eta|^2(|\xi|+|\eta|)^2|k|^4.
\]
The term of order $\Lambda^\delta$ is bounded by
\[
    C\Lambda^\delta|\xi+\eta|^2(|\xi|+|\eta|)^3|k|^3\, ,
\]
and the term of order $1$ is bounded by
\[
    C|\xi+\eta|^2(|\xi|+|\eta|)^4|k|^2\, .
\]
Since $|\xi|+|\eta|\lesssim \Lambda^\delta |k|$, the terms of orders $\Lambda^\delta$ and $1$ are absorbed into
\[
    C\Lambda^{2\delta}|\xi+\eta|^2(|\xi|+|\eta|)^2|k|^4\, .
\]
Thus the numerator in the common--denominator expression~\eqref{eq:numerator} is bounded in absolute value by
\[
    C\Lambda^{3\delta}|\xi+\eta|^2|\xi-\eta||k|^5
    +
    C\Lambda^{2\delta}|\xi+\eta|^2(|\xi|+|\eta|)^2|k|^4\, .
\]
Using now~\eqref{eq:ksquare} to approximate the denominator, we find that
\begin{align*}
\left|
M_{\rm sym}(\xi+\Lambda^\delta k, \eta - \Lambda^\delta k)
- \frac{\left((\xi+\eta)\cdot k\right)^2}
{2\pi^2\Lambda^{2\delta}|\xi+\eta|^2|k|^4}
\right|
&\lesssim
\frac{
\Lambda^{3\delta}|\xi+\eta|^2|\xi-\eta||k|^5
}
{
\Lambda^{6\delta}|k|^8|\xi+\eta|^2
}
+
\frac{
\Lambda^{2\delta}|\xi+\eta|^2(|\xi|+|\eta|)^2|k|^4
}
{
\Lambda^{6\delta}|k|^8|\xi+\eta|^2
} \\
&\lesssim
\Lambda^{-3\delta}|k|^{-3}|\xi-\eta|
+
\Lambda^{-4\delta}|k|^{-4}(|\xi|+|\eta|)^2.
\end{align*}
This proves the claim.
\end{proof}
With Lemma \ref{lem:remainder_est_take2} in hand, we define
$$
R_{\Lambda,k}(\xi,\eta) = M_{\rm sym}(\xi + \Lambda^\delta k,\eta - \Lambda^\delta k) - \frac{((\xi + \eta) \cdot k)^2}{2\pi^2\Lambda^{2\delta} |\xi + \eta|^2 |k|^4}
    $$
and further decompose $B_\Delta$ into
    $$
    B_\Delta^1 = \sum_{k \not=0} \Lambda^{3\delta-1} \left|\hat{\psi}(\Lambda^{\delta-1}k) \hat \rho(\Lambda^\delta k) \right|^2 \int_{\R^3} \int_{\R^3} \frac{((\xi + \eta) \cdot k)^2}{2\pi^2\Lambda^{2\delta} |\xi + \eta|^2 |k|^4} \hat{a}(\xi) \hat{a}(\eta) e^{2\pi i (\xi + \eta) \cdot x}\, d\xi\, d\eta
    $$
    and
    $$
    B_\Delta^2 = \sum_{k \not=0} \Lambda^{3\delta-1} \left|\hat{\psi}(\Lambda^{\delta-1}k) \hat \rho(\Lambda^\delta k) \right|^2 \int_{\R^3} \int_{\R^3}  R_{\Lambda,k}(\xi,\eta) \hat{a}(\xi) \hat{a}(\eta) e^{2\pi i (\xi + \eta) \cdot x}\, d\xi\, d\eta\, .
    $$
It is clear that $B_\Delta = B_\Delta^1 + B_\Delta^2$.
\bigskip

\noindent\texttt{Step 2c: Analysis of $B_{\Delta}^1$. }  We have that
    \begin{align}\label{eq:B_Delta^1}
        B_\Delta^1 &= \frac{1}{2\pi^2}\int_{\R^3} \int_{\R^3} \left(\sum_{k \not=0} \frac{\left|\hat{\psi}(\Lambda^{\delta-1}k) \hat \rho(\Lambda^\delta k)\right|^2}{|\Lambda^{\delta-1}k|^2} \frac{((\xi + \eta) \cdot (\Lambda^{\delta-1}k))^2}{|\xi + \eta|^2 |\Lambda^{\delta-1}k|^2} \Lambda^{3(\delta-1)}\right) \hat{a}(\xi) \hat{a}(\eta) e^{2\pi i (\xi + \eta) \cdot x}\, d\xi\, d\eta \, \notag \\
        &= \int_{\R^3} \int_{\R^3} \left(\sum_{k \not=0} \frac{\left|\hat{\psi}(\Lambda^{\delta-1}k) \hat \rho(\Lambda^\delta k)\right|^2}{|\Lambda^{\delta-1}k|^2} \frac{\Lambda^{2\delta-2}k_\ell k_m}{|\Lambda^{\delta-1}k|^2} \frac{\Lambda^{3(\delta-1)}}{2\pi^2}\right) \frac{(\xi_\ell + \eta_\ell)(\xi_m+\eta_m)}{|\xi+\eta|^2} \hat{a}(\xi) \hat{a}(\eta) e^{2\pi i (\xi + \eta) \cdot x}\, d\xi\, d\eta \, .
    \end{align}
We now isolate and analyze the expression inside the parentheses in~\eqref{eq:B_Delta^1}. We now split into two cases, based on whether or not $\rho = \delta_0$.  In the case that $\rho = \delta_0$, note that $\hat \rho \equiv 1$.  Then from~\eqref{eq:sum:identity},
$$ \sum_{k \not=0} \frac{\left|\hat{\psi}(\Lambda^{\delta-1}k) \hat \rho(\Lambda^\delta k)\right|^2}{|\Lambda^{\delta-1}k|^2} \frac{\Lambda^{2\delta-2}k_\ell k_m}{|\Lambda^{\delta-1}k|^2} \frac{\Lambda^{3(\delta-1)}}{2\pi^2} = \Id \, . $$
After combining this identity with the observation that
$$ \int_{\R^3} \int_{\R^3} \frac{(\xi_\ell + \eta_\ell)(\xi_m+\eta_m)}{|\xi+\eta|^2} \hat{a}(\xi) \hat{a}(\eta) e^{2\pi i (\xi + \eta) \cdot x}\, d\xi\, d\eta = -\mathcal{R}_\ell \mathcal{R}_m [a^2] \, , $$
this concludes the analysis in the case $\rho = \delta_0$. If however $\rho \neq \delta_0$, we claim that there exists $C_\psi$ such that
\begin{equation}\label{eq:Riemann_Sum}
  \left| \sum_{k \not=0} \frac{\left|\hat{\psi}(\Lambda^{\delta-1}k)\hat \rho(\Lambda^\delta k)\right|^2}{|\Lambda^{\delta-1}k|^2} \frac{\Lambda^{2\delta-2}k\otimes k}{|\Lambda^{\delta-1}k|^2 } \Lambda^{3(\delta-1)} \right| \leq C_\psi \, .
\end{equation}
To justify this, note that by the positivity of $\rho$ (and a short computation with the Fourier transform), $|\hat \rho(\Lambda^\delta k)| \leq \hat \rho(0) = 1$.  Therefore 
$$ \left| \sum_{k \not=0} \frac{\left|\hat{\psi}(\Lambda^{\delta-1}k)\hat \rho(\Lambda^\delta k)\right|^2}{|\Lambda^{\delta-1}k|^2} \frac{\Lambda^{2\delta-2}k\otimes k}{|\Lambda^{\delta-1}k|^2 } \Lambda^{3(\delta-1)} \right| \leq  \sum_{k \not=0} \frac{\left|\hat{\psi}(\Lambda^{\delta-1}k) \right|^2}{|\Lambda^{\delta-1}k|^2} \Lambda^{3(\delta-1)} \, . $$
We recognize this sum as a Riemann sum approximation with mesh size $\Lambda^{\delta -1}$ of
$$ \int_{\R^3} \frac{|\hat \psi(z)|^2}{|z|^2} \, dz \, , $$
which is finite since $\psi$ is Schwartz.  Since $\Lambda^{\delta-1} \leq 1$ from the assumption that $\Lambda, \Lambda^\delta \in \mathbb{N}$ and $\delta^{-1} \geq 6$, we have that these Riemann sum approximations are uniformly bounded in $\Lambda$ and $\delta$.  Therefore the expression inside parentheses in~\eqref{eq:B_Delta^1} is equal to a matrix $C_{\ell m}$ which depends on $\rho$, $\Lambda$, and $\delta$ but is bounded independently of $\rho$, $\Lambda$, and $\delta$, concluding the analysis in the case $\rho \neq \delta_0$.
\bigskip

\noindent\texttt{Step 2d: estimate for $B_\Delta^2$. } 
Taking the Fourier transform in $x$, we find that
\begin{align*}
\widehat{B^2_\Delta}(\zeta) &= \sum_{k \not=0} \Lambda^{3\delta-1} \left|\hat{\psi}(\Lambda^{\delta-1}k) \hat \rho(\Lambda^\delta k) \right|^2 \int_{\R^3} \int_{\R^3}  R_{\Lambda,k}(\xi,\eta) \hat{a}(\xi) \hat{a}(\eta) \underbrace{\int_{\R^3} e^{2\pi i (\xi + \eta - \zeta) \cdot x}\, dx}_{= \delta_0(\xi+\eta-\zeta)}  \, d\xi\, d\eta \\
&=\sum_{k\neq 0}
\Lambda^{3\delta-1}
\left|\widehat{\psi}(\Lambda^{\delta-1}k) \hat \rho(\Lambda^\delta k) \right|^2
\int_{\mathbb R^3}
R_{\Lambda,k}(\xi,\zeta-\xi)
\widehat a(\xi)\widehat a(\zeta-\xi)
\,d\xi  \, .
\end{align*}
Put $\delta' = \frac{\delta}{2}$ and define
\[
E_{\delta'}
=
\left\{(\xi,\eta)\in \mathbb R^3\times \mathbb R^3:
|\xi|\lesssim \Lambda^{\delta'},\ |\eta|\lesssim \Lambda^{\delta'}\right\},
\qquad
F_{\delta'}=(\mathbb R^3\times \mathbb R^3)\setminus E_{\delta'}\, .
\]
We split $\widehat{B^2_\Delta}$ according to whether the two input frequencies $(\xi,\zeta-\xi)$ lie in $E_{\delta'}$ or in $F_{\delta'}$. Namely, set
\[
\widehat{B^{2,1}_\Delta}(\zeta)
=
\sum_{k\neq 0}
\Lambda^{3\delta-1}
\left|\widehat{\psi}(\Lambda^{\delta-1}k) \hat \rho(\Lambda^\delta k) \right|^2
\int_{\mathbb R^3}
\chi_{E_{\delta'}}(\xi,\zeta-\xi)
R_{\Lambda,k}(\xi,\zeta-\xi)
\widehat a(\xi)\widehat a(\zeta-\xi)
\,d\xi
\]
and
\[
\widehat{B^{2,2}_\Delta}(\zeta)
=
\sum_{k\neq 0}
\Lambda^{3\delta-1}
\left|\widehat{\psi}(\Lambda^{\delta-1}k) \hat \rho(\Lambda^\delta k) \right|^2
\int_{\mathbb R^3}
\chi_{F_{\delta'}}(\xi,\zeta-\xi)
R_{\Lambda,k}(\xi,\zeta-\xi)
\widehat a(\xi)\widehat a(\zeta-\xi)
\,d\xi .
\]
Then
\[
\widehat{B^2_\Delta}
=
\widehat{B^{2,1}_\Delta}
+
\widehat{B^{2,2}_\Delta}\, .
\]
The characteristic functions provide frequency localization so that we can apply Lemma \ref{lem:remainder_est_take2}. From the triangle inequality and Plancherel theorem we have
    $$
    \Vert B_\Delta^2\Vert_{L^2(\R^3)} \lesssim \Vert \widehat{B_\Delta^{2,1}} \Vert_{L^2(\R^3)} + \Vert \widehat{B_\Delta^{2,2}}\Vert_{L^2(\R^3)}\, .
    $$
From Lemma \ref{lem:remainder_est_take2}, the assumption that $(\xi, \zeta-\xi) \in E_{\delta'}$, $k \neq 0$, and $\delta'<\delta$, we have that
\begin{align*}
|R_{\Lambda, k}(\xi, \zeta-\xi)| &= O(\Lambda^{-3\delta}|k|^{-3} |2\xi-\zeta| ) + O( \Lambda^{-4\delta} |k|^{-4} (|\xi| + |\zeta-\xi| )^2 ) \\
&\les \Lambda^{-3\delta} |k|^{-3}\Lambda^{\delta'} + \Lambda^{-4\delta} |k|^{-4} \Lambda^{2\delta'} \\
&\les \Lambda^{-3\delta} |k|^{-3}\Lambda^{\delta'}  \, . 
\end{align*}
Plugging this into $B_{\Delta}^{2,1}$ and using the Schwartz decay of $\hat \psi$, the compact support of $\hat{a}$, and the inequality $|\hat\rho(\Lambda^\delta k)|\leq 1$ observed earlier, we have
\begin{align*}
    \Vert \widehat{B_{\Delta}^{2,1}} \Vert_{L^2} &\lesssim \sum_{k \not=0} \Lambda^{3\delta-1} \left|\hat{\psi}(\Lambda^{\delta-1}k)\right|^2 \left(\int_{\R^3} \left(\int_{\R^3} \chi_{E_\delta'}(\xi,\zeta-\xi)  \left|R_{\Lambda,k}(\xi,\zeta - \xi)\right| \left|\hat{a}(\xi) \hat{a}(\zeta - \xi)\right|\, d\xi\right)^2\, d\zeta\right)^{\sfrac 12}\\
    &\lesssim \sum_{k \not=0} \Lambda^{3\delta-1} \left|\hat{\psi}(\Lambda^{\delta-1}k)\right|^2 \left(\int_{\R^3} \left(\int_{\R^3} \chi_{E_\delta'}(\xi,\zeta-\xi) \chi_{\{|\zeta| \les 2 \Lambda^{\delta'}\}}(\zeta)  \left|R_{\Lambda,k}(\xi,\zeta - \xi)\right| \left|\hat{a}(\xi) \hat{a}(\zeta - \xi)\right|\, d\xi\right)^2\, d\zeta\right)^{\sfrac 12}\\
    &\lesssim \sum_{k \not=0} \Lambda^{3\delta-1} \left|\hat{\psi}(\Lambda^{\delta-1}k)\right|^2 \left(\int_{|\zeta| \lesssim 2\Lambda^{\delta'}} \left(\int_{|\xi| \lesssim \Lambda^{\delta'}} \left|R_{\Lambda,k}(\xi,\zeta - \xi)\right| \left|\hat{a}(\xi) \hat{a}(\zeta - \xi)\right|\, d\xi\right)^2\, d\zeta\right)^{1/2}\\
    &\lesssim \sum_{k \not=0} \Lambda^{-1+\delta'} \frac{\left|\hat{\psi}(\Lambda^{\delta-1}k)\right|^2}{|k|^3} \left(\int_{\R^3} \left( \int_{\R^3}  \left|\hat{a}(\xi) \hat{a}(\zeta - \xi)\right|\, d\xi\right)^2\, d\zeta\right)^{1/2}\\
    &\lesssim \| |\hat a | \ast |\hat a| \|_{L^2(\R^3)} \Lambda^{-1+\delta'}\left( \sum_{1 \leq |k| \lesssim \Lambda^{1-\delta}} \frac{\left|\hat{\psi}(\Lambda^{\delta-1}k)\right|^2}{|k|^3} + \sum_{|k| \gtrsim \Lambda^{1-\delta}} \frac{\left|\hat{\psi}(\Lambda^{\delta-1}k)\right|^2}{|k|^3}\right)\\
    &\lesssim \| \hat a \|_{L^1(\R^3)} \| \hat a \|_{L^2(\R^3)} \Lambda^{-1+\delta'} \left( \log(\Lambda^{1-\delta}) + \sum_{|k| \gtrsim \Lambda^{1-\delta}} \frac{(\Lambda^{\delta-1}|k|)^{-\frac{1}{100}}}{|k|^3} \right) \\
    &\lesssim \| a \|_{H^2(\R^3)} \| a \|_{L^2(\R^3)} \Lambda^{-1+\delta'}\left(\log(\Lambda^{1-\delta}) + \Lambda^{\frac{1-\delta}{100}} \right)\\
    &\lesssim \| a \|_{W^{4, 6/5}(\R^3)}^2 \Lambda^{-\frac{99}{100}+\delta'} \, .
    \end{align*}
Since we chose $\delta' = \frac{\delta}{2}$, this term obeys a bound commensurate with~\eqref{e:estimate}.

Now for $\widehat{B_\Delta^{2,2}}$, we note that 
$$ \chi_{F_{\delta'}}(\xi, \zeta-\xi) \leq \chi_{\{|\xi|\gtrsim \Lambda^{\delta'}\}} + \chi_{\{|\zeta - \xi|\gtrsim \Lambda^{\delta'}\}} \, . $$
Plugging this in, we obtain
\begin{align}
\Vert \widehat{B_{\Delta}^{2,2}} \Vert_{L^2} &\lesssim \sum_{k \not=0} \Lambda^{3\delta-1} \left|\hat{\psi}(\Lambda^{\delta-1}k)\right|^2 \left(\int_{\R^3} \left(  \int_{\R^3} \chi_{F_{\delta'}}(\xi,\zeta-\xi) \left|R_{\Lambda,k}(\xi,\zeta - \xi)\right| \left|\hat{a}(\xi) \hat{a}(\zeta - \xi)\right|\, d\xi\right)^2\, d\zeta\right)^{1/2} \nonumber \\
&\lesssim \sum_{k \not=0} \Lambda^{3\delta-1} \left|\hat{\psi}(\Lambda^{\delta-1}k)\right|^2 \left(\int_{\R^3} \left( \int_{|\xi| \gtrsim \Lambda^{\delta'}}  \left|R_{\Lambda,k}(\xi,\zeta - \xi)\right| \left|\hat{a}(\xi) \hat{a}(\zeta - \xi)\right|\, d\xi\right)^2\, d\zeta\right)^{1/2} \nonumber \\
    &+ \sum_{k \not=0} \Lambda^{3\delta-1} \left|\hat{\psi}(\Lambda^{\delta-1}k)\right|^2 \left(\int_{\R^3} \left( \int_{|\zeta - \xi| \gtrsim \Lambda^{\delta'}}  \left|R_{\Lambda,k}(\xi,\zeta - \xi)\right| \left|\hat{a}(\xi) \hat{a}(\zeta - \xi)\right|\, d\xi\right)^2\, d\zeta\right)^{1/2} \, . \label{eq:B_Delta^22}
\end{align}
Now upon using the assumption that $\Lambda^{\delta'}=\Lambda^{\frac{\delta}{2}} > \lambda$, which is the magnitude of the maximum frequency of $a$, we can force both integrals in \eqref{eq:B_Delta^22} to vanish, and thus
$$
B_{\Delta}^{2,2} = 0.
$$
Hence we have concluded the analysis of $B^2_\Delta$.
\bigskip

\noindent\texttt{Step 2e: Estimate of $B_{\not\Delta}$. } We now analyze $B_{\not\Delta}$, which we recall was defined in~\eqref{eq:B_off_diag}. Heuristically the support of $\hat{a}(\xi - \Lambda^\delta k_1)$ and $\hat{a}(\eta - \Lambda^\delta k_2)$ force $\xi$ and $\eta$ to be quite distant from each other, and so one would hope to have good estimates in $H^{-100\delta^{-1}}$. We make this heuristic precise. Let
    $$
    J(x) = \int_{\R^3} \int_{\R^3} M(\xi,\eta) \hat{a}(\xi - \Lambda^\delta k_1) \hat{a}(\eta - \Lambda^\delta k_2) e^{2\pi i (\xi + \eta) \cdot x}\, d\xi\, d\eta
    $$
    for $k_1 \not= - k_2$ and $k_1, k_2 \neq 0$. Taking the Fourier transform we have
    $$
    \hat{J}(\eta) = \int_{\R^3} M(\xi,\eta - \xi) \hat{a}(\xi - \Lambda^\delta k_1) \hat{a}(\eta - \xi - \Lambda^\delta k_2)\, d\xi 
    $$
    and thus
    $$
    \Vert J \Vert_{H^{-100\delta^{-1}}(\R^3)} \lesssim \left(\int_{\R^3} (1 + |\eta|^2)^{-100\delta^{-1}} \left(\int_{\R^3} |M(\xi,\eta-\xi)| |\hat{a}(\xi - \Lambda^\delta k_1) \hat{a}(\eta - \xi - \Lambda^\delta k_2)|\, d\xi\right)^2\, d\eta\right)^{1/2}.
    $$
    Now make the change of variables $\xi \mapsto \xi + \Lambda^\delta k_1$ followed by $\eta \mapsto \eta + \Lambda^\delta (k_1 + k_2)$ to rewrite the right--hand side of the above as
    $$ \left(\int_{\R^3} (1 + |\eta + \Lambda^\delta(k_1+k_2)|^2)^{-100\delta^{-1}} \left(\int_{\R^3} |M(\xi + \Lambda^\delta k_1,\eta-\xi + \Lambda^\delta k_2)| |\hat{a}(\xi) \hat{a}(\eta - \xi)|\, d\xi\right)^2\, d\eta\right)^{1/2}.
    $$
    Using~\eqref{eq:multiplier} to deduce that
    $$
    |M(\xi + \Lambda^\delta k_1,\eta-\xi + \Lambda^\delta k_2)| \lesssim \frac{1}{|\xi + \Lambda^\delta k_1|^2} + \frac{1}{|\eta + \Lambda^\delta (k_1 + k_2)||\xi + \Lambda^\delta k_1|}
    $$
    one is naturally led to consider
    $$
    J_1 = \left(\int_{\R^3} (1 + |\eta + \Lambda^\delta(k_1+k_2)|^2)^{-100\delta^{-1}} \left(\int_{\R^3}  \frac{|\hat{a}(\xi) \hat{a}(\eta - \xi)|}{|\xi + \Lambda^\delta k_1|^{2}}\, d\xi\right)^2\, d\eta\right)^{1/2}
    $$
    and
    $$
    J_2 = \left(\int_{\R^3} (1 + |\eta + \Lambda^\delta(k_1+k_2)|^2)^{-100\delta^{-1}} \left(\int_{\R^3} \frac{|\hat{a}(\xi) \hat{a}(\eta - \xi)|}{|\xi + \Lambda^\delta k_1||\eta + \Lambda^\delta (k_1+k_2)|}\, d\xi\right)^2\, d\eta\right)^{1/2}\, .
    $$
    We start with $J_1$. Notice that by assumption on the frequency support of $\hat a$, we must have $|\xi|\leq \lambda$ and $|\eta - \xi|\leq \lambda$, and thus $|\eta| \leq 2\lambda$. Then using that $k_1 + k_2 \neq 0$ and using that $\Lambda^{\frac \delta 2}> 2 \lambda$, we have that
    $$
    J_1 =  \left(\int_{|\eta| \leq \frac{1}{2}\Lambda^\delta|k_1+k_2|} (1 + |\eta + \Lambda^\delta(k_1+k_2)|^2)^{-100\delta^{-1}} \left(\int_{\R^3} \frac{|\hat{a}(\xi) \hat{a}(\eta - \xi)|}{|\xi + \Lambda^\delta k_1|^{2}}\, d\xi\right)^2\, d\eta\right)^{1/2}.
    $$
As a consequence,
    \begin{equation}\label{eq:singular_term}
        (1 + |\eta + \Lambda^\delta(k_1+k_2)|^2)^{-100\delta^{-1}} \simeq \Lambda^{-200}|k_1+k_2|^{-200\delta^{-1}} \leq \Lambda^{-200}\, .
    \end{equation}
Then using Young's convolution inequality and the fact that $k_1\neq 0$ and $\Lambda^\delta > 2\lambda$ implies that $|\xi+\Lambda^\delta k_1| \approx |\Lambda^\delta k_1| \gg 1$, we can estimate
\begin{equation*}
    \begin{split}
    J_1 &\lesssim \Lambda^{-100} \left(\int_{\R^3} \left(\int_{\R^3}  \frac{|\hat{a}(\xi) \hat{a}(\eta - \xi)|}{|\xi + \Lambda^\delta k_1|^{2}}\, d\xi\right)^2\, d\eta\right)^{1/2} \\
    &\les \Lambda^{-100} \| |\hat a| \ast |\hat a| \|_{L^2(\R^3)} \\
    &\les \Lambda^{-100} \| a \|_{W^{4, 6/5}(\R^3)}^2 \, ,
        \end{split}
    \end{equation*}
where in the last line we have used the same inequality derived for $|\hat a| \ast |\hat a|$ derived earlier.

Turning our attention to $J_2$ and using the same methodology as for $J_1$, we have
    \begin{equation*}
    \begin{split}
    J_2 &= \left(\int_{|\eta| \leq \frac{1}{2}\Lambda^\delta |k_1+k_2|} (1 + |\eta + \Lambda^\delta(k_1+k_2)|^2)^{-100\delta^{-1}} \left(\int_{\R^3} \frac{|\hat{a}(\xi) \hat{a}(\eta - \xi)|}{|\xi + \Lambda^\delta k_1||\eta + \Lambda^\delta (k_1+k_2)|}\, d\xi\right)^2\, d\eta\right)^{1/2}\\
    &\lesssim \Lambda^{-100} \left(\int_{\R^3} \left( \int_{\R^3} {|\hat{a}(\xi) \hat{a}(\eta - \xi)|}\, d\xi\right)^2\, d\eta\right)^{1/2}\\
    &\les \Lambda^{-100} \| a \|_{W^{4, 6/5}(\R^3)}^2\, .
        \end{split}
    \end{equation*}
    Thus
    $$
    \Vert J \Vert_{H^{-100\delta^{-1}}(\R^3)} \lesssim \Lambda^{-100} \| a \|_{W^{4, 6/5}(\R^3)}^2 \, ,
    $$
    and so recalling the full expression for $B_{\not\Delta}$ from~\eqref{eq:B_off_diag}, we have
\begin{equation}\label{eq:B_notDelta_est}
    \begin{split}
    \Vert B_{\not\Delta} \Vert_{H^{-100\delta^{-1}}(\R^3)} &\lesssim \sum_{\substack{k_1 \not= -k_2 , \\k_1,k_2 \not =0}} \Lambda^{3\delta-1} \left|\hat{\psi}(\Lambda^{\delta-1}k_1) \hat{\psi}(\Lambda^{\delta-1}k_2) \right| \Lambda^{-100} \| a \|_{W^{4, 6/5}(\R^3)}^2 \\
    &\lesssim \sum_{\substack{k_1 \not= -k_2 , \\k_1,k_2 \not =0}} \Lambda^{6(\delta-1)} \Lambda^{{-100-3\delta + 5}} \left|\hat{\psi}(\Lambda^{\delta-1}k_1) \hat{\psi}(\Lambda^{\delta-1}k_2)\right| \| a \|_{W^{4, 6/5}(\R^3)}^2 \\
    &\lesssim \Lambda^{-50} \| a \|_{W^{4, 6/5}(\R^3)}^2 \int_{\R^3}\int_{\R^3} \left|\hat{\psi}(\xi) \hat{\psi}(\eta)\right|\, d\xi\, d\eta \, ,
        \end{split}
    \end{equation}
where the implicit constant depends only on $\psi$. 
\end{proof}

\section{Inductive set-up, iterative proposition, and proof of the main theorem}\label{ss:ind:r3}
We will need to measure the error $E_q$ in a combination of $H^{-s}(\R^3)$ and $\dot W^{1,p}(\R^3)$ norms.  By $\dot W^{1,p}(\R^3)$, we mean the seminorm $\| D f \|_{L^p(\R^3)}$ which is well-defined for $f \in W^{1,p}(\R^3)$ but does not distinguish between functions which differ by a constant.  We therefore set the following convention.

\begin{definition}[\textbf{Semi-norm for the error}]
    Let $E \in L^2(\R^3) \cap W^{k,\infty}(\R^3)$ for all $k >0$ and $\delta$ be as in Proposition \ref{prop:intermittent_funcs}. We denote
    \begin{align}
        \| E \|_{X} = \inf \bigg{\{} &\| E_1 \|_{H^{-100\delta^{-1}}} + \| E_2 \|_{\dot W^{1,4}}\, : \, \, E_1 + E_2 = E \, ,   E_1\in L^1_{\rm loc}(\R^3) \cap H^{-100\delta^{-1}}(\R^3) \, ,  E_2\in  W^{1,4}(\R^3) \,  \bigg{\}}\, . \notag 
    \end{align}
\end{definition}

We shall need the following basic lemma regarding sequences $\{E_q\}_{q \in \mathbb{N}}$ for which $\| E_q \|_X \rightarrow 0$ as $q \rightarrow \infty$. 
\begin{lemma}[\textbf{Properties of $\| \cdot \|_X$}]\label{lem:x:con}
    If $\{E_q\}_{q \in \mathbb{N}} \subset L^2(\R^3) \cap W^{k,\infty}(\R^3)$ for all $k \geq 0$ is such that $\| E_q \|_X \rightarrow 0$ as $q \rightarrow \infty$, then for all Schwartz functions $\phi$,
    \begin{equation}
        \left\langle \phi, \pa_{ii} E_q \right \rangle_{\mathcal{S}, \mathcal{S}'}= \int_{\R^3} \phi \pa_{ii} E_q \rightarrow 0 \, . 
    \end{equation}
\end{lemma}
\begin{proof}
    By the assumption that $E_q \in X$, there exist $E_{q,1} \in L^1_{\rm loc}(\R^3) \cap H^{-100\delta^{-1}}(\R^3)$ and $E_{q,2} \in W^{1,4}(\R^3)$ with $E_q = E_{q,1} + E_{q,2}$ and $\| E_{q,1}\|_{H^{-100\delta^{-1}}} + \| E_{q,2} \|_{\dot W^{1,4}} \rightarrow 0$.  In addition, $E_{q,1}, E_{q,2}\in \mathcal{S}'(\R^3)$ for each $q$.  Thus given $\phi \in \mathcal{S}(\R^3)$, we may write
    \begin{align*}
        \int_{\R^3} \phi \pa_{ii} E_q&=\left\langle \phi, \pa_{ii} E_q \right \rangle_{\mathcal{S}, \mathcal{S}'} \\
        &= \left\langle \phi, \pa_{ii} \left( E_{q,1} + E_{q,2} \right) \right \rangle_{\mathcal{S}, \mathcal{S}'} \\
        &= \left \langle \pa_{ii} \phi , E_{q,1} \right \rangle_{H^{100\delta^{-1}}, H^{-100\delta^{-1}}} - \left \langle \pa_i \phi , \pa_i E_{q,2} \right \rangle_{L^{4/3}, L^4} \, , 
    \end{align*}
which converges to zero as $q \rightarrow \infty$.
\end{proof}

We will also need to bound iterated Riesz transforms on $X$.
\begin{lemma}[\textbf{Boundedness of Riesz transforms on $X$}]\label{lem:Riesz_bdd}
    For all $1 \leq i,j \leq 3$ and all $f \in L^2(\R^3) \cap W^{k,\infty}(\R^3)$ for all $k \geq 0$, we have
    $$
    \Vert \mathcal{R}_i \mathcal{R}_j f \Vert_X \leq 3 \Vert f \Vert_X\, .
    $$
\end{lemma}
\begin{proof}
    Fix $f$ and consider an admissible decomposition $f = f_1 + f_2$. Utilizing that the operator norm of the Riesz transforms on $H^{-100\delta^{-1}}(\R^3)$ is $1$ as well as that the operator norm of $\mathcal{R}_i\mathcal{R}_j$ on $L^4(\R^3)$ is bounded by $3$ from~\cite[Theorem 4]{BW}, we have
    \begin{equation*}
        \begin{split}
            \Vert \mathcal{R}_i \mathcal{R}_j f \Vert_X &\leq \Vert \mathcal{R}_i\mathcal{R}_j f_1 \Vert_{H^{-100\delta^{-1}}(\R^3)} + \Vert \mathcal{R}_i \mathcal{R}_j f_2 \Vert_{\dot{W}^{1,4}(\R^3)}\\
            &\leq \Vert f_1 \Vert_{H^{-100\delta^{-1}}(\R^3)} + 3 \Vert f_2 \Vert_{\dot{W}^{1,4}(\R^3)}\\
            &\leq 3\left(\Vert f_1 \Vert_{H^{-100\delta^{-1}}(\R^3)} + \Vert f_2 \Vert_{\dot{W}^{1,4}(\R^3)}\right)\, .
        \end{split}
    \end{equation*}
    Now taking the infimum over all decompositions $f = f_1 + f_2$ gives the result.
\end{proof}

\subsection{Inductive hypotheses and iterative proposition}\label{ss:ind}

Fix $1 < p_0 < \sfrac 65$ and $\delta$ independent of $q$ such that
\begin{equation}
0 < \delta < \min\left(\frac{6-5p_0}{3(2-p_0)}, \frac{1}{5}\right) \quad \text{and} \quad \delta^{-1} \in \N\, .  \label{def:delta}
\end{equation}
Fix $\delta'$ independent of $q$ such that $\frac{1}{10}\delta < \delta' < \frac{1}{5}\delta$. We assume that there exists $C > 0$, independent of $q$, such that the following hold.
\begin{enumerate}[(i)]
\item\label{i:ind:1} $(f_q,E_q)$ are qualitatively smooth with $f_q \in \mathcal{S}(\R^3)$ and $f_q(x) \geq 0$ for all $x\in \R^3$, and $E_q \in L^2(\R^3) \cap W^{k,\infty}(\R^3)$ for all $k > 0$.  We assume that the pair solves
\begin{equation}\label{eq:rel}
     \pa_{ii}  B(f_q,f_q) = \pa_{ii} E_q \, .  
\end{equation}
in the sense that for all $\phi \in \mathcal{S}$,
\begin{equation}\label{eq:weak_rel}
    \int_{\R^3} \pa_{ii} \phi B(f_q, f_q) = \int_{\R^3} \pa_{ii} \phi E_q \, . 
\end{equation}
We verify these assumptions at level $q+1$ in Section \ref{pf:item1}.  
\item\label{i:ind:2}We assume that there exist three sequences of nonnegative integers: $\Lambda_0, \Lambda_1, \dots, \Lambda_q$; $\lambda_0, \lambda_1, \dots, \lambda_q$; and $R_0, R_1, \ldots, R_q$.  These will play a role in a number of the assumptions below.  First, $f_q$ satisfies
\begin{equation}
    \| f_q \|_{ L^{1}(\R^3)} \geq  C^{-1}(1 + 2^{-q})\, .  \label{ind:f:estimates:1} 
\end{equation}
Next, $f_q$ may be decomposed as $\displaystyle f_q = f_0 + \sum_{q'=0}^{q-1} \left(f_{q'+1} - f_{q'}\right) =: \sum_{q'=0}^{q} g_{q'}$, where $g_{q'}$ is qualitatively smooth and nonnegative for all $0 \leq q' \leq q$.  Furthermore, $g_{q'}$ has maximum ``effective frequency'' $\Lambda_{q'}$ in the sense that for all $0 \leq q' \leq q$ and all $k \leq 6$,
\begin{equation}
\Lambda_{q'}^{-k}\left( \| \nabla^k g_{q'} \|_{L^{1}(\R^3)} + \| \nabla^k g_{q'} \|_{L^{p_0}(\R^3)} \right) \leq C2^{-q'-11} \, . \label{ind:g:estimate} 
\end{equation}
and
\begin{equation}
    \left\Vert g_q \exp(\langle x\rangle^q) \right\Vert_{L^1(\R^3)} \leq 2^{-q}\, .  \label{ind:g:estimates:2} 
\end{equation}
We verify these at level $q+1$ in Section \ref{pf:item2}.
\item\label{i:ind:ugh} Furthermore, 
\begin{equation}\label{Def:wq}
g_{q'} = a_{q'} U_{q'} \, ,    
\end{equation}
where $a_{q'}\in \mathcal{S}$ is nonnegative and  compactly supported in space in a ball of radius $R_{q'}$.  Then we assume that $U_{q'}:(\T/\Lambda_{q'}^\delta)^3 \rightarrow [0,\infty)$ is as defined by $U_{q'} = U_{\Lambda_{q'}}$ as in Proposition~\ref{prop:intermittent_funcs} with $\Lambda = \Lambda_{q'}$.
Next, we assume that
\begin{equation}\label{eq:growth:bounds} R_{q'}^{R_{q'}} + \lambda_{q'}^{\lambda_{q'}} \leq \Lambda_q \quad \forall\,  0 \leq q' \leq q \, , \qquad \Lambda_{q'}^{\Lambda_{q'}} \leq \min(\lambda_{q'+1}, R_{q'+1}) \quad \forall\,  0 \leq q' \leq q-1 \, . 
\end{equation}
Finally, we assume that
\begin{equation}\label{ind:aq':bounds}
\| a_{q'} \|_{W^{200\delta^{-1},1}(\R^3)} \leq \lambda_{q'} \, , \qquad \left\| \mathcal{R}_i \mathcal{R}_j \left[ \left(\bp_{\leq \Lambda_{q'}^{\delta'}} (a_{q'}) \right)^2 \right] \right\|_{X} \leq  C2^{-q'} \, . 
\end{equation}
We verify these assumptions at level $q+1$ in Section \ref{pf:itemugh}.
\item\label{i:ind:3}
We assume that\footnote{We do not have that $E_q$ converges to $0$ in $L^\infty(\R^3)$.  This is visible in~\eqref{eq:error:growing}, which contains the term $\chi_{q+1}^2 2 \| E_q \|_{L^\infty(\R^3)} $.  The presence of this term is due to the fact that the cancellation mechanism can only cancel negatively signed errors, since the high--high--low cancellation mechanisms spits out the positive low--frequency term $a_{q+1}^2$. This is the reason $a_{q+1}$ is defined as it is in Definition~\ref{def:gq+1}.  As a consequence, the error term $\chi_{q+1}^2 2 \| E_q\|_{L^\infty(\R^3)}$ is only small after taking a derivative and choosing the cutoff $\chi_{q+1}$ to have very large radius, so that the derivative is small in $L^4(\R^3)$.}
\begin{equation}
    \| E_q \|_{X} \leq C2^{-q-10} \, . \label{ind:E} 
\end{equation}
We verify this at level $q+1$ in Section \ref{pf:item3}.
\end{enumerate}

The core of the argument is the following iterative proposition.

\begin{proposition}[\textbf{Iterative proposition}]\label{prop:ind:ks}
Let $C$, $p_0$, $\delta$, and $\delta'$ be given, and fix $q \in \mathbb{N}\cup \{0\}$.  Assume that the pair $(f_q, E_q)$ satisfies the inductive assumptions listed in items~\eqref{i:ind:1}--\eqref{i:ind:3} with the given $C>0$ and with $p_0$ and $\delta$ given. Then there exists $g_{q+1}=f_{q+1}-f_q$ and $E_{q+1}$ such that the pair $(f_{q+1}, E_{q+1})$ satisfies the inductive assumptions listed in items~\eqref{i:ind:1}--\eqref{i:ind:3} with $q$ replaced by $q+1$ and the same $C$, $p_0$, $\delta$, and $\delta'$.   
\end{proposition}

\subsection{Proof of Theorem~\ref{thm:main} from Proposition~\ref{prop:ind:ks}}\label{ss:limit}
We start by verifying the base case of the iterative proposition, and then we will apply it to prove the theorem. So let us put $R_0 = \lambda_0 = 1$ and $f_0(x) = Aa_0(x)U_0(x)$ for some $a_0 \in \mathcal{S}(\R^3) \setminus \{0\}$, $a_0 \geq 0$, supported in $B(0,1)$, $A > 0$ a constant whose size will be determined, and $U_0= U_{\Lambda_0}$ from Proposition~\ref{prop:intermittent_funcs} for $\Lambda_0 \geq 10$ sufficiently large so that Proposition~\ref{prop:intermittent_funcs} is valid. Define
$$
E_0 = B(f_0,f_0)\, .
$$
The smoothness of $f_0$ and $E_0$ is immediate. By construction $f_0 \geq 0$ and is Schwartz. We claim that $E_0 \in L^2(\R^3) \cap W^{k,\infty}(\R^3)$ for all $k \geq 0$. To see this, recall that
$$
B(f_0,f_0) = f_0(-\Delta)^{-1}_{\R^3}f_0 + 2\partial_j(-\Delta)^{-1}_{\R^3}(f_0 \partial_j(-\Delta)^{-1}_{\R^3}f_0)\, .
$$
Since $f_0 \in \mathcal{S}(\R^3)$, $f_0(-\Delta)^{-1}_{\R^3}f_0 \in \mathcal{S}(\R^3)$ and $f_0\partial_j(-\Delta)^{-1}_{\R^3}f_0 \in \mathcal{S}(\R^3)$ since they are the product of Schwartz functions with smooth functions with controlled decay. We claim that if $h \in \mathcal{S}(\R^3)$ then $\partial_j(-\Delta)^{-1}_{\R^3}h \in L^2(\R^3) \cap W^{k,\infty}(\R^3)$. To prove this, fix a multi-index $\alpha$. Then we have
$$
\Vert \nabla^\alpha \partial_j (-\Delta)^{-1}_{\R^3} h \Vert_{L^\infty(\R^3)} \lesssim \int_{\R^3} \left|\xi^\alpha \frac{\xi_j}{|\xi|^2} \hat{h}(\xi)\right|\, d\xi \lesssim \int_{\R^3} \frac{|\hat{h}(\xi)|}{|\xi|^{1-|\alpha|}}\, d\xi\, .
$$
When $|\alpha| \geq 1$, this integral is clearly finite since $h \in \mathcal{S}(\R^3)$. When $|\alpha| = 0$ then
$$
\int_{\R^3} \frac{|\hat{h}(\xi)|}{|\xi|}\, d\xi = \int_{|\xi| \leq 1} \frac{\Vert \hat{h} \Vert_{L^\infty(\R^3)}}{|\xi|}\, d\xi + \int_{|\xi| > 1} |\hat{h}(\xi)|\, d\xi \lesssim \Vert \hat{h} \Vert_{L^\infty(\R^3)} + \Vert \hat{h} \Vert_{L^1(\R^3)}\, .
$$
Similarly, utilizing the Plancherel theorem we have
$$
\Vert \partial_j (-\Delta)^{-1}_{\R^3} h \Vert_{L^2(\R^3)}^2 \simeq \int_{\R^3} \frac{|\hat{h}(\xi)|^2}{|\xi|^2}\, d\xi \lesssim \Vert \hat{h} \Vert_{L^\infty(\R^3)}^2 + \Vert \hat{h} \Vert_{L^2(\R^3)}^2\, .
$$
Putting $h = f_0\partial_j(-\Delta)^{-1}_{\R^3}f_0$ proves that $E_0 \in L^2(\R^3) \cap W^{k,\infty}(\R^3)$ for all $k$. Finally \eqref{eq:rel} and \eqref{eq:weak_rel} hold by construction. This verifies item~\ref{i:ind:1} at level $q=0$.

Now choose $C$ large enough and $A$ small enough such that \eqref{ind:f:estimates:1}, \eqref{ind:g:estimate} and \eqref{ind:g:estimates:2} hold for all $k \leq 6$ (and putting $f_0 = g_0$). This gives \ref{i:ind:2} at level $q=0$.

Next, \eqref{Def:wq} holds by construction. From our choices of $R_0$, $\lambda_0$, and $\Lambda_0$ we have that \eqref{eq:growth:bounds} hold. And finally by choosing $A$ small enough we may ensure that
$$
\Vert Aa_0 \Vert_{W^{200\delta^{-1},1}(\R^3)} \leq 1 \quad \text{and} \quad \Vert A^2\mathcal{R}_i\mathcal{R}_j a_0^2 \Vert_X \leq C \, .
$$
This gives \ref{i:ind:ugh} at level $q=0$. \ref{i:ind:3} similarly follows by choosing the constant $A$ small enough, or $C$ large enough. 

Now that the base case has been verified, we apply the iterative proposition to construct $f = \lim_{q \to \infty} f_q$. From \eqref{ind:f:estimates:1} and \eqref{ind:g:estimate}, we see that $f$ is not identically zero and $f \in L^{p_0}(\R^3,dx)$.  In addition, by the assumption that $f_q \geq 0$ and $f_q \rightarrow f$ in $L^1(\R^3)$, we have that $f \geq 0$.  In order to prove that $ f\in \bigcap_{m=0}^\infty L^1(\R^3, \exp(\langle x \rangle^m)\, dx)$, fix $m \geq 0$.  Then by non-negativity, compact support, smoothness of increments, and~\eqref{ind:g:estimates:2},
\begin{align*}
   \int_{\R^3} |f(x)| \exp(\langle x \rangle^m) \, dx &= \sum_{q'=0}^{m-1} \int_{\R^3} |g_{q'}(x)| \exp(\langle x \rangle^m) \, dx +  \sum_{q'=m}^{\infty} \int_{\R^3} |g_{q'}(x)| \exp(\langle x \rangle^m) \, dx  \\
   &\leq C(m) +  \sum_{q'=m}^{\infty} \int_{\R^3} |g_{q'}(x)| \exp(\langle x \rangle^{q'}) \, dx \\
   &\leq C(m) + 1 \, .
\end{align*}

To conclude the proof of the main theorem, we must show that in the sense of tempered distributions,
$$ \lim_{\epsilon \rightarrow 0} Q_{\rm KS} (f^\rho_\epsilon, f^\rho_\epsilon)= \lim_{\epsilon \rightarrow 0} \pa_{ii}B(f^\rho_\epsilon,f^\rho_\epsilon) = 0 $$
for $f^\rho_\epsilon = f \ast \rho_\epsilon$, for all admissible mollifiers $\rho$ (which are nonnegative, Schwartz, and satisfy  $\int_{\R^3} \rho = 1$).  Note that we have used Proposition~\ref{prop:reform} to rewrite $Q_{\rm KS} = \pa_{ii} B$, which is allowed since $f_\epsilon^\rho \in W^{k,1}(\R^3)$ for all $k \geq 0$. Since $\rho$ will be fixed throughout this argument, we omit the superscript to condense the notation. From Lemma~\ref{lem:x:con}, it will suffice to show that
\begin{equation}
\lim_{\epsilon \rightarrow 0} \left\| B(f_\epsilon,f_\epsilon) \right\|_{X} = 0 \, .   \label{eq:thm:to:show}
\end{equation}
For $\epsilon > 0$ fixed, define\footnote{Without loss of generality, we may assume $\epsilon$ is sufficiently small so that $q(\epsilon) \geq 3$, justifying the appearance of $q(\epsilon)-2$ at various points in the proof.} 
$$
q(\epsilon) = \min \left\{q : \Lambda_{q}^{\frac{\delta}{2}} > \epsilon^{-1}\right\}\, .
$$
We then break up the proof of~\eqref{eq:thm:to:show} into the following steps.  
\begin{enumerate}
\item In \texttt{Step 1}, we show that $\rho_\epsilon \ast g_{q'}$ can be made arbitrarily small when $q' \geq q(\epsilon)$ and $\epsilon$ is sufficiently small.
\item In \texttt{Step 2}, we apply \texttt{Step 1} to show that $$ \lim_{\epsilon \rightarrow 0} B(f_\epsilon, f_\epsilon) = \lim_{\epsilon \rightarrow 0} B\left(\sum_{q< q(\epsilon)} \rho_\epsilon \ast g_{q}, \sum_{q< q(\epsilon)} \rho_\epsilon \ast g_{q}\right) \, . $$
\item In \texttt{Step 3}, we show that for $q'\leq q(\epsilon) - 2$, $\rho_\epsilon \ast g_{q'} - g_{q'}$ can be made arbitrarily small.
\item In \texttt{Step 4}, we apply \texttt{Step 3} to show that 
$$ \lim_{\epsilon \rightarrow 0} B\left(\sum_{q< q(\epsilon)} \rho_\epsilon \ast g_{q}, \sum_{q'< q(\epsilon)} \rho_\epsilon \ast g_{q'}\right)= \lim_{\epsilon \rightarrow 0} B\left(\rho_\epsilon \ast g_{q(\epsilon)-1} , \rho_\epsilon \ast g_{q(\epsilon)-1}\right) \, . $$
\item In \texttt{Step 5}, we apply Proposition~\ref{prop:hhl} to show that $B(\rho_\epsilon \ast g_{q(\epsilon)-1}, \rho_\epsilon \ast g_{q(\epsilon)-1})=0$, thus concluding that $\lim_{\epsilon \rightarrow 0} B(f_\epsilon, f_\epsilon)=0$.
\end{enumerate}

\bigskip

\noindent\texttt{Step 1: Bounding tail terms in the mollification.}  In this step, we prove the following lemma.
\begin{lemma}[\textbf{Bounding the tail terms in the mollification}]\label{lem:tail_killing}
Let $\rho\in \mathcal{S}$ with $\int_{\R^3} \rho = 1$.  Then there exists $C_{\rho}$ and $\epsilon_0$ depending only on $\rho$ such that for all $0<\epsilon<\epsilon_0$ and $q' \geq q(\epsilon)$, 
$$
\Vert g_{q'} \ast \rho_\epsilon \Vert_{W^{5,1}(\R^3)} \leq C_{\rho} \Lambda_{q'}^{-\frac{1}{1000}} \, .
$$
\end{lemma}
\begin{proof}
We decompose the definition of $g_{q'}$ from~\eqref{Def:wq} into
   $$
   g_{q',0} = a_{q'} \mathbb{P}_{=0}(U_{q'}) \quad \text{and} \quad g_{q',1} = a_{q'} \mathbb{P}_{\not=0}(U_{q'}) \, .
   $$
Then from inductive assumption~\ref{i:ind:ugh},~\eqref{eq:U_est} with $p=1$, the parameter bounds in~\eqref{eq:growth:bounds}, and a sufficiently small choice of $\epsilon_0$, we have
\begin{equation*}
    \begin{split}
        \Vert g_{q',0} \ast \rho_\epsilon \Vert_{W^{5,1}(\R^3)} &= |\mathbb{P}_{=0}(U_{q'})| \Vert a_{q'} \ast \rho_\epsilon \Vert_{W^{5,1}(\R^3)}\\
    &\leq \lambda_{q'} \Vert \rho \Vert_{L^1(\R^3)} \Vert U_{q'} \Vert_{L^1(\T^3)} \\
    &\leq C_\rho \Lambda_{q'}^{-\frac{1}{100}} \, .
        \end{split}
    \end{equation*}
    Now for the second term, our strategy is to apply integration by parts many times. Writing out the convolution\footnote{The convolution makes sense, since it is a Schwartz function convolved with a function with bounded derivatives of all orders.} and integrating by parts with the Laplacian $m$ times for $m=10\delta^{-1}$, we obtain
\begin{align*}
    (g_{q',1} \ast \rho_\epsilon)(x) &=\int_{\R^3} \left(\rho_\epsilon(x-y) a_{q'}(y)\right) (-\Delta_y)^m (-\Delta_y)_{\T^3}^{-m} \mathbb{P}_{\not=0}(U_{q'})(y)\, dy \\
    &= \int_{\R^3} (-\Delta_y)^m\left(\rho_\epsilon(x-y) a_{q'}(y)\right) (-\Delta_y)_{\T^3}^{-m} \mathbb{P}_{\not=0}(U_{q'})(y)\, dy \, .
\end{align*}
So for $|\alpha| \leq 5$ we have
    $$
     \int_{\R^3} (-\Delta_y)^m\left( \nabla_x^\alpha \rho_\epsilon(x-y) a_{q'}(y)\right) (-\Delta_y)_{\T^3}^{-m} \mathbb{P}_{\not=0}(U_{q'})(y)\, dy \, .
    $$
    Writing $(-\Delta_y)^m = \nabla_y^M$ and applying the product rule we obtain
\begin{equation*}
    \begin{split}
    (-\Delta_y)^m(\nabla_x^\alpha\rho_\epsilon(x-y)a_{q'}(y)) &= \sum_{|\beta| \leq 2m} C_{\beta} \nabla_x^\alpha \nabla_y^{\beta} \rho_\epsilon(x-y) \nabla_{y}^{M-\beta} a_{q'}(y) \, ,
        \end{split}
    \end{equation*}
    and thus from the bound $\|a_{q'}\|_{W^{200\delta^{-1},1}(\R^3)} \leq \lambda_{q'}$, the assumption $q' \geq q(\epsilon)$, the choice of $m=10\delta^{-1}$, the parameter bounds in~\eqref{eq:growth:bounds}, and a sufficiently small choice of $\epsilon_0$,
\begin{align*}
    \Vert g_{q',1} \ast \rho_\epsilon \Vert_{W^{5,1}(\R^3)}&\leq \sum_{|\alpha| \leq 5, \, |\beta| \leq 2m} C_\beta \| \nabla^{|\alpha|+|\beta|} \rho_\epsilon \|_{L^1(\R^3)} \| \nabla^{M-|\beta|} a_{q'} \|_{L^1(\R^3)} \|(-\Delta_y)_{\T^3}^{-m} \bp_{\neq 0} (U_{q'}) \|_{L^\infty(\R^3)} \\
    &\leq C_{\rho} \epsilon^{-2m-5} \lambda_{q'}\Lambda_{q'}^{-2m\delta + 4} \| \bp_{\neq 0} (U_{q'}) \|_{L^1(\T^3)} \\
    &\leq C_\rho \Lambda_{q(\epsilon)}^{\frac{\delta}{2}(2m+5)} \lambda_{q'} \Lambda_{q'}^{-2m\delta+4} \\
    &\leq C_\rho \Lambda_{q'}^{-m\delta+4+ \frac{5\delta}{2}} \lambda_{q'} \\
    &\leq C_\rho \Lambda_{q'}^{-1} \, .
    \end{align*}
Adding this together with the bound for $g_{q',0}$ concludes the proof.
\end{proof}
\bigskip

\noindent\texttt{Step 2: Applying Step 1 to simplify $B(f_\epsilon, f_\epsilon)$. } We compute using \texttt{Step 1}, Lemma~\ref{lem:HLS-Holder}, items~\ref{i:hls:4} and~\ref{i:hls:6}, and~\eqref{ind:g:estimate} that
\begin{align*}
\sum_{\substack{q'  \geq q(\epsilon) \\ \textnormal{or } q''\geq q(\epsilon)}} \left\| B\left( g_{q'} \ast \rho_\epsilon , g_{q''} \ast \rho_\epsilon \right) \right\|_{H^{-10}(\R^3)} &\leq \sum_{\substack{q'  \geq q(\epsilon) \\ \textnormal{and } q''< q(\epsilon)}} \left\| B\left( g_{q'} \ast \rho_\epsilon , g_{q''} \ast \rho_\epsilon \right) \right\|_{H^{-10}(\R^3)} \\
&\qquad + \sum_{\substack{q'  \geq q(\epsilon) \\ \textnormal{and } q''< q(\epsilon)}} \left\| B\left( g_{q''} \ast \rho_\epsilon , g_{q'} \ast \rho_\epsilon \right) \right\|_{H^{-10}(\R^3)} \\
&\qquad \qquad + \sum_{\substack{q'  \geq q(\epsilon) \\ \textnormal{and } q''\geq q(\epsilon)}} \left\| B\left( g_{q'} \ast \rho_\epsilon , g_{q''} \ast \rho_\epsilon \right) \right\|_{H^{-10}(\R^3)} \\
&\leq 2C_\rho \sum_{\substack{q'  \geq q(\epsilon) \\ \textnormal{and } q''< q(\epsilon)}} \Lambda_{q'}^{-\frac{1}{1000}} \Lambda_{q''}^{1000} + \sum_{q', q'' \geq q(\epsilon)} \Lambda_{q'}^{-\frac{1}{1000}} \Lambda_{q''}^{-\frac{1}{1000}} \, .
\end{align*}
Using the growth bounds from~\eqref{eq:growth:bounds},
we observe that
$$
\sum_{\substack{q'\geq q(\epsilon) \\ \textnormal{or } q'' \geq q(\epsilon)}} B(g_{q'} \ast \rho_\epsilon, g_{q''} \ast \rho_\epsilon)
$$
is absolutely summable in $H^{-10}(\R^3)$. Upon sending $\epsilon \to 0$, we see that $q(\epsilon) \to \infty$ and thus
$$
\lim_{\epsilon \to 0} \sum_{\substack{q'\geq q(\epsilon) \\ \textnormal{or } q'' \geq q(\epsilon)}} B(g_{q'} \ast \rho_\epsilon, g_{q''} \ast \rho_\epsilon) = 0
$$
in $H^{-10}(\R^3)$, and hence in the sense of tempered distributions as well.

\bigskip
\noindent\texttt{Step 3: Mollification doesn't affect $g_{q'}$ if $q'\leq q(\epsilon)-2$}.
The next lemma asserts that from the growth assumptions on the sequence $\{\Lambda_q\}$ in inductive item~\ref{i:ind:ugh}, if $q \leq q(\epsilon)-2$, then $\Lambda_{q}$ is much less than $\epsilon^{-1}$. Since $\delta$ is fixed, for any $k$ fixed we may choose $\epsilon$ small enough so that
$$
\Lambda_q^k \ll \Lambda_{q(\epsilon)-1}^{\sfrac \delta 2} \leq \epsilon^{-1} \, ,
$$
leading to the following lemma.

\begin{lemma}[\textbf{Mollification doesn't affect $g_{q'}$ if $q'\leq q(\epsilon)-2$}]\label{lem:bulk_killing}
Let $\rho\in \mathcal{S}$ with $\int_{\R^3} \rho = 1$.  Then there exists $C_{\rho}$ and $\epsilon_0$ depending only on $\rho$ such that for all $0<\epsilon < \epsilon_0$ and $q' \leq q(\epsilon)-2$, 
$$ \Vert \rho_\epsilon \ast g_{q'} - g_{q'} \Vert_{W^{5,1}(\R^3)} \leq C_\rho \epsilon \Lambda_{q'}^{1000} \, .  $$
\end{lemma}
\begin{proof}
    First we write that
    $$
    \left(\rho_\epsilon \ast g_{q'} - g_{q'}\right)(x) = \int_{\R^3} \rho_{\epsilon}(y) \left(g_{q'}(x-y) - g_{q'}(x)\right)\, dy \, .
    $$
    Then for $|\alpha| \leq 5$ we have that
    $$
    \left\Vert \nabla^\alpha \left(\rho_\epsilon \ast g_{q'} - g_{q'}\right)\right\Vert_{L^1(\R^3)} \leq \int_{\R^3}\int_{\R^3} \left|\rho_{\epsilon}(y)\right| \left|\nabla^\alpha g_{q'}(x-y) - \nabla^\alpha g_{q'}(x)\right|\, dy\, dx \, .
    $$
    From the fundamental theorem of calculus we have that
    $$
    \left|\nabla^\alpha g_{q'}(x-y) - \nabla^\alpha g_{q'}(x)\right| \leq |y| \int_0^1 \left| \nabla \left(\nabla^{\alpha } g_{q'}(x-ty)\right)\right|\, dt \, .
    $$
    Hence from~\eqref{ind:g:estimate}, we have
    $$
    \left\Vert \nabla^\alpha \left(\rho_\epsilon \ast g_{q'} - g_{q'}\right)\right\Vert_{L^1(\R^3)} \leq \Vert \nabla\left( \nabla^{\alpha} g_{q'}\right) \Vert_{L^1(\R^3)} \int_{\R^3} |y| |\rho_\epsilon(y)|\, dy \leq  \epsilon \Lambda_{q'}^{|\alpha| + 1}\Lambda_{q'}^{\frac{3\delta-1}{2}} \lambda_{q'}^{|\alpha|+6} \Vert \cdot\rho(\cdot) \Vert_{L^1(\R^3)}.
    $$
    Summing over $|\alpha| \leq 5$, using the growth condition~\eqref{eq:growth:bounds}, and choosing $\epsilon_0$ sufficiently small completes the proof.
\end{proof}
\bigskip

\noindent\texttt{Step 4: Applying Step 3 to simplify $\lim_{\epsilon \rightarrow 0} B(\sum_{q< q(\epsilon)} \rho_\epsilon \ast g_{q}, \sum_{q'< q(\epsilon)} \rho_\epsilon \ast g_{q'})$. }  We split
$$ \sum_{q<q(\epsilon)} \rho_\epsilon \ast g_q = \sum_{q<q(\epsilon)-1} (\rho_\epsilon - \Id) g_q + \sum_{q<q(\epsilon)-1} g_q + \rho_\epsilon \ast g_{q(\epsilon)-1}  $$
and similarly for $q'$.  Plugging this into the bilinear form, we obtain
\begin{align*}
B&\left(\sum_{q< q(\epsilon)} \rho_\epsilon \ast g_{q}, \sum_{q'< q(\epsilon)} \rho_\epsilon \ast g_{q'}\right) \\
&= B\left( \sum_{q<q(\epsilon)-1} (\rho_\epsilon - \Id) g_q , \sum_{q'<q(\epsilon)-1} (\rho_\epsilon - \Id) g_{q'} \right) \\
    &\qquad +  B\left( \sum_{q<q(\epsilon)-1} g_q , \sum_{q'<q(\epsilon)-1} (\rho_\epsilon - \Id) g_{q'} \right) + B \left( \rho_\epsilon \ast g_{q(\epsilon)-1} , \sum_{q'<q(\epsilon)-1} (\rho_\epsilon - \Id) g_{q'} \right) \\
    &+ B\left( \sum_{q<q(\epsilon)-1} (\rho_\epsilon - \Id) g_q , \sum_{q'<q(\epsilon)-1} g_{q'} \right) \\
    &\qquad +  B\left( \sum_{q<q(\epsilon)-1} g_q , \sum_{q'<q(\epsilon)-1}  g_{q'} \right) + B \left( \rho_\epsilon \ast g_{q(\epsilon)-1} , \sum_{q'<q(\epsilon)-1}  g_{q'} \right) \\
    &+B\left( \sum_{q<q(\epsilon)-1} (\rho_\epsilon - \Id) g_q , \rho_\epsilon \ast g_{q(\epsilon)-1}\right) \\
    &\qquad +  B\left( \sum_{q<q(\epsilon)-1} g_q , \rho_\epsilon \ast g_{q(\epsilon)-1}\right) + B \left( \rho_\epsilon \ast g_{q(\epsilon)-1} , \rho_\epsilon \ast g_{q(\epsilon)-1} \right) \\
    &= I + II + III + IV + V + VI + VII + VIII + IX \, .
\end{align*}
The term $IX$ is the one term we will leave remaining in this step and will be analyzed in \texttt{Step 5}.  To analyze the remaining terms, we recall that by the definition of $q(\epsilon)$, $\epsilon^{-1} \geq \Lambda_{q(\epsilon)-1}^{\sfrac \delta 2}$.  Therefore by Lemma~\ref{lem:HLS-Holder}, Lemma~\ref{lem:bulk_killing}, and the growth conditions in~\eqref{eq:growth:bounds},
$$ \|  I \|_{H^{-10}(\R^3)} \leq \sum_{q,q'< q(\epsilon)-1} C_\rho^2 \epsilon^2 \Lambda_q^{1000} \Lambda_{q'}^{1000} < \sum_{q,q'< q(\epsilon)-1} C_\rho^2 \Lambda_{q(\epsilon)-1}^{-\delta} \Lambda_q^{1000} \Lambda_{q'}^{1000} \rightarrow 0 $$
as $\epsilon \rightarrow 0$.  Using instead now~\eqref{ind:g:estimate}, we similarly bound $II$ by
$$ \| II \|_{H^{-10}(\R^3)} \leq \sum_{q,q' < q(\epsilon)-1} C_\rho \epsilon \Lambda_{q'}^{1000} \Lambda_q^5 \rightarrow 0 $$
as $\epsilon \rightarrow 0$.  To bound $III$, we appeal to Lemma~\ref{lem:HLS-Holder}, item~\ref{i:hls:6}, Lemma~\ref{lem:decoupling}, and~\eqref{eq:U_est} to bound
$$ \| III \|_{H^{-10}(\R^3)} \les \| g_{q(\epsilon)-1} \|_{L^{1+\gamma}(\R^3)} \sum_{q'<q(\epsilon)-1} \| g_{q'} \|_{W^{5,1}(\R^3)} \les \Lambda_{q(\epsilon)-1}^{-\frac{1}{1000}} \sum_{q'<q(\epsilon)-1} \Lambda_{q'}^{5} \rightarrow 0 $$
as $\epsilon \rightarrow 0$, where as usual we have used the growth conditions~\ref{eq:growth:bounds}.  The term $IV$ can be bounded similarly to $II$, and we leave the details to the reader.  The term $V$ converges to zero by the inductive assumptions~\eqref{eq:rel},~\eqref{ind:E}, and Lemma~\ref{lem:x:con}.  The terms $VI$, $VII$, and $VIII$ we bound in a manner similar to $III$, and we again omit the details.  Finally, we leave $IX$ for the next step.  
\bigskip

\noindent\texttt{Step 5: Showing that $\lim_{\epsilon \rightarrow 0}B(\rho_\epsilon \ast g_{q(\epsilon)-1}, \rho_\epsilon \ast g_{q(\epsilon)-1})=0$. }  Using~\eqref{Def:wq}, we split
$$
\rho_\epsilon \ast g_{q(\epsilon)-1} = a_{q(\epsilon)-1} \rho_\epsilon \ast U_{q(\epsilon)-1} + \left[\rho_\epsilon \ast,a_{q(\epsilon)-1}\right]U_{q(\epsilon)-1} \, .
    $$
    The commutator term we expect to be able to make small, and thus the main contribution comes from the first term. Our bound for the commutator is as follows, and we prove this in \texttt{Step 5c} below. 
    \begin{lemma}[\textbf{Bound for the commutator}]\label{lem:com:bound}
    For any $\alpha \in \{0,1\}$, we have
\begin{equation}\label{eq:comm_est}
    \left\Vert \left[\rho_\epsilon \ast,a_{q(\epsilon)-1}\right]U_{q(\epsilon)-1} \right\Vert_{W^{\alpha,1}(\R^3)} \leq C_\rho R_{q(\epsilon)-1}^6 \lambda_{q(\epsilon)-1} \Lambda_{q(\epsilon)-1}^{\frac{3\delta-1}{2}} \max\left(\Lambda_{q(\epsilon)-1}^{\alpha-1},\epsilon^{1-\alpha}\right)  \, .
        \end{equation}
    \end{lemma}
\noindent Assuming for the moment that the lemma holds, we continue with \texttt{Step 5}.  Using the splitting above, we can write
\begin{align}
    B(\rho_\epsilon \ast g_{q(\epsilon)-1}, \rho_\epsilon \ast g_{q(\epsilon)-1}) &= B(a_{q(\epsilon)-1}\rho_\epsilon \ast U_{q(\epsilon)-1}, a_{q(\epsilon)-1}\rho_\epsilon \ast U_{q(\epsilon)-1}) \notag\\
    &\quad + B(a_{q(\epsilon)-1}\rho_\epsilon \ast U_{q(\epsilon)-1}, \left[\rho_\epsilon \ast,a_{q(\epsilon)-1}\right]U_{q(\epsilon)-1}) \notag \\
    &\qquad + B(\left[\rho_\epsilon \ast,a_{q(\epsilon)-1}\right]U_{q(\epsilon)-1}, a_{q(\epsilon)-1}\rho_\epsilon \ast U_{q(\epsilon)-1}) \notag \\
    &\qquad \quad + B(\left[\rho_\epsilon \ast,a_{q(\epsilon)-1}\right]U_{q(\epsilon)-1}, \left[\rho_\epsilon \ast,a_{q(\epsilon)-1}\right]U_{q(\epsilon)-1}) \notag \\
    &= I + II + III + IV \, .  \label{1234}
\end{align}
We estimate $II$, $III$, and $IV$ using \ref{i:hls:7} in \texttt{Step 5b}; in \texttt{Step 5a} we use Proposition~\ref{prop:hhl} to estimate $I$.
\bigskip

\noindent\texttt{Step 5a: the high--high--low term $I$. }  Recall we had fixed a parameter $\delta'$ in the inductive hypotheses such that $\frac{1}{10}\delta < \delta' < \frac{1}{5}\delta$ and now we split $I$ using
$$ a_{q(\epsilon)-1} = \bp_{\leq \Lambda_{q(\epsilon)-1}^{\delta'}} (a_{q(\epsilon)-1}) + \bp_{> \Lambda_{q(\epsilon)-1}^{\delta'}} (a_{q(\epsilon)-1})   = a_{\rm low} + a_{\rm high} $$
as
\begin{align}
    I &= B(a_{\rm low}\rho_\epsilon \ast U_{q(\epsilon)-1}, a_{\rm low}\rho_\epsilon \ast U_{q(\epsilon)-1}) + B(a_{\rm low}\rho_\epsilon \ast U_{q(\epsilon)-1}, a_{\rm high}\rho_\epsilon \ast U_{q(\epsilon)-1}) \notag \\
    &\qquad + B(a_{\rm high}\rho_\epsilon \ast U_{q(\epsilon)-1}, a_{\rm low}\rho_\epsilon \ast U_{q(\epsilon)-1}) + B(a_{\rm high}\rho_\epsilon \ast U_{q(\epsilon)-1}, a_{\rm high}\rho_\epsilon \ast U_{q(\epsilon)-1}) \, . \label{eq:a:limit} 
\end{align}
We apply the Proposition to the first term with the following choices:
\begin{align*}
    \Lambda = \Lambda_{q(\epsilon)-1} \, , \qquad  U_\Lambda = U_{q(\epsilon)-1} \, , \qquad \delta \textnormal{ as in~\eqref{def:delta}}\, , \qquad \rho \in \mathcal{S} \textnormal{ as in Proposition~\ref{prop:hhl}} \, , \qquad  \\
    a = a_{\rm low} \, , \qquad \lambda = \Lambda_{q(\epsilon)-1}^{\delta'} \, .
\end{align*}
From Proposition~\ref{prop:intermittent_funcs}, $\Lambda$, $\Lambda^\delta$ are natural numbers.  From~\eqref{def:delta}, $\delta^{-1} \in \mathbb{N}$ and $\delta^{-1} \geq 6$. From the choice of $\lambda$ and the definition of $a_{\rm{low}}$, the condition $\Lambda^{\frac{\delta}{2}} > 2\lambda$ from Proposition~\ref{prop:hhl} is satisfied, and $a_{\rm{low}}$ has compact frequency support in a ball of radius $\lambda = \Lambda_{q(\epsilon)-1}^{\delta'}$.  Then from~\eqref{eq:hhl} and Remark~\ref{rem:reeses}, we have that
$$
B(a_{\rm{low}}\rho_\epsilon \ast U_{q(\epsilon)-1},a_{\rm{low}} \rho_\epsilon \ast U_{q(\epsilon)-1}) = C_{ij} \mathcal{R}_i \mathcal{R}_j a_{\rm{low}}^2 + \tilde{E}_{q(\epsilon)-1} \, ,
$$
where from~\eqref{e:estimate} and Lemma~\ref{lem:aq+1_props},
\begin{equation}\label{eq:bilin_form_error:limit}
    \Vert \tilde{E}_{q(\epsilon)-1} \Vert_X \leq \Vert \tilde{E}_{q(\epsilon)-1} \Vert_{H^{-100\delta^{-1}}(\R^3)} \lesssim \Lambda_{q(\epsilon)-1}^{\frac{5\delta-1}{2}} \Vert a_{\rm{low}} \Vert_{W^{4,\sfrac 65}(\R^3)}^2 \lesssim \Lambda_{q(\epsilon)-1}^{\frac{5\delta-1}{2}} \Vert a_{q(\epsilon)-1} \Vert_{W^{4,\sfrac 65}}^2 \lesssim  \lambda_{q(\epsilon)-1}^2 \Lambda_{q(\epsilon)-1}^{\frac{5\delta-1}{2}} \, .
\end{equation}
From~\eqref{eq:growth:bounds}, this converges to zero as $\epsilon\rightarrow 0$.  In addition, from~\eqref{ind:aq':bounds}, $C_{ij} \mathcal{R}_i \mathcal{R}_j a_{\rm low}^2\rightarrow 0$ in $X$ as $\epsilon \rightarrow 0$.

To handle the remaining terms from~\eqref{eq:a:limit}, all of which contain a factor of $a_{\rm high}$, we will use that $a_{\rm{high}} = (\Id - \bp_{\leq \Lambda_{q(\epsilon)-1}^{\delta'}})a_{q(\epsilon)-1}$, and that $(\Id - \bp_{\leq \Lambda_{q(\epsilon)-1}^{\delta'}}) a_{q(\epsilon)-1}$ satisfies a good bound due to Lemma~\ref{lem:good_kernel}.   Therefore we appeal to Lemma \ref{lem:proj_est}, Lemma \ref{lem:good_kernel}, Lemma \ref{lem:decoupling}, and Lemma \ref{lem:HLS-Holder} and choose $180\delta^{-1} < N < 190\delta^{-1}$ to see that
\begin{equation*}
    \begin{split}
    &\Vert B(a_{\rm{low}}\rho_\epsilon \ast U_{q(\epsilon)-1}, a_{\rm{high}}\rho_\epsilon \ast U_{q(\epsilon)-1}) \Vert_{H^{-10}(\R^3)}  \\
    &\qquad \lesssim \Vert a_{\rm{high}} \rho_\epsilon \ast U_{q(\epsilon)-1} \Vert_{L^1(\R^3)} \Vert a_{\rm{low}}\rho_\epsilon \ast U_{q(\epsilon)-1} \Vert_{W^{5,1}(\R^3)}\\
        &\qquad \lesssim \Vert a_{\rm{high}} \Vert_{W^{4,1}(\R^3)} \Vert U_{q(\epsilon)-1} \Vert_{L^1(\T^3)} \Vert a_{\rm{low}} \Vert_{W^{9,1}(\R^3)} \Vert U_{q(\epsilon)-1} \Vert_{W^{5,1}(\T^3)}\\
        &\qquad \lesssim \Vert a_{\rm{high}} \Vert_{W^{4,1}(\R^3)} \Vert a_{\rm{low}} \Vert_{W^{9,1}(\R^3)} \Vert U_{q(\epsilon)-1} \Vert_{W^{5,1}(\T^3)}^2\\
        &\qquad \lesssim \Lambda_{q(\epsilon)-1}^{-N\delta'} \lambda_{q(\epsilon)-1}^2 \Lambda_{q(\epsilon)-1}^{10} \Lambda_{q(\epsilon)-1}^{3\delta-1}\\
        &\qquad \lesssim \lambda_{q(\epsilon)-1}^2 \Lambda_{q(\epsilon)-1}^{-5} \, .
    \end{split}
\end{equation*}
Then from~\eqref{eq:growth:bounds}, this goes to zero as $\epsilon \rightarrow 0$.  Using very similar methods for $0 < \gamma \ll 1$ chosen small enough, we have
\begin{equation*}
    \begin{split}
        &\Vert B(a_{\rm{high}}\rho_\epsilon \ast U_{q(\epsilon)-1}, a_{\rm{low}}\rho_\epsilon \ast U_{q(\epsilon)-1}) \Vert_{H^{-10}(\R^3)} \\ &\qquad \lesssim \Vert a_{\rm{high}}\rho_\epsilon \ast U_{q(\epsilon)-1} \Vert_{L^{1+\gamma}(\R^3)} \Vert a_{\rm{low}} \rho_\epsilon \ast U_{q(\epsilon)-1} \Vert_{W^{5,1}(\R^3)}\\
        &\qquad \lesssim \Vert a_{\rm{high}} \Vert_{W^{4,1+\gamma}(\R^3)} \Vert U_{q(\epsilon)-1} \Vert_{L^{1+\gamma}(\T^3)} \Vert a_{\rm{low}} \Vert_{W^{9,1}(\R^3)} \Vert U_{q(\epsilon)-1} \Vert_{W^{5,1}(\T^3)}\\
        &\qquad \lesssim \Vert a_{\rm{high}} \Vert_{W^{5,1}(\R^3)} \Vert a_{\rm{low}} \Vert_{W^{9,1}(\R^3)} \Vert U_{q(\epsilon)-1} \Vert_{W^{5,1}(\T^3)}^2\\
        &\qquad \lesssim \Lambda_{q(\epsilon)-1}^{-N\delta'}\lambda_{q(\epsilon)-1}^2 \Lambda_{q(\epsilon)-1}^{10} \Lambda_{q(\epsilon)-1}^{3\delta-1}\\
        &\qquad \lesssim \lambda_{q(\epsilon)-1}^2 \Lambda_{q(\epsilon)-1}^{-5} \, ,
    \end{split}
\end{equation*}
and similarly
\begin{equation*}
    \begin{split}
        &\Vert B(a_{\rm{high}}\rho_\epsilon \ast U_{q(\epsilon)-1}, a_{\rm{high}}\rho_\epsilon \ast U_{q(\epsilon)-1}) \Vert_{H^{-10}(\R^3)}  \\
        &\qquad \lesssim \Vert a_{\rm{high}} \rho_\epsilon \ast U_{q(\epsilon)-1} \Vert_{L^1(\R^3)} \Vert a_{\rm{high}}\rho_\epsilon \ast U_{q(\epsilon)-1} \Vert_{W^{5,1}(\R^3)}\\
        &\qquad \lesssim  \Vert a_{\rm{high}} \Vert_{W^{4,1}(\R^3)} \Vert U_{q(\epsilon)-1} \Vert_{L^1(\T^3)} \Vert a_{\rm{high}} \Vert_{W^{9,1}(\R^3)} \Vert U_{q(\epsilon)-1} \Vert_{W^{5,1}(\T^3)}\\
        &\qquad \lesssim \Lambda_{q(\epsilon)-1}^{-2N\delta'} \lambda_{q(\epsilon)-1}^2 \Lambda_{q(\epsilon)-1}^{10} \Lambda_{q(\epsilon)-1}^{3\delta-1}\\
        &\qquad \lesssim \lambda_{q(\epsilon)-1}^2 \Lambda_{q(\epsilon)-1}^{-5}\, .
    \end{split}
\end{equation*}
Both approach zero as $\epsilon \rightarrow 0$ from~\eqref{eq:growth:bounds}.

\bigskip

\noindent\texttt{Step 5b: the commutator terms $II$, $III$, and $IV$.}
We start by estimating $II+III$ (to keep the symmetrized form of the bilinear operator).  From Lemma~\ref{lem:HLS-Holder}, Lemma~\ref{lem:decoupling},~\eqref{ind:aq':bounds}, Proposition~\ref{prop:intermittent_funcs}, the Sobolev embedding $W^{1,1}(\R^3) \subset W^{\sfrac 12,1}(\R^3) \subset L^{\sfrac 65}(\R^3)$, and Lemma~\ref{lem:com:bound} with $\alpha=1$, we have that
\begin{align*}
&\left\| B(a_{q(\epsilon)-1}\rho_\epsilon \ast U_{q(\epsilon)-1}, \left[\rho_\epsilon \ast,a_{q(\epsilon)-1}\right]U_{q(\epsilon)-1}) + B(\left[\rho_\epsilon \ast,a_{q(\epsilon)-1}\right]U_{q(\epsilon)-1}, a_{q(\epsilon)-1}\rho_\epsilon \ast U_{q(\epsilon)-1}) \right\|_{H^{-10}(\R^3)} \\
&\qquad \les \left\| a_{q(\epsilon)-1}\rho_\epsilon \ast U_{q(\epsilon)-1} \right\|_{L^{\sfrac 65}(\R^3)} \left\| \left[\rho_\epsilon \ast,a_{q(\epsilon)-1}\right]U_{q(\epsilon)-1} \right\|_{L^{\sfrac 65}(\R^3)} \\
&\qquad \les \lambda_{q(\epsilon)-1} \Lambda_{q(\epsilon)-1}^\delta \left(  C_\rho  R_{q(\epsilon)-1}^6 \lambda_{q(\epsilon)-1} \Lambda_{q(\epsilon)-1}^{\frac{3\delta-1}{2}} \right) \\
&\qquad \les  C_\rho  R_{q(\epsilon)-1}^6 \lambda_{q(\epsilon)-1}^2 \Lambda_{q(\epsilon)-1}^{\frac{5\delta-1}{2}}    \, .
\end{align*}
Using~\eqref{eq:growth:bounds} and the fact that $\delta \leq \sfrac{1}{6}$,  this term approaches zero as $\epsilon\rightarrow 0$. Finally, $IV$ satisfies even better bounds, since the commutator is in both slots (and so the bilinear form is automatically symmetrized), so we omit the details and conclude the estimates for $II$, $III$, and $IV$ from~\eqref{1234}.
\bigskip
    
\noindent\texttt{Step 5c: proof of Lemma~\ref{lem:com:bound}. } The commutator can be written as
\begin{equation}\label{eq:comm_form_1}
        \int_{\R^3} \rho_\epsilon(x-y) \left(a_{q(\epsilon)-1}(y) - a_{q(\epsilon)-1}(x)\right) U_{q(\epsilon)-1}(y)\, dy
    \end{equation}
    and
    \begin{equation}\label{eq:comm_form_2}
        \int_{\R^3} \rho_\epsilon(y) \left(a_{q(\epsilon)-1}(x-y) - a_{q(\epsilon)-1}(x)\right) U_{q(\epsilon)-1}(x-y)\, dy\, .
    \end{equation}
The specific representation of the commutator we utilize depends on the size of $\epsilon^{-1}$ compared to $\Lambda_{q(\epsilon)-1}$.  The former is useful when $\epsilon^{-1} \leq \Lambda_{q(\epsilon)-1}$, so that $\pa_x$ landing on the mollifier instead of $U_{q(\epsilon)-1}$ is helpful.  The latter is useful if $\epsilon^{-1}> \Lambda_{q(\epsilon)-1}$, so that $\pa_x$ landing on $U_{q(\epsilon)-1}$ is useful.
      
\begin{enumerate}
    \item We first assume that $\Lambda_{q(\epsilon)-1}^{\sfrac{\delta}{2}} < \epsilon^{-1} \leq \Lambda_{q(\epsilon)-1}$. We apply spatial derivatives $\nabla_x^m$ for fixed $m=0,1$ to~\eqref{eq:comm_form_1} and compute as follows.  Let $m_1, m_2 = 0, 1$ be such that $m_1+m_2 = m$. Using the periodicity of $U_{q(\epsilon)-1}$, we have that
\begin{equation}\label{eq:L1_com_est_1}
    \begin{split}
    &\left\Vert \int_{\R^3} \nabla_x^{m_1}\rho_\epsilon(x-y) \nabla_x^{m_2}\left(a_{q(\epsilon)-1}(y) - a_{q(\epsilon)-1}(x)\right) U_{q(\epsilon)-1}(y)\, dy \right\Vert_{L^1_x(\R^3)}\\
    &\lesssim \int_{\R^3} \int_{\R^3} |\nabla_x^{m_1}\rho_\epsilon(x-y)| |\nabla_x^{m_2}(a_{q(\epsilon)-1}(y) - a_{q(\epsilon)-1}(x))||U_{q(\epsilon)-1}(y)|\, dy\, dx \\
&= \sum_{k \in \Z^3} \int_{\T^3} |U_{q(\epsilon)-1}(y+k)| \int_{\R^3} |\nabla_x^{m_2}(a_{q(\epsilon)-1}(y+k)-a_{q(\epsilon)-1}(x))| |\nabla_x^{m_1}\rho_\epsilon(x-y-k)| \, dx \, dy \\
&= \int_{\T^3} |U_{q(\epsilon)-1}(y)| \sum_{k \in \Z^3} \int_{\R^3} |\nabla_x^{m_2}(a_{q(\epsilon)-1}(y+k)-a_{q(\epsilon)-1}(x))| |\nabla_x^{m_1}\rho_\epsilon(x-y-k)| \, dx \, dy \, . 
\end{split}
\end{equation}
For $|k| \leq  R_{q(\epsilon)-1}^2$, we may use the inductive bound on $a_{q(\epsilon)-1}$ from item~\ref{i:ind:ugh} and the $L^1(\T^3)$ bound on $U_{q(\epsilon)-1}$ from Proposition~\ref{prop:intermittent_funcs} to estimate
\begin{equation}\label{eq:L1_com_est_2}
    \begin{split}
    &\int_{\T^3} |U_{q(\epsilon)-1}(y)|  \sum_{|k| \leq R_{q(\epsilon)-1}^2} \int_{\R^3}  |\nabla_x^{m_2}(a_{q(\epsilon)-1}(y + k) - a_{q(\epsilon)-1}(x))| |\nabla_x^{m_1}\rho_{\epsilon}(x - y - k)|\, dx\, dy\\
    &\lesssim \int_{\T^3} |U_{q(\epsilon)-1}(y)| \lambda_{q(\epsilon)-1} \sum_{|k| \leq R_{q(\epsilon)-1}^2}\int_{\R^3} |x -y-k|^{1-\min(m_2,1)} |\nabla_x^{m_1}\rho_\epsilon(x-y-k)|\, dx\, dy\\
    &\lesssim \int_{\T^3} |U_{q(\epsilon)-1}(y)| \lambda_{q(\epsilon)-1} R_{q(\epsilon)-1}^6 \int_{\R^3} |x-y|^{1-\min(m_2,1)} |\nabla_x^{m_1}\rho_\epsilon(x-y)| \, dx \, dy \\
    &\leq C_\rho R_{q(\epsilon)-1}^6 \lambda_{q(\epsilon)-1} \epsilon^{1- \min(m_2,1)-m_1} \Lambda_{q(\epsilon)-1}^{\frac{3\delta-1}{2}} \, \\
    &\leq C_\rho R_{q(\epsilon)-1}^6 \lambda_{q(\epsilon)-1} \epsilon^{1-m} \Lambda_{q(\epsilon)-1}^{\frac{3\delta-1}{2}} \, .
            \end{split}
        \end{equation}
When $|k| \geq R_{q(\epsilon)-1}^2$, we split into two cases. If $|x| \geq \frac{1}{2}|k|$, then by the compact support of $a_{q(\epsilon)-1}$ given in item~\ref{i:ind:ugh}, $\nabla_x^{m_2}a_{q(\epsilon)-1}(x)\equiv 0$.  In addition, since $|k| \geq R_{q(\epsilon)-1}^2$ and $y \in \T^3$, $a_{q(\epsilon)-1}(y+k) \equiv 0$ as well. So it suffices to treat the case $|k| \geq R_{q(\epsilon)-1}^2$ under the additional assumption that $|x| \leq \frac{1}{2} |k|$.  Using that $\rho$ is Schwartz and $|y| \les 1$ since $y \in \T^3$ , we estimate
$$
|\nabla_x^{m_1}\rho_\epsilon(x-y-k)| \leq C_{\rho} \epsilon^{100-3-m_1}|x - y - k|^{-100} \leq C_{\rho} \epsilon^{97-m_1}|k|^{-100} \, . 
        $$
Hence
\begin{align}\label{eq:L1_com_est_3}
    &\int_{\T^3} |U_{q(\epsilon)-1}(y)|  \sum_{|k| \geq R_{q(\epsilon)-1}^2} \int_{\{|x| \leq \frac 12 |k|\}}  |\nabla_{x}^{m_2}(a_{q(\epsilon)-1}(y + k) - a_{q(\epsilon)-1}(x))| |\nabla_x^{m_1}\rho_{\epsilon}(x - y - k)|\, dx\, dy \notag \\
    &\lesssim \int_{\T^3} |U_{q(\epsilon)-1}(y)|  \sum_{|k| \geq R_{q(\epsilon)-1}^2} \lambda_{q(\epsilon)-1} \int_{\{|x| \leq \frac 12 |k|\}} |\nabla_x^{m_1}\rho_{\epsilon}(x - y - k)|\, dx\, dy \notag \\
    &\lesssim C_{\rho} \int_{\T^3} |U_{q(\epsilon)-1}(y)|  \sum_{|k| \geq R_{q(\epsilon)-1}^2} \lambda_{q(\epsilon)-1} \epsilon^{97-m_1} |k|^{3} |k|^{-100} \notag \\
    &\lesssim C_{\rho} R_{q(\epsilon)-1}^{-50} \lambda_{q(\epsilon)-1} \epsilon^{97-m_1} \Lambda_{q(\epsilon)-1}^{\frac{3\delta-1}{2}} \notag \\
    &\lesssim C_{\rho} \lambda_{q(\epsilon)-1} \epsilon^{1-m} \Lambda_{q(\epsilon)-1}^{\frac{3\delta-1}{2}} \, .
\end{align}
Then combining~\eqref{eq:L1_com_est_2} and \eqref{eq:L1_com_est_3} to estimate~\eqref{eq:L1_com_est_1}, we have that for $\alpha =0$ or $\alpha=1$,
\begin{align}
    &\left\Vert \int_{\R^3} \rho_\epsilon(x-y) \left(a_{q(\epsilon)-1}(y) - a_{q(\epsilon)-1}(x)\right) U_{q(\epsilon)-1}(y)\, dy \right\Vert_{W^{\alpha,1}(\R^3)} \lesssim R_{q(\epsilon)-1}^6 \lambda_{q(\epsilon)-1} \epsilon^{1-\alpha} \Lambda_{q(\epsilon)-1}^{\frac{3\delta-1}{2}} \, .  \label{eq:interpolation} 
\end{align}  
\item Now we assume $\Lambda_{q(\epsilon)-1} < \epsilon^{-1}$.   We apply spatial derivatives $\nabla_x^m$ for fixed $m=0,1$ to~\eqref{eq:comm_form_2} and compute as follows. Let $m_1, m_2 = 0 , 1$ be such that $m_1 + m_2 = m$. Using the periodicity of $U_{q(\epsilon)-1}$, we have that
\begin{equation}\label{eq:L1_com_est_1_redux}
    \begin{split}
    &\left\Vert \int_{\R^3} \rho_\epsilon(y) \nabla_x^{m_2}\left(a_{q(\epsilon)-1}(x-y) - a_{q(\epsilon)-1}(x)\right) \nabla_x^{m_1}U_{q(\epsilon)-1}(x-y)\, dy \right\Vert_{L^1_x(\R^3)}\\
    &\leq \int_{\R^3} \int_{\R^3} |\rho_\epsilon(y)| |\nabla_x^{m_2}(a_{q(\epsilon)-1}(x-y) - a_{q(\epsilon)-1}(x))||\nabla_x^{m_1}U_{q(\epsilon)-1}(x-y)|\, dx\, dy \\
&= \int_{\R^3} \int_{\R^3} |\rho_\epsilon(y)| |\nabla_{x'}^{m_2}(a_{q(\epsilon)-1}(x') - a_{q(\epsilon)-1}(x'+y))||\nabla_{x'}^{m_1}U_{q(\epsilon)-1}(x')|\, dx'\, dy \\
&= \int_{\R^3_y} \sum_{k \in \Z^3} \int_{\T^3_x} |\nabla_x^{m_1}U_{q(\epsilon)-1}(x'+k)| |\nabla_x^{m_2}(a_{q(\epsilon)-1}(x'+k)-a_{q(\epsilon)-1}(x'+k+y))| |\rho_\epsilon(y)| \,dx' \, dy \\
&= \sum_{k \in \Z^3} \int_{\R^3} \int_{\T^3} |\nabla_{x'}^{m_1}U_{q(\epsilon)-1}(x')| |\nabla_{x'}^{m_2}(a_{q(\epsilon)-1}(x'+k)-a_{q(\epsilon)-1}(x'+k-y))| |\rho_\epsilon(y)| \, dx' \, dy \, . 
\end{split}
\end{equation}
For $|k| \leq  R_{q(\epsilon)-1}^2$, we may use the inductive bound on $a_{q(\epsilon)-1}$ from item~\ref{i:ind:ugh}, the $L^1(\T^3)$ bound on $\nabla^{m_1} U_{q(\epsilon)-1}$ from Proposition~\ref{prop:intermittent_funcs}, and the inequality $\Lambda_{q(\epsilon)-1} \leq \epsilon^{-1}$ to estimate
\begin{equation}\label{eq:L1_com_est_2_redux}
    \begin{split}
    &\sum_{|k| \leq R_{q(\epsilon)-1}^2} \int_{\R^3}\int_{\T^3} |\nabla_{x'}^{m_1}U_{q(\epsilon)-1}(x')|    |\nabla_{x'}^{m_2}(a_{q(\epsilon)-1}(x'+k) - a_{q(\epsilon)-1}(x'+k-y))| |\rho_{\epsilon}(y)|\, dx'\, dy\\
    &\lesssim \sum_{|k| \leq R_{q(\epsilon)-1}^2}\int_{\R^3}\int_{\T^3} |\nabla_{x'}^{m_1}U_{q(\epsilon)-1}(x')| \lambda_{q(\epsilon)-1}  |y|^{1-\min(m_2,1)} |\rho_\epsilon(y)|\, dx'\, dy\\
    &\leq C_\rho R_{q(\epsilon)-1}^6 \lambda_{q(\epsilon)-1} \epsilon^{1- \min(m_2,1)} \Lambda_{q(\epsilon)-1}^{\frac{3\delta-1}{2}+m_1} \, \\
    &\leq C_\rho R_{q(\epsilon)-1}^6 \lambda_{q(\epsilon)-1} \epsilon^{1-\min(m_2, 1)-m_1} \Lambda_{q(\epsilon)-1}^{\frac{3\delta-1}{2}} \\
    &\leq C_\rho R_{q(\epsilon)-1}^6 \lambda_{q(\epsilon)-1} \epsilon^{1-m} \Lambda_{q(\epsilon)-1}^{\frac{3\delta-1}{2}} \, .
    \end{split}
    \end{equation}
When $|k| \geq R_{q(\epsilon)-1}^2$, then by the compact support of $a_{q(\epsilon)-1}$ given in item~\ref{i:ind:ugh} and the fact that $x' \in \T^3$,
$$ |\nabla_x^{m_2}(a_{q(\epsilon)-1}(x'+k) - a_{q(\epsilon)-1}(x'+k-y))| = |\nabla_x^{m_2}a_{q(\epsilon)-1}(x'+k-y)| \neq 0 $$
only if $|y-(-k)| \leq 2R_{q(\epsilon)-1}$. In addition, by the fact that $\rho$ is Schwartz, there exists $C_\rho$ such that
$$ \int_{\{|y-(-k)|\leq 2R_{q(\epsilon)-1}\}}|\rho_\epsilon(y)| \, dy = \int_{\{|\epsilon y'-(-k)|\leq 2R_{q(\epsilon)-1}\}}|\rho(y')| \, dy' \leq C_\rho R_{q(\epsilon)-1}^3 \epsilon^{-3+1000}|k|^{-1000} \, . $$
Then using the inductive bound on $a_{q(\epsilon)-1}$ from item~\ref{i:ind:ugh}, Sobolev embedding, and the $L^1(\T^3)$ bound on $\nabla^{m_1} U_{q(\epsilon)-1}$ from Proposition~\ref{prop:intermittent_funcs}, we have that
\begin{equation}\label{eq:L1_com_est_4_redux}
\begin{split}
   &\sum_{|k|\geq R_{q(\epsilon)-1}^2} \int_{\R^3} \int_{\T^3} |\nabla_{x'}^{m_1}U_{q(\epsilon)-1}(x')| |\nabla_{x'}^{m_2}(a_{q(\epsilon)-1}(x'+k)-a_{q(\epsilon)-1}(x'+k-y))| |\rho_\epsilon(y)| \, dx' \, dy \\
   &= \sum_{|k| \geq R_{q(\epsilon)-1}^2} \int_{\{|y-(-k)|\leq 2R_{q(\epsilon)-1}\}} \int_{\T^3} |\nabla_{x'}^{m_1} U_{q(\epsilon)-1}(x')| \lambda_{{q(\epsilon)-1}} |\rho_\epsilon(y)| \, dx' \, dy \\
   &\les \sum_{|k| \geq R_{q(\epsilon)-1}^2} C_\rho R_{q(\epsilon)-1}^3 \epsilon^{997} |k|^{-1000} \Lambda_{{q(\epsilon)-1}}^{\frac{3\delta-1}{2} + m_1} \\
   &\leq C_\rho R_{q(\epsilon)-1}^3 \Lambda_{q(\epsilon)-1}^{-1} \, .
    \end{split}
\end{equation}
Combining~\eqref{eq:L1_com_est_4_redux} with~\eqref{eq:L1_com_est_2_redux}, we have that for $\alpha =0$ or $\alpha=1$, 
$$ \left\| \int_{\R^3} \rho_\epsilon(y) \left(a_{q(\epsilon)-1}(x-y) - a_{q(\epsilon)-1}(x)\right) U_{q(\epsilon)-1}(x-y)\, dy \right\|_{W^{\alpha,1}(\R^3)} \leq C_\rho R_{q(\epsilon)-1}^6 \lambda_{q(\epsilon)-1} \Lambda_{q(\epsilon)-1}^{\alpha-1} \Lambda_{q(\epsilon)-1}^{\frac{3\delta-1}{2}} \, . $$
Combining this with~\eqref{eq:interpolation}, we obtain~\eqref{eq:comm_est}.  
\end{enumerate}

\section{Proof of Proposition~\ref{prop:ind:ks}}

Throughout this section we will be working with three parameters $R_{q+1} \ll \lambda_{q+1} \ll \Lambda_{q+1}$ all assumed to be integers. In addition, we will assume that $\Lambda_{q+1}^\delta \in \N$. The size of $R_{q+1}$ is allowed to depend on all constants fixed at stages $q$ and earlier from the iteration procedure and propagated in the inductive assumptions. Then $\lambda_{q+1}$ is allowed to depend on $R_{q+1}$, and finally $\Lambda_{q+1}$ is allowed to depend on $\lambda_{q+1}$.

\subsection{The increment}

\begin{definition}[\textbf{The increment}]\label{def:gq+1}
    Put $\chi_{q+1} := \chi_{R_{q+1}}$ from Definition \ref{def:projs}, and $U_{q+1} := U_{\Lambda_{q+1}}$ from Proposition \ref{prop:intermittent_funcs}, and define $g_{q+1}:\R^3 \to \R$ by
    $$
    g_{q+1} := a_{q+1}U_{q+1}
    $$
    where
    $$
    a_{q+1} := \left(2 \Vert E_q \Vert_{L^\infty(\R^3)} - E_q\right)^{1/2} \chi_{q+1} \, .
    $$
    The size of $\Lambda_{q+1}$, $\lambda_{q+1}$, and $R_{q+1}$ will be chosen based on the estimates in this section.
\end{definition}

\begin{lemma}[\textbf{Properties of $a_{q+1}$}]\label{lem:aq+1_props}
    Recall $a_{q+1}$ from Definition \ref{def:gq+1}. Then there exists $\lambda_{q+1}$ sufficiently large so that the following hold.
    \begin{enumerate}
        \item $a_{q+1}$ is smooth and compactly supported in $B(0,R_{q+1})$;
        \item $\Vert a_{q+1} \Vert_{W^{200\delta^{-1},1}(\R^3)} \leq \lambda_{q+1}$;
        \item $\left\| \left(\bp_{\leq \Lambda_{q+1}^{\delta'}} a_{q+1} \right)^2 \right\|_{X} \leq  C2^{-q-8}$ and $\left\| \mathcal{R}_i \mathcal{R}_j \left[ \left(\bp_{\leq \Lambda_{q+1}^{\delta'}} a_{q+1} \right)^2 \right] \right\|_{X} \leq  C2^{-q-5}$.
    \end{enumerate}
\end{lemma}
\begin{proof}
    The smoothness of $a_{q+1}$ is clear from the fact that $E_q$ and $\chi_{q+1}$ are smooth and
    $$
    0 < \frac{1}{2} \Vert E_q \Vert_{L^\infty(\R^3)} < 2\Vert E_q \Vert_{L^\infty(\R^3)} - E_q < \infty\, .\footnote{The first inequality holds since there is no $q$ such that $E_q$ is identically $0$. Otherwise the iteration procedure would terminate and we will have constructed a nontrivial solution to the stationary Krieger--Strain equation which lies in $C_c^\infty(\R^3) \subset L^{\sfrac 65}(\R^3)$. This contradicts Theorem \ref{prop:rigidity}.}
    $$
    The compact support within $B(0,R_{q+1})$ follows from Definition \ref{def:projs}. Now by definition we have
    \begin{equation}\label{eq:aq+1_est}
        \Vert a_{q+1} \Vert_{W^{200\delta^{-1},1}(\R^3)} = \sum_{|\alpha| \leq 200\delta^{-1}} \left\Vert \partial^\alpha\left[ (2\Vert E_q \Vert_{L^\infty(\R^3)} - E_q)^{1/2} \chi_{q+1} \right]\right\Vert_{L^1(\R^3)} \, .
    \end{equation}
    The quantity on the right-hand side of \eqref{eq:aq+1_est} depends only on $\delta$, $\chi$, $E_q$, and $R_{q+1}$. Upon choosing $\lambda_{q+1}$ large enough and depending on these parameters, we may ensure that
    $$
    \Vert a_{q+1} \Vert_{W^{200\delta^{-1},1}(\R^3)} \leq \lambda_{q+1} \, .
    $$
   Now we estimate $a_{q+1}^2$, before using a large choice of $\Lambda_{q+1}$ to deduce that the same bounds hold for $\bp_{\leq \Lambda_{q+1}^{\delta'}} a_{q+1}$.  We begin by writing
   $$
   \Vert a_{q+1}^2 \Vert_X = \Vert (2\Vert E_q \Vert_{L^\infty(\R^3)} - E_q)\chi_{q+1}^2\Vert_X \leq 2\Vert E_q \Vert_{L^\infty(\R^3)} \Vert \chi_{q+1}^2 \Vert_{\dot{W}^{1,4}(\R^3)} + \Vert E_q (1 - \chi_{q+1}^2) \Vert_{L^2(\R^3)} + \Vert E_q \Vert_X \, .
   $$ 
   Direct computation shows that
    $$
    \int_{|x|<R_{q+1}} \left| \nabla \left(\chi_{q+1}^2\right)\right|^4 \leq 256 \int_{|x| < R_{q+1}} \left|\nabla \chi_{q+1}\right|^4 \leq 512 R_{q+1}^{-1} \int_{|x| < 2} \left|\nabla \chi \right|^4 \leq 8192 R_{q+1}^{-1} \, .
    $$
    With this and the fact that $E_q \in L^2(\R^3)$, we may choose $R_{q+1}$ large enough such that
    $$
    2\Vert E_q \Vert_{L^\infty(\R^3)} \Vert \chi_{q+1}^2 \Vert_{\dot{W}^{1,4}(\R^3)} + \Vert E_q (\chi_{q+1}^2 -1)\Vert_{L^2(\R^3)} < C2^{-q-10} \, .
    $$
    Then using inductive assumption~\ref{i:ind:3}, we have
    $$
    \Vert a_{q+1}^2\Vert_X = \Vert \left(2\Vert E_q \Vert_{L^\infty(\R^3)} - E_q\right)\chi_{q+1}^2 \Vert_X < C2^{-q-8} \, .
    $$

Now we write
$$
\left[\mathbb{P}_{\leq \Lambda_{q+1}^{\delta'}} (a_{q+1})\right]^2 = a_{q+1}^2 - 2 a_{q+1} \mathbb{P}_{>\Lambda_{q+1}^{\delta'}}(a_{q+1}) + \left[\mathbb{P}_{>\Lambda_{q+1}^{\delta'}}(a_{q+1})\right]^2
$$
so then
$$
\left\Vert \left[\mathbb{P}_{\leq \Lambda_{q+1}^{\delta'}} (a_{q+1})\right]^2 \right\Vert_X \leq \left\Vert a_{q+1}^2 \right\Vert_X + 2 \Vert a_{q+1} \Vert_{L^\infty(\R^3)} \Vert \mathbb{P}_{>\Lambda_{q+1}^{\delta'}}(a_{q+1}) \Vert_{L^1(\R^3)} + \Vert \mathbb{P}_{>\Lambda_{q+1}^{\delta'}}(a_{q+1}) \Vert_{L^2(\R^3)}^2 \, .
$$
Appealing to Lemma \ref{lem:good_kernel}, we may choose $\Lambda_{q+1}$, depending on $\lambda_{q+1}$, large enough so that
$$
2 \Vert a_{q+1} \Vert_{L^\infty(\R^3)} \Vert \mathbb{P}_{>\Lambda_{q+1}^{\delta'}}(a_{q+1}) \Vert_{L^1(\R^3)} + \Vert \mathbb{P}_{>\Lambda_{q+1}^{\delta'}}(a_{q+1}) \Vert_{L^2(\R^3)}^2 \ll 1 \, .
$$
Thus for this sufficiently large choice of $\Lambda_{q+1}$ we have that
$$
\left\Vert \left[\mathbb{P}_{\leq \Lambda_{q+1}^{\delta'}} (a_{q+1})\right]^2 \right\Vert_X \leq C2^{-q-8} \, .
$$
Finally, applying Lemma \ref{lem:Riesz_bdd} we have that
$$
\Vert \mathcal{R}_i \mathcal{R}_j a_{q+1}^2 \Vert_X < 4 \Vert a_{q+1}^2 \Vert_X < C2^{-q-5} \, ,
$$
and from a large choice of $\Lambda_{q+1}$, the desired inequality holds for $\left(\bp_{\leq \Lambda_{q+1}^{\delta'}} a_{q+1}\right)^2$.
\end{proof}

\begin{lemma}[\textbf{Properties of $g_{q+1}$}]\label{lem:gq+1_est}
    For $g_{q+1}$ from Definition \ref{def:gq+1}, $1 \leq p \leq \infty$, and $k \leq 200\delta^{-1} - 8$, we have that (for an implicit constant which does not depend on $q$)
    \begin{equation*}
        \Vert g_{q+1} \Vert_{W^{k,p}(\R^3)} \lesssim \lambda_{q+1}\Lambda_{q+1}^{k + 1 + 3(1-\delta)\left(\frac{1}{2} - \frac{1}{p}\right)} \, ,
    \end{equation*}
and $g_{q+1}$ is smooth, nonnegative, and compactly supported in $B(0,R_{q+1})$.
\end{lemma}
\begin{proof}
    Applying Lemma \ref{lem:decoupling}, Proposition \ref{prop:intermittent_funcs}, Lemma \ref{lem:aq+1_props}, and Sobolev embedding, we have
    \begin{equation*}
    \Vert g_{q+1} \Vert_{W^{k,p}(\R^3)} \lesssim \Vert a_{q+1} \Vert_{W^{k+8,1}(\R^3)} \Vert U_{q+1} \Vert_{W^{k,p}(\T^3)} \lesssim \lambda_{q+1} \Lambda_{q+1}^{k + 1 + 3(1-\delta)\left(\frac{1}{2} - \frac{1}{p}\right)} \, .
    \end{equation*}
The compact support follows from Definition~\ref{def:projs}, and the smoothness follows from Proposition~\ref{prop:intermittent_funcs}, Definition~\ref{def:projs}, and the smoothness of $E_q$, which implies the smoothness of $(2\| E_q \|_{L^\infty(\R^3)} - E_q)^{\sfrac 12}$.  The non-negativity follows from the non-negativity of $a_{q+1}$, $U_{q+1}$, and $\chi_{q+1}$. The support being contained in $B(0,R_{q+1})$ follows from the support of $a_{q+1}$ being contained in the same ball.
\end{proof}

\subsection{Proof of item \ref{i:ind:1}}\label{pf:item1}
With $f_q$ and $E_q$ given and $g_{q+1}$ constructed in Definition~\ref{def:gq+1}, we define
$$
f_{q+1} = f_q + g_{q+1} \qquad \textnormal{and} \qquad
E_{q+1} = B(f_{q+1},f_{q+1}) \, .
$$
We have that $f_{q+1}$ is Schwartz since $f_q$ is Schwartz by induction, and $g_{q+1}$ is Schwartz from Lemma~\ref{lem:gq+1_est}. Since $E_{q+1} = B(f_{q+1},f_{q+1})$, it is clear from Lemma~\ref{lem:HLS-Holder} and the fact that $f_{q+1}$ is Schwartz that $E_{q+1}$ is smooth. The non-negativity of $f_{q+1}$ follows from the non-negativity of $f_q$ (propagated inductively) and $g_{q+1}$ (from Lemma~\ref{lem:gq+1_est}). The same argument employed in the base case to show $E_0 \in L^2(\R^3) \cap W^{k,\infty}(\R^3)$ for all $k$ applies here as well since $f_{q+1} \in \mathcal{S}(\R^3)$, and so $E_{q+1} \in L^2(\R^3) \cap W^{k,\infty}(\R^3)$ for all $k$. Finally since \eqref{eq:rel} holds by definition, we also have \eqref{eq:weak_rel}.

\subsection{Proof of item \ref{i:ind:2}}\label{pf:item2}
First note that upon completion of this section, we will have chosen $\Lambda_{q+1}$, $\lambda_{q+1}$, and $R_{q+1}$. Next, we must prove \eqref{ind:f:estimates:1} at level $q+1$.
From Lemma \ref{lem:gq+1_est} we may choose $\Lambda_{q+1}$ large enough depending on $\lambda_{q+1}$ so that
$$
\Vert g_{q+1} \Vert_{L^1(\R^3)} < C^{-1}2^{-q-1} \, .
$$
As a consequence, we have
$$
\Vert f_{q+1} \Vert_{L^1(\R^3)} \geq \Vert f_q \Vert_{L^1(\R^3)} - \Vert g_{q+1} \Vert_{L^1(\R^3)} \geq C^{-1}(1 + 2^{-q}) - C^{-1}2^{-q-1} = C^{-1}(1 + 2^{-q-1}).
$$
The smoothness and non-negativity of $g_{q+1}$ follow from Lemma~\ref{lem:gq+1_est}. From Lemma \ref{lem:gq+1_est} for $k \leq 6$, we have that
$$
\Vert \nabla^k g_{q+1} \Vert_{L^1(\R^3)} \lesssim \lambda_{q+1} \Lambda_{q+1}^{k+\frac{3\delta-1}{2}}
$$
and
$$
\Vert \nabla^k g_{q+1} \Vert_{L^{p_0}(\R^3)} \lesssim \lambda_{q+1} \Lambda_{q+1}^{k+1+3(1-\delta)\left(\frac{1}{2} - \frac{1}{p_0}\right)} \, ,
$$
and so
$$
\Lambda_{q+1}^{-k}\left(\Vert \nabla^k g_{q+1} \Vert_{L^1(\R^3)} + \Vert \nabla^k g_{q+1} \Vert_{L^{p_0}(\R^3)}\right) \lesssim \lambda_{q+1}\Lambda_{q+1}^{1 + 3(1-\delta)\left(\frac{1}{2} - \frac{1}{p_0}\right)} \, .
$$
Since we have chosen $\delta$ small enough such that
$$
0<\delta < \frac{6-5p_0}{3(2-p_0)} = 1- \frac{1}{3\left(\frac{1}{p_0} - \frac{1}{2}\right)} \, ,
$$
we may choose $\Lambda_{q+1}$ large enough, depending on $\lambda_{q+1}$, such that
$$
\Lambda_{q+1}^{-k}\left(\Vert \nabla^k g_{q+1} \Vert_{L^1(\R^3)} + \Vert \nabla^k g_{q+1} \Vert_{L^{p_0}(\R^3)}\right) < C2^{-q-12} \, .
$$
Finally, upon choosing $\Lambda_{q+1}$ large enough depending on only $\lambda_{q+1}$ and $R_{q+1}$, we may ensure that
$$
\int_{\R^3} g_{q+1}(x)\exp(\langle x \rangle^{q+1})\, dx \leq \Vert g_{q+1}\Vert_{L^1(\R^3)} \exp(\langle R_{q+1} \rangle^{q+1}) < 2^{-q-1} \, .
$$

\subsection{Proof of item \ref{i:ind:ugh}}\label{pf:itemugh}
\eqref{Def:wq} at level $q+1$ follows from Definition \ref{def:gq+1}, as do the non-negativity and compact support of $a_{q+1}$.  Using the definition of $U_{q+1}= U_{\Lambda_{q+1}}$ from Definition~\ref{def:gq+1}, we have the desired properties of $U_{q+1}$ as well.  Next,~\ref{eq:growth:bounds} at level $q+1$ holds by induction for $q' \leq q$ in the first inequality and $q' \leq q-1$ in the second, and at level $q+1$ in the first inequality and $q$ in the second by a wise choice of $R_{q+1}$, $\lambda_{q+1}$, and $\Lambda_{q+1}$, in that order.  Finally, the bounds for $a_{q+1}$ in~\eqref{ind:aq':bounds} follow from Lemma~\ref{lem:aq+1_props}.

\subsection{Proof of item \ref{i:ind:3}}\label{pf:item3}
Recall that
\begin{equation*}
    \begin{split}
        E_{q+1} &= B(f_{q+1},f_{q+1})\\
        &= B(f_q,f_q) + B(f_q,g_{q+1}) + B(g_{q+1},f_q) + B(g_{q+1},g_{q+1})\\
        &= E_q + B(g_{q+1},g_{q+1}) + B(f_q,g_{q+1}) + B(g_{q+1},f_q) \, .
    \end{split}
\end{equation*}
We split the analysis into the linear error $B(f_q, g_{q+1}) + B(g_{q+1}, f_q)$ and the nonlinear error $E_q + B(g_{q+1}, g_{q+1})$, before combining estimates in the third step.
\bigskip

\noindent\texttt{Step 1: Nonlinear error. } We start by analyzing $E_q + B(g_{q+1},g_{q+1})$. Let us write
$$
a_{q+1,\rm{low}} = \mathbb{P}_{\leq \Lambda_{q+1}^{\delta'}}\left(a_{q+1}\right) \quad \text{and} \quad a_{q+1,\rm{high}} = \mathbb{P}_{> \Lambda_{q+1}^{\delta'}}\left(a_{q+1}\right) \, , \qquad \textnormal{so that} \qquad  a_{q+1} = a_{q+1, \rm{low}} + a_{q+1, \rm{high}}
$$
for some $\frac{1}{10}\delta < \delta' < \frac{1}{5}\delta$ which we had fixed in the inductive hypotheses. We now apply Proposition \ref{prop:hhl} with the choices
\begin{align*}
    U_\Lambda = U_{q+1} \, , \qquad \Lambda = \Lambda_{q+1} \, , \qquad \delta \textnormal{ as in~\eqref{def:delta}} \, , \qquad \rho = \delta_0 \, , \qquad a = a_{q+1, \rm{low}} \, , \qquad \lambda = \Lambda_{q+1}^{\delta'} \, , \qquad
    C_{ij} = -\delta_{ij} \, .
\end{align*}
From Proposition~\ref{prop:intermittent_funcs}, $\Lambda$, $\Lambda^\delta$ are natural numbers.  From~\eqref{def:delta}, $\delta^{-1} \in \mathbb{N}$ and $\delta^{-1} \geq 6$. From the choice of $\lambda$ and the definition of $a_{q+1, \rm{low}}$, the condition $\Lambda^{\frac{\delta}{2}} > 2\lambda$ from Proposition~\ref{prop:hhl} is satisfied, and $a_{q+1, \rm{low}}$ has compact frequency support in a ball of radius $\lambda = \Lambda_{q+1}^{\delta'}$.  Then from~\eqref{eq:hhl} and Remark~\ref{rem:reeses}, we have that
$$
B(a_{q+1,\rm{low}}U_{q+1},a_{q+1,\rm{low}}U_{q+1}) = a_{q+1,\rm{low}}^2 + \tilde{E}_{q+1} \, ,
$$
where from~\eqref{e:estimate} and Lemma~\ref{lem:aq+1_props},
\begin{equation}\label{eq:bilin_form_error}
    \Vert \tilde{E}_{q+1} \Vert_X \leq \Vert \tilde{E}_{q+1} \Vert_{H^{-100\delta^{-1}}(\R^3)} \lesssim \Lambda_{q+1}^{\frac{5\delta-1}{2}} \Vert a_{q+1,\rm{low}} \Vert_{W^{4,6/5}(\R^3)}^2 \lesssim \Lambda_{q+1}^{\frac{5\delta-1}{2}} \Vert a_{q+1} \Vert_{W^{4,6/5}}^2 \lesssim  \lambda_{q+1}^2 \Lambda_{q+1}^{\frac{5\delta-1}{2}} \, .
\end{equation}

Now we may write
$$
a_{q+1,\rm{low}}^2 = a_{q+1}^2 - 2a_{q+1}a_{q+1,\rm{high}} + a_{q+1,\rm{high}}^2 \, ,
$$
which implies that
\begin{align}
E_q + B(a_{q+1, \rm{low}} U_{q+1} , a_{q+1, \rm{low}} U_{q+1}) &= E_q + a_{q+1}^2 - 2a_{q+1} a_{q+1, \rm{high}} + a_{q+1,\rm{high}}^2 + \tilde{E}_{q+1} \label{eq:error:growing} \\
&=  E_q + \chi_{q+1}^2 \left( 2 \|E_q\|_{L^\infty(\R^3)} - E_q\right) - 2a_{q+1} a_{q+1, \rm{high}} + a_{q+1,\rm{high}}^2 + \tilde{E}_{q+1} \, . \notag
\end{align}
Estimating the right-hand side and using that $L^1(\R^3)$ embeds in $H^{-100\delta^{-1}}(\R^3)$ and H\"older's inequality to estimate the terms with $a_{q+1, \rm{high}}$ in $L^1(\R^3)$, we obtain
\begin{equation}\label{eq:bilin_form_est}
    \begin{split}
        \Vert E_q + B(a_{q+1,\rm{low}}U_{q+1},a_{q+1,\rm{low}}U_{q+1}) \Vert_X &\leq \Vert (1- \chi_{q+1}^2)E_q\Vert_{L^2(\R^3)} + 2\Vert E_q\Vert_{L^\infty(\R^3)} \Vert \chi_{q+1}^2 \Vert_{\dot{W}^{1,4}(\R^3)} + \Vert \tilde{E}_{q+1} \Vert_{X}\\
        &+ 2\Vert a_{q+1} \Vert_{L^2(\R^3)} \Vert a_{q+1,\rm{high}} \Vert_{L^2(\R^3)} + \Vert a_{q+1,\rm{high}} \Vert_{L^2(\R^3)}^2 \, .
    \end{split}
\end{equation}
We will bound the right-hand side of \eqref{eq:bilin_form_est} term-by-term. First, since $E_q \in L^2(\R^3)$, absolute continuity of the integral gives a choice of $R_{q+1}$ large enough such that
\begin{equation}\label{eq:abs_cont_integral}
    \Vert (1-\chi_{q+1}^2)E_q \Vert_{L^2(\R^3)} < C2^{-q-100} \, .
\end{equation}
For the second term, we compute
$$
\int_{|x|<R_{q+1}} \left| \nabla \left(\chi_{q+1}^2\right)\right|^4 \leq 256 \int_{|x| < R_{q+1}} \left|\nabla \chi_{q+1}\right|^4 \leq 512 R_{q+1}^{-1} \int_{|x| < 2} \left|\nabla \chi \right|^4 \leq 8192 R_{q+1}^{-1} \, ,
$$
and so one can choose $R_{q+1}$ large enough such that
\begin{equation}\label{eq:homogeneous_sob_est}
    2\Vert E_q \Vert_{L^\infty(\R^3)} \Vert \chi_{q+1}^2 \Vert_{\dot{W}^{1,4}(\R^3)} < C2^{-q-100} \, .
\end{equation}
From \eqref{eq:bilin_form_error}, we may choose $\Lambda_{q+1}$ large enough depending on $\lambda_{q+1}$ such that
\begin{equation}\label{eq:bilin_form_error_2}
    \Vert \tilde{E}_{q+1} \Vert_X < C2^{-q-100} \, .
\end{equation}
Now combining Lemma \ref{lem:aq+1_props} and Lemma \ref{lem:good_kernel} we may choose $\Lambda_{q+1}$ large enough, depending on $\lambda_{q+1}$, such that
\begin{equation}\label{eq:aq+1high_est1}
    2\Vert a_{q+1} \Vert_{L^2(\R^3)} \Vert a_{q+1,\rm{high}} \Vert_{L^2(\R^3)} < C2^{-q-100}.
\end{equation}
Similarly we may ensure
\begin{equation}\label{eq:aq+1high_est2}
    \Vert a_{q+1,\rm{high}} \Vert_{L^2(\R^3)}^2 < C2^{-q-100} \, .
\end{equation}
Hence combining \eqref{eq:abs_cont_integral}, \eqref{eq:homogeneous_sob_est}, \eqref{eq:bilin_form_error_2}, \eqref{eq:aq+1high_est1}, and \eqref{eq:aq+1high_est2}, and choosing $R_{q+1}$, then $\lambda_{q+1}$, then $\Lambda_{q+1}$ sufficiently large, we may bound \eqref{eq:bilin_form_est} by
\begin{equation}\label{eq:oscillation_error_cancellation}
    \Vert E_q + B(a_{q+1,\rm{low}}U_{q+1},a_{q+1,\rm{low}}U_{q+1})\Vert_X < C2^{-q-97} \, .
\end{equation}

For the remaining terms in the nonlinear error, we will use that $a_{q+1, \rm{high}} = (\Id - \bp_{\leq \Lambda_{q+1}^{\delta'}})a_{q+1}$, and that $(\Id - \bp_{\leq \Lambda_{q+1}^{\delta'}}) a_{q+1}$ satisfies a good bound due to Lemma~\ref{lem:good_kernel}.   Therefore we appeal to Lemma \ref{lem:proj_est}, Lemma \ref{lem:good_kernel}, Lemma \ref{lem:decoupling}, and Lemma \ref{lem:HLS-Holder} and choose $180\delta^{-1} < N < 190\delta^{-1}$ to see that
\begin{equation*}
    \begin{split}
        \Vert B(a_{q+1,\rm{low}}U_{q+1}, a_{q+1,\rm{high}}U_{q+1}) \Vert_{H^{-10}(\R^3)} &\lesssim \Vert a_{q+1,\rm{high}} U_{q+1} \Vert_{L^1(\R^3)} \Vert a_{q+1,\rm{low}}U_{q+1} \Vert_{W^{5,1}(\R^3)}\\
        &\lesssim \Vert a_{q+1,\rm{high}} \Vert_{W^{4,1}(\R^3)} \Vert U_{q+1} \Vert_{L^1(\T^3)} \Vert a_{q+1,\rm{low}} \Vert_{W^{9,1}(\R^3)} \Vert U_{q+1} \Vert_{W^{5,1}(\T^3)}\\
        &\lesssim \Vert a_{q+1,\rm{high}} \Vert_{W^{4,1}(\R^3)} \Vert a_{q+1,\rm{low}} \Vert_{W^{9,1}(\R^3)} \Vert U_{q+1} \Vert_{W^{5,1}(\T^3)}^2\\
        &\lesssim \Lambda_{q+1}^{-N\delta'} \lambda_{q+1}^2 \Lambda_{q+1}^{10} \Lambda_{q+1}^{3\delta-1}\\
        &\lesssim \lambda_{q+1}^2 \Lambda_{q+1}^{-5} \, .
    \end{split}
\end{equation*}
Then choosing $\Lambda_{q+1}$ large enough depending on $\lambda_{q+1}$, we may ensure that
\begin{equation}\label{eq:B_alow_ahigh}
    \Vert B(a_{q+1,\rm{low}}U_{q+1}, a_{q+1,\rm{high}}U_{q+1}) \Vert_{H^{-10}(\R^3)} < C2^{-q-100} \, .
\end{equation}
Using very similar methods for $0 < \gamma \ll 1$ chosen small enough we have
\begin{equation*}
    \begin{split}
        &\Vert B(a_{q+1,\rm{high}}U_{q+1}, a_{q+1,\rm{low}}U_{q+1}) \Vert_{H^{-10}(\R^3)} \\ &\qquad \lesssim \Vert a_{q+1,\rm{high}}U_{q+1} \Vert_{L^{1+\gamma}(\R^3)} \Vert a_{q+1,\rm{low}}U_{q+1} \Vert_{W^{5,1}(\R^3)}\\
        &\qquad \lesssim \Vert a_{q+1,\rm{high}} \Vert_{W^{4,1+\gamma}(\R^3)} \Vert U_{q+1} \Vert_{L^{1+\gamma}(\T^3)} \Vert a_{q+1,\rm{low}} \Vert_{W^{9,1}(\R^3)} \Vert U_{q+1} \Vert_{W^{5,1}(\T^3)}\\
        &\qquad \lesssim \Vert a_{q+1,\rm{high}} \Vert_{W^{5,1}(\R^3)} \Vert a_{q+1,\rm{low}} \Vert_{W^{9,1}(\R^3)} \Vert U_{q+1} \Vert_{W^{5,1}(\T^3)}^2\\
        &\qquad \lesssim \Lambda_{q+1}^{-N\delta'}\lambda_{q+1}^2 \Lambda_{q+1}^{10} \Lambda_{q+1}^{3\delta-1}\\
        &\qquad \lesssim \lambda_{q+1}^2 \Lambda_{q+1}^{-5} \, ,
    \end{split}
\end{equation*}
and similarly
\begin{equation*}
    \begin{split}
        &\Vert B(a_{q+1,\rm{high}}U_{q+1}, a_{q+1,\rm{high}}U_{q+1}) \Vert_{H^{-10}(\R^3)}  \\
        &\qquad \lesssim \Vert a_{q+1,\rm{high}} U_{q+1} \Vert_{L^1(\R^3)} \Vert a_{q+1,\rm{high}}U_{q+1} \Vert_{W^{5,1}(\R^3)}\\
        &\qquad \lesssim  \Vert a_{q+1,\rm{high}} \Vert_{W^{4,1}(\R^3)} \Vert U_{q+1} \Vert_{L^1(\T^3)} \Vert a_{q+1,\rm{high}} \Vert_{W^{9,1}(\R^3)} \Vert U_{q+1} \Vert_{W^{5,1}(\T^3)}\\
        &\qquad \lesssim \Lambda_{q+1}^{-2N\delta'} \lambda_{q+1}^2 \Lambda_{q+1}^{10} \Lambda_{q+1}^{3\delta-1}\\
        &\qquad \lesssim \lambda_{q+1}^2 \Lambda_{q+1}^{-5} \, .
    \end{split}
\end{equation*}
By choosing $\Lambda_{q+1}$ large enough, we obtain
\begin{equation}\label{eq:B_ahigh_alow}
    \Vert B(a_{q+1,\rm{high}}U_{q+1}, a_{q+1,\rm{low}}U_{q+1}) \Vert_{H^{-10}(\R^3)} < C2^{-q-100}
\end{equation}
and
\begin{equation}\label{eq:B_ahigh_ahigh}
    \Vert B(a_{q+1,\rm{high}}U_{q+1}, a_{q+1,\rm{high}}U_{q+1}) \Vert_{H^{-10}(\R^3)} < C2^{-q-100} \, .
\end{equation}
\bigskip

\noindent\texttt{Step 2: Linear errors. }  This leaves estimating $B(f_q,g_{q+1})$ and $B(g_{q+1},f_q)$. From Lemma \ref{lem:HLS-Holder}, items \ref{i:hls:4} and \ref{i:hls:6}, we have
$$
\Vert B(f_q,g_{q+1}) \Vert_{H^{-10}(\R^3)} \lesssim \Vert g_{q+1} \Vert_{L^1(\R^3)} \Vert f_q \Vert_{W^{5,1}(\R^3)}
$$
and
$$
\Vert B(g_{q+1},f_q) \Vert_{H^{-10}(\R^3)} \lesssim \Vert g_{q+1} \Vert_{L^{1+\gamma}(\R^3)} \Vert f_q \Vert_{W^{5,1}(\R^3)}
$$
for some fixed $0 < \gamma \ll 1/5$.
Since $\Vert f_q \Vert_{W^{5,1}(\R^3)}$ is a finite constant which is independent of our level $q+1$ parameters, upon appealing to Lemma \ref{lem:gq+1_est} and choosing $\Lambda_{q+1}$ large enough depending on $\lambda_{q+1}$ and these level $q$-constants one can enforce
\begin{equation}\label{eq:Nash_est_1}
    \Vert B(f_q,g_{q+1}) \Vert_{H^{-10}(\R^3)} < C2^{-q-100}
\end{equation}
and
\begin{equation}\label{eq:Nash_est_2}
    \Vert B(g_{q+1},f_q) \Vert_{H^{-10}(\R^3)} < C2^{-q-100} \, .
\end{equation}
\bigskip

\noindent\texttt{Step 3: Combining bounds. }  Combining \eqref{eq:oscillation_error_cancellation}, \eqref{eq:B_alow_ahigh}, \eqref{eq:B_ahigh_alow}, \eqref{eq:B_ahigh_ahigh}, \eqref{eq:Nash_est_1}, and \eqref{eq:Nash_est_2} we have
$$
\Vert E_{q+1} \Vert_X < C2^{-q-97} + 5C2^{-q-100} < C2^{-q-90} < C2^{-q-11} \, ,
$$
which completes the proof.

\section{Rigidity}
We begin by quickly recalling the formal conservation laws of the Krieger--Strain equation which have previously been observed in the literature; see e.g.~\cite{GZSurvey}.
\begin{lemma}[\textbf{Formal conservation laws and H-theorem}]\label{lem:conservation_laws}
Suppose $f \colon [0,T]\times \R^3 \to \R^+$ is a smooth, rapidly decaying solution to the Krieger--Strain equation. Then, conservation of mass and momentum hold: 
\begin{equation*}
    \frac{d}{dt} \int_{\R^3} f(v) \, dv = \frac{d}{dt}\int_{\R^3} v_i f(v) \, dv = 0 \, .
\end{equation*}
Additionally, kinetic energy increases: 
\begin{equation*}
    \frac{d}{dt} \int_{\R^3} \abs{v}^2 f \, dv =  \frac{1}{\pi} \int_{\R^3}\int_{\R^3} \frac{f(v)f(w)}{\abs{v-w}} \, dw \, dv = 4 \norm{f}_{\dot{H}^{-1}}^2 \, .
\end{equation*}
Finally, the $H$-theorem takes the form
\begin{equation*}
    \frac{d}{dt} \int_{\R^3} f\log f \, dv = - \frac{1}{2\pi}\int_{\R^3}\int_{\R^3} \frac{\abs{\nabla_{v-w} \sqrt{f(v)f(w)}}^2}{\abs{v-w}} \, dv \, dw \le 0 \, .
\end{equation*}

\end{lemma}

\begin{proof}
Writing the Krieger--Strain equation in collisional form,
\begin{equation*}
    \partial_t f = \nabla_v \cdot \int_{\R^3} \frac{\nabla_{v-w}(f(v)f(w))}{4\pi \abs{v-w}} \, dw \, .
\end{equation*}
Multiplying by a test function $\varphi$ and integrating by parts in the collision operator, we obtain:
\begin{equation*}
    \int_{\R^3} \varphi \partial_t f \, dv = -\int_{\R^3}\int_{\R^3}  \nabla_v \varphi(v) \cdot \frac{(\nabla_v - \nabla_w)(f(v)f(w))}{4\pi \abs{v-w}} \, dw \, dv \, .
\end{equation*}
Symmetrizing in $v$ and $w$, we find
\begin{equation}\label{eq:symmetrized_weak_form}
    \int_{\R^3} \varphi \partial_t f \, dv = -\frac{1}{2}\int_{\R^3}\int_{\R^3}  \left[\nabla_v \varphi(v) - \nabla_w\varphi(w) \right] \cdot \frac{(\nabla_v - \nabla_w)(f(v)f(w))}{4\pi \abs{v-w}} \, dw \, dv \, .
\end{equation}
Taking the choices $\varphi(v) = 1$ and $\varphi(v) = v$, we deduce conservation of mass and momentum immediately. Taking $\varphi(v) = \abs{v}^2$, we find a first formula for the kinetic energy growth:
\begin{equation*}
    \frac{d}{dt} \int_{\R^3} \abs{v}^2 f \, dv = -\int_{\R^3}\int_{\R^3}  \frac{v-w}{4\pi\abs{v-w}} \cdot (\nabla_v - \nabla_w)(f(v)f(w)) \, dw \, dv \, .
\end{equation*}
Notice that the integrand depends entirely on the direction $v-w$. Thus, integrating by parts, we find a sign:
\begin{equation*}
\begin{aligned}
    \frac{d}{dt} \int_{\R^3} \abs{v}^2 f \,  dv &= \frac{1}{4\pi}\int_{\R^3}\int_{\R^3}  \left(\nabla_v -\nabla_w\right)\cdot \left(\frac{v-w}{\abs{v-w}}\right) f(v)f(w) \, dw \, dv\\
        &= \frac{1}{2\pi}\int_{\R^3}\int_{\R^3}  \left(\nabla \cdot \frac{z}{\abs{z}}\right)\bigg|_{z=v-w} f(v)f(w) \, dw \, dv\\
        &= \frac{1}{\pi}\int_{\R^3}\int_{\R^3}  \frac{f(v)f(w)}{\abs{v-w}} \, dw \, dv \, .
\end{aligned}
\end{equation*}
Lastly, the $H$-theorem follows from \eqref{eq:symmetrized_weak_form} as well. By conservation of mass, we have
\begin{equation*}
    \frac{d }{dt}\int_{\R^3} f \log(f) \, dv = \int_{\R^3} \log(f)\partial_t f \, dv + \int_{\R^3} \partial_t f \, dv = \int_{\R^3} \log(f)\partial_t f \, dv \, .
\end{equation*}
Taking $\varphi = \log(f)$ in \eqref{eq:symmetrized_weak_form}, we find
\begin{equation*}
\begin{aligned}
    \frac{d}{dt}\int_{\R^3} f \log(f) \, dv &=  -\frac{1}{8\pi}\int_{\R^3}\int_{\R^3}  \left[\nabla_v \log(f) - \nabla_w\log (f) \right] \cdot \frac{(\nabla_v - \nabla_w)(f(v)f(w))}{\abs{v-w}} \, dw \, dv\\
        &= -\frac{1}{8\pi}\int_{\R^3}\int_{\R^3}  \left[\frac{\nabla_v f}{f(v)} - \frac{\nabla_w f}{f(w)}\right] \cdot \frac{f(w)\nabla_v f - f(v)\nabla_w f}{\abs{v-w}} \, dw \, dv\\
        &= -\frac{1}{8\pi}\int_{\R^3}\int_{\R^3} \frac{f(v)f(w)}{\abs{v-w}} \abs{\frac{\nabla_v f}{f(v)} - \frac{\nabla_w f}{f(w)}}^2 \, dw \, dv \, .
\end{aligned}
\end{equation*}
This concludes the proof of the claimed formal identities.
\end{proof}

The formal proof of Theorem \ref{prop:rigidity} is now clear. For a stationary solution $f$, each time derivative in Lemma \ref{lem:conservation_laws} vanishes. In particular, the kinetic energy production, computed as the $\dot{H}^{-1}$ seminorm vanishes. Consequently, $f$ is constant and from a global integrability constraint we conclude $f \equiv 0$. The remaining work is to extend this formal argument to the integrability class $L^{\sfrac65}(\R^3)$ specified in Theorem \ref{prop:rigidity}.

\begin{proof}[Proof of Theorem~\ref{prop:rigidity}]
    Suppose $f:\R^3\to \R^+$ belongs to $L^{\sfrac 65}(\R^3)$ and $f$ is a stationary weak solution to Krieger--Strain $Q_{\rm KS}(f) = 0$ in the sense of Definition \ref{def:weak_soln}. More precisely, this means that for any $\varphi\in C^\infty_c(\R^3)$,
    \begin{equation}
        \int_{\R^3}\int_{\R^3} \nabla_{v-w} \cdot \left(\frac{\nabla\varphi(v) - \nabla\varphi(w)}{8\pi\abs{v-w}}\right) f(v)f(w) \, dv\, dw = 0 \, . \label{eq:weak:1}
    \end{equation}
    Expanding~\eqref{eq:weak:1} yields
    \begin{equation}\label{eq:truncated_equality}
        \int_{\R^3}\int_{\R^3} \left(\frac{ {\Delta_v \varphi  + \Delta_w\varphi}}{8\pi\abs{v-w}} {-} \frac{(v-w)\cdot (\nabla_v\varphi-\nabla_w\varphi) }{4\pi\abs{v-w}^3}\right) f(v)f(w) \, dv\, dw = 0 \, .
    \end{equation}
    Now, to use the kinetic energy production, take $\varphi(v) = \chi_R(v)\abs{v}^2$, where $\chi_R$ is the smooth cutoff introduced in Definition \ref{def:projs}. Then, $\chi_R \to 1$ pointwise, while all derivatives of $\chi_R$ converge pointwise to $0$ as $R\to \infty$ and we compute the pointwise limit of the integrand as 
    \begin{equation*}
    \begin{aligned}
        \frac{\Delta_v \varphi  + \Delta_w\varphi}{8\pi\abs{v-w}} - \frac{(v-w)\cdot \left(\nabla_v\varphi-\nabla_w\varphi\right)}{4\pi\abs{v-w}^3} &\longrightarrow \frac{\Delta_v \abs{v}^2  + \Delta_w \abs{w}^2}{8\pi\abs{v-w}} - \frac{(v-w)\cdot \left(\nabla_v\abs{v}^2-\nabla_w\abs{w}^2\right)}{4\pi\abs{v-w}^3}\\
            &\quad = \frac{12}{8\pi\abs{v-w}} - \frac{(v-w)\cdot (2v-2w)}{4\pi\abs{v-w}^3}\\
            &\quad =\frac{1}{\pi\abs{v-w}} \, .
    \end{aligned}
    \end{equation*}
    Thus, assuming for the moment that we may pass the limit $R\to \infty$ under the integral in \eqref{eq:truncated_equality}, we obtain
    \begin{equation}\label{eq:truncated_equality_limit}
        \int_{\R^3}\int_{\R^3} \frac{f(v)f(w)}{\pi\abs{v-w}} \, dv\, dw = 0 \, .
    \end{equation}
    The integrand is nonnegative and we conclude $f \equiv 0$ almost everywhere. Finally, \eqref{eq:truncated_equality_limit} is justified by the Lebesgue dominated convergence theorem. We compute
    \begin{equation*}
    \begin{aligned}
        &\nabla \varphi(v) = \abs{v}^2\nabla \chi_R(v)  +  2v\chi_R(v) \qquad &\text{so that} \qquad \abs{\nabla\varphi} \le CR\\
        &\nabla^2 \varphi(v) = \abs{v}^2\nabla^2 \chi_R(v)  +  4(v \otimes \nabla \chi_R(v))^{\rm sym} + 2 \chi_R(v) \Id \qquad &\text{so that} \qquad \abs{\nabla^2\varphi} \le C \, .
    \end{aligned}
    \end{equation*}
    These simple scaling consistent estimates yield the uniform-in-$R$ bound
    \begin{equation*}
        \abs{\frac{\Delta_v \varphi  + \Delta_w\varphi}{8\pi\abs{v-w}} - \frac{(v-w)\cdot \left(\nabla_v\varphi-\nabla_w\varphi\right)}{4\pi\abs{v-w}^3}}f(v)f(w) \le \frac{C\norm{\nabla^2\varphi}_{L^{\infty}}f(v)f(w)}{\abs{v-w}} \le \frac{Cf(v)f(w)}{\abs{v-w}} \, .
    \end{equation*}
    The identity \eqref{eq:truncated_equality_limit} is then justified and the proof is complete provided the right-hand side is in $L^1(\R^6)$. This integrability is furnished by our assumption on $f$, H\"older's inequality and the Hardy--Littlewood--Sobolev inequality:
    \begin{equation*}
        \int_{\R^6}\frac{f(v)f(w)}{\abs{v-w}} \, dv \, dw \lesssim \norm{f}_{L^{\sfrac65}(\R^3)}\norm{(-\Delta)^{-1}_{\R^3} f}_{L^6(\R^3)} \lesssim \norm{f}_{L^{\sfrac65}(\R^3)}^2 < \infty \, .
    \end{equation*}
\end{proof}

\appendix

\section{The renormalized power--law solution to \texorpdfstring{$\Delta u + u^2 = 0$}{dfu}}\label{appendix:power_law_elliptic}

In this appendix, we clarify different notions under which $u_{\rm sing}(v): = \frac{-2}{|v|^2}$ does, and does not, solve the equation
$$ \Delta u + u^2 = 0 \, . $$
In a nutshell, $u_{\rm sing}$ solves the equation under renormalization\footnote{Since this is only a toy example, we perform the computation completely only for one choice of renormalization, inspired by~\cite{DPL}.  We then explain how to adjust the argument to other types of renormalizations, such as those found in~\cite{Villani96}}, but it is \textit{not} true that
$$ \Delta u_\epsilon^\rho + \left(u_\epsilon^\rho\right)^2 \longrightarrow 0 \qquad \textnormal{as} \quad \epsilon\longrightarrow 0 $$
in the sense of distributions for all (in fact, not for any) mollifications $u_\epsilon^\rho = u\ast \rho_\epsilon$.  The latter claim is relatively easy to check; let $\rho$ be unit mass, smooth, and compactly supported.  First, we note that $\Delta u_\epsilon^\rho \rightarrow \Delta u$ in the sense of distributions.  On the other hand, choosing any smooth, compactly supported, nonnegative $\varphi$ with $\varphi(0)>0$ and using Fatou's lemma yields 
$$ \liminf_{\epsilon\rightarrow 0} \int_{\R^3} \left(u^\rho_\epsilon\right)^2 (v) \varphi(v) \, dv  \geq \int_{\R^3} \liminf_{\epsilon\rightarrow 0} \left(u^\rho_\epsilon\right)^2(v) \varphi(v) \, dv = \int_{\R^3} \frac{4}{|v|^4} \varphi(v) \, dv = \infty \, .  $$
Therefore $\left(u^\rho_\epsilon\right)^2$ does not converge as $\epsilon\rightarrow 0$, and neither does $\Delta u_\epsilon^\rho + \left(u^\rho_\epsilon\right)^2$.

We will now show that $u_{\rm sing}$ satisfies the equation in a particular renormalized sense.  Towards this end, let $\beta:\R \to \R$ be a smooth function, and suppose $u$ were a smooth solution.  Then $\beta(u)$ would solve
$$
\Delta(\beta(u)) = \beta''(u)|\nabla u|^2 + \beta'(u) \Delta u = \beta''(u) |\nabla u|^2 - \beta'(u) u^2 \, .
$$
We aim to choose $\beta$ such that all terms above can be interpreted as distributions for $u = u_{\rm sing}$, and then check if
\begin{equation}\label{eq:renorm_form}
    \int_{\R^3} \left( \beta(u_{\rm sing}) \Delta \phi - \beta''(u_{\rm sing}) |\nabla u_{\rm sing}|^2 \phi + \beta'(u_{\rm sing}) u_{\rm sing}^2 \phi \right) = 0 
\end{equation}
for all test functions $\phi \in C_c^\infty(\R^3)$. Let us now make the choice $\beta(x) =- \log(1-x)$, so that $\beta'(x) = \frac{1}{1-x}$ and $\beta''(x) = \frac{1}{(1-x)^2}$. Then using that 
$$ \nabla u_{\rm sing}(v) = \frac{4v}{|v|^4} \, , $$
we find that
\begin{subequations}
\begin{align}
    \beta(u_{\rm sing}) \Delta \phi &= -\log\left(1 + \frac{2}{|v|^2}\right)\Delta \phi(v)\in L^1(\R^3) \, \\
    \beta''(u_{\rm sing}) |\nabla u_{\rm sing}|^2 &= \frac{16}{|v|^2(|v|^2+2)^2} \in L^1(\R^3) \, \\
    \beta'(u_{\rm sing})u_{\rm sing}^2 &= \frac{4}{|v|^2(|v|^2+2)} \in L^1(\R^3) \, .
\end{align}\label{sub:bulks}
\end{subequations}
Thus for $\phi \in C_c^\infty(\R^3)$ the integral in \eqref{eq:renorm_form} is well defined for $u = u_{\rm sing}$.

Now fix $\phi \in C^\infty_c(\R^3)$ and $r_0 > 0$.  On $\{|v| > r_0\}$ all functions
are smooth, so integrating by parts twice,
\begin{equation}\label{eq:excision}
    \int_{|v| > r_0} \beta(u_{\rm sing}) \Delta\phi \, dv
    = \int_{|v|>r_0} \Delta\big(\beta(u_{\rm sing})\big) \phi \, dv
    - \oint_{|v| = r_0} \Big( \beta(u_{\rm sing}) \pa_r \phi
    - \phi \, \pa_r \beta(u_{\rm sing}) \Big) dS \, ,
\end{equation}
where the outward normal derivative for the domain
$\{|v|>r_0\}$ is $-\pa_r$ in the radial variable.  We estimate the two
boundary terms.  First, since $\phi \in C^\infty_c(\R^3)$,
\begin{equation*}
    \left| \oint_{r=r_0} \beta(u_{\rm sing}) \pa_r \phi \, dS \right|
    \les r_0^2 \log\left(1 + \frac{2}{r_0^2}\right)
    \norm{\nabla \phi}_{L^\infty(\R^3)} \xrightarrow[r_0 \to 0^+]{} 0 \, .
\end{equation*}
Second, using $\pa_r \beta(u_{\rm sing}) = \beta'(u_{\rm sing}) \pa_r u_{\rm sing}$
and $\pa_r u_{\rm sing}(r_0) = 4r_0^{-3}$,
\begin{equation}\label{eq:flux_A}
    \left| \oint_{r=r_0} \phi \, \pa_r \beta(u_{\rm sing}) \, dS \right|
    \les r_0^2 \cdot \frac{r_0^2}{r_0^2+2} \cdot \frac{4}{r_0^3}
    \norm{\phi}_{L^\infty(\R^3)}
    \les \frac{r_0}{r_0^2 + 2} \norm{\phi}_{L^\infty(\R^3)}
    \xrightarrow[r_0 \to 0^+]{} 0 \, .
\end{equation}
On $\{|v| > r_0\}$ the chain rule and the equation $\Delta u_{\rm sing} = - u_{\rm sing}^2$ give
\begin{equation*}
\Delta\big(\beta(u_{\rm sing})\big)
    = \beta''(u_{\rm sing}) \abs{\nabla u_{\rm sing}}^2
    + \beta'(u_{\rm sing}) \Delta u_{\rm sing}
    = \beta''(u_{\rm sing}) \abs{\nabla u_{\rm sing}}^2
    - \beta'(u_{\rm sing}) u_{\rm sing}^2 \, .
\end{equation*}
Substituting this into~\eqref{eq:excision}, letting $r_0 \to 0^+$, and applying
the dominated convergence theorem gives~\eqref{eq:renorm_form}.

\begin{remark}[\textbf{Other renormalizations}]
The most singular term in~\eqref{eq:excision} is $\pa_r \beta(u_{\rm sing})$.  On a sphere of radius $r_0$, the bound for this term is 
$$ r_0^2 \| \phi \|_{L^\infty} \beta'(u_{\rm sing}(r_0)) |\nabla u_{\rm sing}(r_0)| \approx \beta'(-r_0^{-2}) r_0^{-1} \, . $$
Therefore the above computation can only work if $\beta'(-s) s^{\sfrac 12} \rightarrow 0$ as $s \rightarrow \infty$. Any $\beta$ with $\beta'(-s) \les (1+s)^{-1}$ will satisfy this criteria.  For example if $\beta_\delta(s) = \frac{s}{1-\delta s}$ with $\delta>0$ (inspired by~\cite{Villani96} and adjusted to reflect the negativity of the profile), then $\beta_\delta'(s) = \frac{1}{1-\delta s} + \frac{\delta s}{(1-\delta s)^2} = \frac{1}{(1-\delta s)^2}$ satisfies this criteria.  Similarly, one can check that the bulk terms are well--defined for this choice of $\beta_\delta$, and so $u_{\rm sing}$ is a renormalized solution if $\beta_\delta(s) = s(1-\delta s)^{-1}$ for any $\delta>0$.  At the heart of the matter is that renormalization allows for cancellation in the $\delta \rightarrow 0$ limit between the terms 
$$  \left[- \beta_\delta''(u_{\rm sing}) | \nabla u_{\rm sing}|^2 + \beta_\delta'(u_{\rm sing}) u_{\rm sing}^2 \right] \phi  $$
from~\eqref{eq:renorm_form}.  On the other hand, mollifier-confluent solutions do not similarly mask the singularity at $v=0$.
\end{remark}

\section{The power--law solution to Krieger--Strain}\label{appendix:power_law}

In this appendix, we derive the family of power--law steady states $f_{\rm sing}(v) = C\abs{v}^{-\sfrac32}$ to Krieger--Strain and clarify the sense in which they solve the equation. 

\subsection{Formal derivation}

We begin with a derivation of these solutions. To this end, we recall the non-divergence form of the stationary Krieger--Strain equation:
\begin{equation*}
    (-\Delta)^{-1}_{\R^3}f\Delta f + f^2 = 0.
\end{equation*}
One significant issue in finding solutions in the above form is the nonlocal term $(-\Delta)^{-1}_{\R^3}f$. To this end, we relax the above nonlocal equation to a \textit{local system} of elliptic equations by introducing an auxiliary quantity $a$:
\begin{equation}\label{eq:local_system}
    a \Delta f + f^2 = 0 \qquad \text{and} \qquad -\Delta a = f.
\end{equation}
Of course, \eqref{eq:local_system} is manifestly equivalent to the original stationary Krieger--Strain equation under decay conditions that imply $a = (-\Delta)^{-1}_{\R^3} f$. The advantage of \eqref{eq:local_system} is that it is entirely local and no decay requirements are necessary to define a notion of weak solution. Moreover, formally, we can even eliminate $f$ leading to a local, albeit nonlinear, equation for $a$:
\begin{equation}\label{eq:a_pde}
    -a\Delta^2 a + (\Delta a)^2 = 0.
\end{equation}
Here $\Delta^2 a$ is the bi-Laplacian. Finding nontrivial solutions to the nonlinear equation \eqref{eq:a_pde} is a substantially simpler problem. We follow a standard approach here.

The equation is manifestly rotation invariant, and making a radial ansatz yields 
\begin{equation*}
    a\left(\partial_{rr} + \frac{2}{r} \partial_r\right)^2 a = \left(\partial_{rr} a + \frac{2}{r} \partial_r a\right)^2.
\end{equation*}
Making the substitution $b(r) = r a(r)$ so that $b'(r) = a(r) + ra'(r)$, $b''(r) = 2a'(r) + ra''(r)$, and applying the identity $\frac{b''}{r} = a'' + \frac{2}{r}a'$ leads to an autonomous ODE:
\begin{equation*}
    b \frac{d^4}{d r^4} b = \left[\frac{d^2}{d r^2} b \right]^2.
\end{equation*}
The above ODE clearly admits a two--parameter family of solutions $b = C_1 + C_2 r$.
However, it additionally admits another power law solution. Making the ansatz $b = r^k$ for $k \neq 0$, one finds
\begin{equation*}
    k(k-1)(k-2)(k-3) = k^2(k-1)^2 \qquad \text{or equivalently} \qquad k = 0,\,1,\,\text{or }\sfrac32.
\end{equation*}
Since the ODE is autonomous and homogeneous, this yields another two-parameter family of solutions:
\begin{equation*}
    b(r) = C (r - r_0)^{3/2} \text{ for any }r > r_0, \qquad \text{with parameters }C \in\R \quad\text{and} \quad r_0  \geq 0 \, .
\end{equation*}
The specific solution of interest occurs with $r_0 = 0$, where unraveling the introduced auxiliary quantities leads to
\begin{equation*}
    f_{\rm sing}(v) \coloneqq C\abs{v}^{-\sfrac32} \qquad \text{and} \qquad a_{\rm sing}(v) \coloneqq -\frac{4C}{3}\abs{v}^{\sfrac12}.
\end{equation*}
Lastly, we note for the record that the autonomous ODE---and consequently the local system \eqref{eq:local_system}---has many more solutions, including two (modified) spherical Bessel functions: for any $C,\, K \in \R$, the pairs
\begin{equation*}
    (f,a) = \left(CK^2\frac{\sin(K\abs{v})}{\abs{v}}, C\frac{\sin(K\abs{v})}{\abs{v}}\right) \qquad \text{and} \qquad (f,a) = \left(CK^2 \frac{\sinh(K\abs{v})}{\abs{v}},-C\frac{\sinh(K\abs{v})}{\abs{v}}\right)
\end{equation*}
solve the local system \eqref{eq:local_system}.

\subsection{Notion of solution}

The derivation in the preceding subsection verified the following conclusion:
\begin{lemma}\label{lem:solution_sense}
For any $C > 0$, define the pair of functions 
\begin{equation*}
    f_{\rm sing}(v) \coloneqq C\abs{v}^{-\sfrac32} \qquad \text{and} \qquad a_{\rm sing}(v) \coloneqq -\frac{4C}{3}\abs{v}^{\sfrac12}.
\end{equation*}
Then, $f_{\rm sing},\, a_{\rm sing} \in C^\infty(\R^3 \setminus\{0\})$ and $(f_{\rm sing},a_{\rm sing})$ satisfy the local elliptic system \eqref{eq:local_system} pointwise on $\R^3 \setminus \{0\}$.
\end{lemma}
\noindent Note that both $f_{\rm sing}$ and $\Delta a_{\rm sing}$ belong to $L^1_{\rm loc}(\R^3)$ and the Poisson equation relating them holds in the sense of distributions. However, neither of the terms appearing in the other equation---namely $f_{\rm sing}^2 = C^2\abs{v}^{-3}$ and $a_{\rm sing} \Delta f_{\rm sing} = -C^2\abs{v}^{-3}$---belong to $L^1_{\rm loc}(\R^3)$; the cancellation occurs only pointwise, not in the sense of distributions.

Additionally, this singular ``solution'' $f_{\rm sing}$ does not decay sufficiently fast at infinity to make sense of any of the original nonlocal formulations of $Q_{\rm KS}(f)$ in \eqref{defn:collision_operator_smooth}; indeed, $a_{\rm sing}$ \textit{cannot} be given by convolution with the usual Newtonian potential---even the sign is wrong for this to be possible.\footnote{The sign switch is due to the following issue regarding inversion of the Laplacian.  Define $f_R(v) = |v|^{-\sfrac 32} \mathbf{1}_{|v|<R}$ for $R>0$. Then one can invert the Newton potential on this object, and direct computation gives $a_R(v) = 2 R^{\sfrac 12} - \sfrac 43 |v|^{\sfrac 12}$ for $|v|\leq R$ and $a_R(v) =  \sfrac 23 R^{\sfrac 32}|v|^{-1}$ for $|v| \geq R$, so that $a_R(v)$ is nonnegative for every finite $R>0$.  But by attempting to define a ``homogeneous'' inverse Laplacian as $R \rightarrow \infty$, we implicitly renormalized by subtracting the diverging constants $2R^{\sfrac 12}$, and the remainder is negative.}

\subsection{Mollifier confluence}

Notice that after mollification, $f_\eps = f_{\rm sing} \ast \eta_\eps \in C^\infty(\R^3)$. However, due to the slow decay at infinity, we still only have $f_\eps \in L^p(\R^3)$ for $2 < p \le \infty$, which is insufficient to make sense of the nonlocality in the Krieger--Strain operator. Consequently, $f_{\rm sing}$ already lies outside the purview of mollifier-confluent steady solutions as defined in Definition~\ref{def:moll:conf}. However, the following lemma detects a more severe issue; namely, that there is a signed singularity which appears in the mollifier-confluent version of the system~\eqref{eq:local_system} for $f_{\rm sing}$.
\begin{lemma}
    Fix a singular solution pair $(f_{\rm sing},a_{\rm sing})$ as in Lemma \ref{lem:solution_sense} for some $C>0$ arbitrary. Let $\eta \in C^\infty_c(\R^3)$ be any standard mollifier and set
    \begin{equation*}
        f_\eps \coloneqq f_{\rm sing} \ast \eta_{\eps} \qquad \text{and similarly} \qquad a_{\eps} \coloneqq a_{\rm sing} \ast \eta_{\eps}.
    \end{equation*}
    Then, the following limit holds:
    \begin{equation*}
       \lim_{\eps \to 0^+} \left[ a_\eps \Delta f_\eps + f_\eps^2 \right] = \left(\frac{32\pi C^2}{3}\right)\delta_0 \qquad \text{in the sense of distributions.}
    \end{equation*}
\end{lemma}

\begin{proof}
Without loss of generality, set $C=1$, as the case $C\neq 1$ is recovered by homogeneity. Away from $v = 0$, all functions are smooth and the desired convergence holds. The only issue is evaluating the limit near $v = 0$.

    To this end, we work in divergence form. Note that by standard properties of mollifiers, $-\Delta a_\eps = f_\eps$. Therefore, working pointwise, we find
    \begin{equation*}
        \nabla \cdot \left(a_\eps \nabla f_\eps - \nabla a_\eps f_\eps\right) = a_\eps \Delta f_\eps + f_\eps^2,
    \end{equation*}
    and the divergence and nondivergence forms agree for $\eps > 0$. Computing the flux at the limit $\eps = 0$, we find
    \begin{equation*}
        a_{\rm sing} \nabla f_{\rm sing} - \nabla a_{\rm sing} f_{\rm sing} = 2\frac{v}{\abs{v}^3} + \frac{2}{3}\frac{v}{\abs{v}^3} = \frac{8v}{3\abs{v}^3}
    \end{equation*}
    Additionally, notice that $a_{\rm sing} \in L^\infty_{\rm loc}(\R^3)$, $\nabla a_{\rm sing} \in L^3_{\rm loc}(\R^3)$, $\nabla f_{\rm sing} \in L^1_{\rm loc}(\R^3)$, and $f_{\rm sing} \in L^{\sfrac32}_{\rm loc}(\R^3)$. Consequently, by standard properties of mollifiers,
    \begin{equation*}
        a_\eps \nabla f_\eps - \nabla a_\eps f_\eps \xrightarrow{\eps \to 0^+} a_{\rm sing} \nabla f_{\rm sing} - \nabla a_{\rm sing} f_{\rm sing} \qquad \text{in }L^1_{\rm loc}(\R^3).
    \end{equation*}
    Thus, for any test function $\varphi \in C^\infty_c(\R^3)$, we conclude
    \begin{equation*}
    \begin{aligned}
        \lim_{\eps \to 0^+}\int_{\R^3} \varphi \left[a_\eps \Delta f_\eps + f_\eps^2\right] \dd v &=  -\int_{\R^3} \nabla \varphi \cdot \left(a_{\rm sing} \nabla f_{\rm sing} - \nabla a_{\rm sing} f_{\rm sing}\right) \dd v\\
            &= - \frac{8}{3} \int_{\R^3} \nabla \varphi \cdot \left( \frac{v}{\abs{v}^3}\right) \dd v = \frac{32\pi}{3} \varphi(0).
    \end{aligned}
    \end{equation*}
    The last equality follows from the classical computation $\nabla \cdot \frac{v}{\abs{v}^3} = 4\pi \delta_0$ in the sense of distributions and completes the proof.
\end{proof}

\medskip

\noindent\textsc{Department of Mathematics, Purdue University, West Lafayette, IN, USA.}
\vspace{.03in}
\newline\noindent\textit{Email address}: \href{mailto:ngismond@purdue.edu}{ngismond@purdue.edu}.
\smallskip

\noindent\textsc{Department of Mathematics, University of Chicago, Chicago, Illinois.}
\vspace{.03in}
\newline\noindent\textit{Email address}: \href{mailto:wgolding@uchicago.edu}{wgolding@uchicago.edu}.
\smallskip

\noindent\textsc{Department of Mathematics, University of Notre Dame, Notre Dame, IN, USA.}
\vspace{.03in}
\newline\noindent\textit{Email address}: \href{mailto:mnovack@nd.edu}{mnovack@nd.edu}.

\end{document}